\documentclass[11pt]{scrartcl}
		\usepackage[utf8]{inputenc}
		\usepackage[T1]{fontenc}
		\usepackage[UKenglish]{babel}
		\usepackage{amsmath, amsfonts, amsthm, amssymb, mathtools, stmaryrd}

		\usepackage{thmtools, thm-restate}

		\usepackage{xcolor}
        \definecolor{darkblue}{rgb}{0.0, 0.0, 0.5}\usepackage[pdfusetitle,colorlinks,linkcolor=darkblue,citecolor=darkblue,urlcolor=darkblue]{hyperref}

        \usepackage[nameinlink,capitalize,noabbrev]{cleveref}
		\usepackage{csquotes}
\usepackage[osf,sc]{mathpazo}
		\addtokomafont{disposition}{\normalfont\bfseries}
		\usepackage{relsize}
		\usepackage{microtype}
		\usepackage{multicol}
		\usepackage{pdfpages}

		\recalctypearea

		\usepackage[a4paper,margin=1in]{geometry}
		
		\usepackage{graphicx}
		
		\usepackage{abstract}
		\usepackage[inline]{enumitem}
		\usepackage{booktabs}
		\usepackage{todonotes}

		\usepackage[style=numeric-comp,maxbibnames=99,maxcitenames=99,backend=biber,bibencoding=utf8,backref=true,sorting=none]{biblatex} 
		\AtEveryBibitem{
        	\clearfield{urldate}
        	\clearfield{urlyear}
        	\clearfield{urlmonth}
        	\clearfield{publisher}
        	\clearfield{issn}
        	\clearfield{isbn}
        	\clearfield{editor}
        	\clearfield{editors}
        	\clearlist{editors}
        	\clearlist{editor}
        	\clearfield{month}
        	\clearfield{day}
        	\clearlist{language}
        	\clearlist{language}
        	\clearlist{Language}
        	\clearlist{langid}
        }
\DeclareSourcemap{
        	\maps[datatype=bibtex]{
        		\map[overwrite]{
        			\step[fieldsource=doi, final]
        			\step[fieldset=url, null]
        			\step[fieldset=eprint, null]
        			\step[fieldset=isbn, null]
        		}
        	}
        }
\DeclareDelimFormat{finalnamedelim}{\ifnumgreater{\value{liststop}}{2}{\finalandcomma}{}\addspace\&\space}
        
        \DefineBibliographyExtras{UKenglish}{\def\finalandcomma{\addcomma}}

		\usepackage{tikz}
		\usetikzlibrary{cd}
		\usepackage{microtype}
		
		\theoremstyle{definition}
		\newtheorem{definition}{Definition}[section]
			
		\newtheorem{remark}[definition]{Remark}
		
		\theoremstyle{plain}
		
		\newtheorem{example}[definition]{Example}
		\newtheorem{lemma}[definition]{Lemma}
		\newtheorem{corollary}[definition]{Corollary}
		\newtheorem{theorem}[definition]{Theorem}
		
		\newtheorem{question}[definition]{Question}
		\newtheorem{proposition}[definition]{Proposition}

		\newtheorem{fact}[definition]{Fact}
		\Crefname{fact}{Fact}{Facts}
		
		\theoremstyle{remark}
		\newtheorem{claim}[definition]{Claim}
		\newtheorem{plainexample}[definition]{Example}
		\Crefname{claim}{Claim}{Claims}
		\newenvironment{claimproof}[1][Proof of Claim]{\begin{proof}[#1] }{ \end{proof}}
		\setlist[enumerate, 1]{font=\upshape, noitemsep, nolistsep}
		\setlist[enumerate, 2]{font=\upshape, noitemsep, nolistsep}
		\setlist[itemize, 1]{noitemsep, nolistsep,font=\upshape}

		\newlist{assertions}{enumerate}{4}
		\setlist[assertions]{font=\upshape, noitemsep, nolistsep}
		\setlist[assertions, 1]{label=(\roman*)}
\Crefname{assertionsi}{Assertion}{Assertions}

		\renewcommand\phi\varphi
		\renewcommand\epsilon\varepsilon
		
		\DeclareMathOperator\Hom{Hom}
		
		\title{Topological and Geometric Perspectives on Homomorphism Indistinguishability}
		\usepackage{authblk}

        \author[1]{Josse van Dobben de Bruyn}
        \author[2]{J\'er\'emie Marqu\`es}
        \author[3]{David E.\ Roberson}
        \author[4]{Tim Seppelt}
        \author[5]{Gian Luca Spitzer}
        \author[1]{Peter Zeman}
        \affil[1]{Charles University, Czech Republic}
        \affil[2]{Université Côte d'Azur, France}
        \affil[3]{Technical University of Denmark}
		\affil[4]{IT-Universitetet i København, Denmark}
        \affil[5]{Université de Bordeaux, France}

		\DeclareMathOperator{\cl}{cl}
		\DeclareMathOperator{\rk}{rk}
		\DeclareMathOperator{\supp}{supp}
		\DeclareMathOperator{\aut}{aut}
		\DeclareMathOperator{\surj}{surj}
		\DeclareMathOperator{\indsub}{indsub}
		\DeclareMathOperator{\sub}{sub}
        \DeclareMathOperator{\emb}{emb}
		\DeclareMathOperator{\spsub}{spsub}
		\DeclareMathOperator{\ev}{ev}
		\DeclareMathOperator{\tw}{tw}
		\newcommand{\opfrac}{{\operatorname{frac}}}
		\DeclareMathOperator{\dist}{dist}
        \DeclareMathOperator{\Spec}{Spec}

		\DeclarePairedDelimiter{\multiset}{\{\!\!\{}{\}\!\!\}}

		\newcommand{\lG}{\mathcal{G}^\circ}
		\newcommand{\G}{\mathcal{G}}
        \newcommand{\lZ}{\mathcal{Z}^\circ}
		\newcommand{\Z}{\mathcal{Z}}
		\newcommand{\C}{\mathbb{C}}
		\newcommand{\R}{\mathbb{R}}

        \renewcommand\theta\vartheta

        \newcommand{\defiff}{:\!\!\iff}
\begin{document}
	\maketitle

	\begin{abstract}
		Two graphs $G$ and $H$ are \emph{homomorphism indistinguishable}
		over a graph class $\mathcal{F}$ if, for every graph $F \in \mathcal{F}$,
		the number of homomorphisms from $F$ to $G$ is equal to the number of homomorphisms from $F$ to $H$.
		Lov\'asz (\emph{Acta Mathematica Academiae Scientiarum Hungarica}, 1967)
		showed that two graphs are isomorphic if, and only if, they are homomorphism indistinguishable over all graphs.
		Subsequently, homomorphism indistinguishability relations of a long list of natural graph classes have been equated 
		with natural graph isomorphism relaxations.

		Given the wealth of such results,
		Atserias, Kolaitis, \& Wu (\textsmaller{LICS}~2021)
		asked for an axiomatic characterisation of homomorphism indistinguishability relations.
		By exhibiting topological and geometric structure associated with homomorphism indistinguishability,
		we derive such an axiomatic characterisation.
        Here, a central ingredient is a novel characterisation of graph parameters of the form $\hom(F, \star)$ for some graph $F$ alternative to a previous result of Lov\'asz \& Schrijver (\textsmaller{JCTA}~2010).
		Moreover, we investigate the topology of homomorphism indistinguishability and
		discuss repercussions for the Ulam--Kelly Reconstruction Conjecture.
	\end{abstract}

    \nocite{lovasz_operations_1967,atserias_expressive_2021,lovasz_dual_2010}

	\section{Introduction}

	Graph data arises in various contexts ranging from program flows, social networks, protein interactions, and trophic networks.
	A central task when presented with graph data is to decide whether two graphs are \emph{isomorphic}, i.e.,
	structurally equivalent.
	However, the algorithmic problem of deciding whether two graphs are isomorphic is riddled with
	complexity-theoretic difficulties \cite{grohe_graph_2020}.
	Furthermore, in many applications, isomorphism is too brittle an equivalence relation between graphs which does not account for noise or perturbations.
	These two shortcomings of graph isomorphism have sparked a decade-long line of work 
	on \emph{graph isomorphism relaxations}, i.e.\
	equivalence relations $\approx$ between graphs that satisfy
	for all graphs $G$ and $H$ that
	\begin{equation} \label{eq:graph-iso-relax}
		G \cong H \implies G \approx H.
	\end{equation}

	Graph isomorphism relaxations may stem from algorithmic heuristics such as the Weisfeiler--Leman algorithm \cite{kiefer_weisfeiler-leman_2020},
	machine learning paradigms such as Graph Neural Networks \cite{grohe_logic_2021,xu_how_2019,morris_weisfeiler_2019}, or from
	theoretical considerations rooted e.g.\ in spectral graph theory or quantum information theory \cite{atserias_quantum_2019}.
	Despite the long list of intensively studied graph isomorphism relaxations,
	they lack a coherent theory predicting their computational complexity or distinguishing power \cite{grohe_thoughts_2025}.

	A proposal for such a theory, which has been brought forward in recent years, is 
	\emph{homomorphism indistinguishability} \cite{seppelt_homomorphism_2024}.
	Two graphs $G$ and $H$ are \emph{homomorphism indistinguishable} over a graph class $\mathcal{F}$,
	in symbols $G \equiv_{\mathcal{F}} H$,
	if, for every graph $F \in \mathcal{F}$, the number $\hom(F, G)$ of homomorphisms 
	from $F$ to $G$ is equal to the number $\hom(F, H)$ of homomorphisms from $F$ to~$H$.
	A diverse range of graph isomorphism relaxations including isomorphism \cite{lovasz_operations_1967}, 
	quantum isomorphism \cite{mancinska_quantum_2020},
	and
	equivalences with respect to counting logic fragments \cite{dvorak_recognizing_2010,dell_lovasz_2018,schindling_homomorphism_2025,montacute_pebble-relation_2024,dawar_lovasz-type_2021,adler_going_2026}
	can be characterised as homomorphism indistinguishability relations over natural graph classes.
	Other examples include graph isomorphism relaxations originating from quantum information theory \cite{mancinska_quantum_2020,seppelt_quantum_2025}, optimisation \cite{kar_npa_2026,roberson_lasserre_2024}, category theory \cite{dawar_lovasz-type_2021,abramsky_discrete_2022}, and machine learning \cite{zhang_beyond_2024,gai_homomorphism_2025}.

	These results demonstrate the versatility of homomorphism indistinguishability
	and justify studying this notion more abstractly. 
	In this context, three research strands have emerged:
	The first studies the computational complexity of homomorphism indistinguishability relations $\equiv_{\mathcal{F}}$ in relation with graph-theoretic or model-theoretic properties of the graph class $\mathcal{F}$ \cite{seppelt_algorithmic_2024,cerny_homomorphism_2026,curticapean_dense_2026}.
	The second is concerned with the distinguishing power of $\equiv_{\mathcal{F}}$ and aims to understand for which relations $\equiv_{\mathcal{F}}$ differ from isomorphism~$\cong$, cf.\ Roberson's conjectures \cite{roberson_oddomorphisms_2022} and \cite{neuen_distinguishing_2026,jimenez_homomorphism_2026,kristjansson_oddomorphisms_2026,kristjansson_counterexamples_2026}.
	The final strand is concerned with abstract properties of homomorphism indistinguishability relations~\cite{seppelt_logical_2024,atserias_expressive_2021,abramsky_discrete_2022,lichter_limitations_2024}.
	In this context, \textcite{atserias_expressive_2021} posed the following problem:
	\begin{quote}
		\enquote{One of the ultimate goals is to characterize the relaxations of graph isomorphism that are captured by restrictions of the left profile [$\ldots$].} \cite[Section~VIII]{atserias_expressive_2021}
	\end{quote}
	In other words, they asked which graph isomorphism relaxations are homomorphism indistinguishability relations. 
	As our central  contribution (\cref{thm:axiomatisation}), we resolve this ultimate goal.
	Towards \cref{thm:axiomatisation}, we significantly advance the structural understanding of homomorphism indistinguishability
	by
	the means of geometry and topology.
    As corollaries, we answer open questions of \textcite{roberson_oddomorphisms_2022,botteron_communication_2024,atserias_expressive_2021} in \cref{xmpl:D-frac,lem:roberson-question,ex:fractional-clique-number} and simplify arguments from \cite{atserias_expressive_2021} in \cref{subsec:FO-equiv}.

    \subsection*{Techniques and results}
    
	For every graph class $\mathcal{F}$, its homomorphism indistinguishability relation $\equiv_{\mathcal{F}}$ is a graph isomorphism relaxation, i.e., satisfies \eqref{eq:graph-iso-relax}.
	In fact, this construction is part of a Galois correspondence between graph classes (ordered by inclusion) and graph isomorphism relaxations (ordered by refinement) given by the maps
	\begin{align}\label{eq:galois-hom}
		\mathcal{F} \mapsto {\equiv_{\mathcal{F}}} \quad \text{and}\quad 
		{\approx} &\mapsto \mathcal{K}({\approx}) \coloneqq \{F \mid \forall G, H.\ G \approx H \implies \hom(F, G) = \hom(F, H)\}.
	\end{align}
	Of particular interest are the objects that are invariant under going back and forth through this correspondence.
	The graph classes $\mathcal{F}$ such that $\mathcal{F} = \mathcal{K}(\equiv_{\mathcal{F}})$
	are precisely the homomorphism distinguishing closed  graph classes.
	A graph class $\mathcal{F}$ is \emph{homomorphism distinguishing closed} \cite{roberson_oddomorphisms_2022} if, for every $F' \not\in \mathcal{F}$,
	the relation $\equiv_{\mathcal{F} \cup \{F'\}}$ is strictly finer than $\equiv_{\mathcal{F}}$.
    For homomorphism distinguished closed graph classes $\mathcal{F}_1$ and $\mathcal{F}_2$,
    the problem of comparing the graph isomorphism relaxations $\equiv_{\mathcal{F}_1}$ and $\equiv_{\mathcal{F}_2}$ reduces to comparing the graph classes $\mathcal{F}_1$ and $\mathcal{F}_2$.
    That is, $\equiv_{\mathcal{F}_1}$ refines $\equiv_{\mathcal{F}_2}$ if, and only if, $\mathcal{F}_1 \supseteq \mathcal{F}_2$.
    This approach has been used for measuring the distinguishing power of graph isomorphism relaxations from mathematical programming \cite{roberson_lasserre_2024}, counting logic \cite{adler_going_2026,schindling_homomorphism_2025}, and machine learning \cite{gai_homomorphism_2025,zhang_beyond_2024}.
    Therefore, it is desirable to identify natural criteria for a graph class $\mathcal{F}$ to be homomorphism distinguishing closed.
	Roberson's (strong) conjecture \cite{roberson_oddomorphisms_2022} asserting that \emph{every graph class closed under taking minors and disjoint unions is homomorphism distinguishing closed} 
	has received much attention as the central hypothesis regarding the distinguishing power of homomorphism indistinguishability relations \cite{neuen_distinguishing_2026,jimenez_homomorphism_2026,kristjansson_oddomorphisms_2026}.
    Its recent refutation by \textcite{kristjansson_counterexamples_2026} warrants identifying sufficient conditions for a graph class to be homomorphism distinguishing closedness other than being minor-closed.

    Dually, the graph isomorphism relaxations $\approx$ such that ${\approx} = {\equiv_{\mathcal{K}({\approx})}}$
	are precisely the homomorphism indistinguishability relations.
	Thus, the quest for homomorphism distinguishing closed graph classes and \citeauthor{atserias_expressive_2021}'s \enquote{ultimate goal} both require to understand the objects preserved under the Galois correspondence~\eqref{eq:galois-hom}. 

	We propose to study the Galois correspondence~\eqref{eq:galois-hom} via an analogy with one of the most fundamental Galois correspondences in mathematics: 
	the correspondence between sets $U \subseteq \mathbb{C}^n$, so-called \emph{varieties}, and \emph{ideals} $\mathfrak{a}$ of the polynomial ring $\mathbb{C}[x_1, \dots, x_n]$ given by
	\begin{align*}
		U &\mapsto I(U) \coloneqq \left\{p \in \mathbb{C}[x_1, \dots, x_n] \mid\forall u \in U .\  p(u) = 0 \right\}, \\
		\mathfrak{a} & \mapsto V(\mathfrak{a}) \coloneqq \left\{u \in \mathbb{C}^n \mid \forall p \in \mathfrak{a}.\ p(u) = 0 \right\}.
	\end{align*}
	The objects invariant under this correspondence are well understood:
	By Hilbert's Nullstellensatz, the ideals $\mathfrak{a}$ such that $\mathfrak{a} = I(V(\mathfrak{a}))$ are precisely the \emph{radical} ideals of $\mathbb{C}[x_1, \dots, x_n]$, i.e., the ideals $\mathfrak{a}$ satisfying that, if $p^\ell \in \mathfrak{a}$ for $p \in \mathbb{C}[x_1, \dots, x_n]$ and $\ell \geq 1$, then $p \in \mathfrak{a}$.
	Dually, a variety $U \subseteq \mathbb{C}^n$ satisfies $U = V(I(U))$ if, and only if, it is closed in the Zariski topology on $\mathbb{C}^n$.
	
	We translate these characterisations of invariant objects to homomorphism indistinguishability. 
	On the side of graph classes, this implies that the homomorphism distinguishing closed sets constitute the closed sets of a Zariski-like topology (\cref{lem:intersections,lem:componental-closed}).
	Homomorphism indistinguishability was not previously known to admit topological structure.
	As a preliminary investigation of this topological space, we show that it is disconnected, non-compact, and Fr\'echet (\cref{prop:disconnected-noncompact,prop:frechet}).
	On the side of relations, we obtain the desired characterisation of homomorphism indistinguishability relations (\cref{thm:axiomatisation}).

	How can the Galois correspondence~\eqref{eq:galois-hom} be cast into the language of algebraic geometry?
	It was observed in \cite{seppelt_logical_2024} that homomorphism indistinguishability relations $\equiv_{\mathcal{F}}$ satisfy a multitude of preservation properties.
	For example, all homomorphism indistinguishability relations are \emph{preserved under categorical products}, i.e.\ for all graphs $G$, $H$, and $K$, if $G \equiv_{\mathcal{F}} H$, then $G \otimes K \equiv_{\mathcal{F}} H \otimes K$, where $\otimes$ denotes the categorical product.
	These preservation properties can be understood as asserting that homomorphism indistinguishability relations are congruences for the semiring on the set of (isomorphism types of) looped-simple graphs $\lG$ with disjoint union as addition and categorical product as multiplication.
	By a standard Grothendieck construction and an extension of scalars,
	we show that homomorphism indistinguishability relations form a class of ideals of the semiring algebra $\mathbb{C}\lG$.
	The algebra $\mathbb{C}\lG$ is a quotient of the complex polynomial ring in countably many variables.
	A version of Hilbert's Nullstellensatz applies to this algebra and lays the foundations for our \cref{thm:axiomatisation}.
	Thinking of the elements of $\mathbb{C}\lG$ as polynomials also allows us to recast the homomorphism distinguishing closed graph classes  as varieties.

	\subsection*{Related work}

	The algebraic structure of homomorphism counts is well explored from multiple perspectives ranging from algorithmics \& complexity \cite{curticapean_homomorphisms_2017,chen_counting_2016,roth_parameterised_2026}, and extremal combinatorics \cite{lovasz_large_2012} to database theory \cite{kwiecien_determinacy_2022}.
	We highlight the following selection of related work.

	By using an algebra whose elements are formal linear combinations of (labelled) graphs,
	\textcite{lovasz_dual_2010} characterised the functions $p \colon \lG \to \mathbb{R}$ that are of the form $p(\star) = \hom(F, \star)$
	for some graph~$F$.
	Their characterisation involves a condition which asserts that certain infinite matrices called connection matrices are positive semidefinite. 
	It is related to similar characterisations for functions $\hom(\star, G)$ and categorical generalisations proven in \cite{freedman_reflection_2007,lovasz_semidefinite_2009,lovasz_large_2012}.
	We prove an alternative such characterisation in \cref{cor:fls-alternative}.
	Our algebra set-up differs from the one from \cite{lovasz_dual_2010}:
	While their algebra is multisorted, we use tools from commutative algebra and elementary algebraic geometry. 

	Another ingredient of this work are graph semirings and homomorphisms on such semirings.
	\textcite{zuiddam_asymptotic_2019} studied similar objects when characterising the Shannon capacity of a graph in terms of asymptotic spectra.
	Here, the \emph{asymptotic spectrum} of a graph is the set of all semiring homomorphisms $\phi \colon (\mathcal{G}, +, \boxtimes, \emptyset, K_1) \to \mathbb{R}_{\geq 0}$ such that for all graphs $G, H$, if there is a homomorphism $\overline{G} \to \overline{H}$, then $\phi(G) \leq \phi(H)$.
	Here, $\boxtimes$ denotes the strong graph product and $\overline{G}$ the complement of $G$.
	We study a different graph semiring (with the categorical product $\otimes$ rather than the strong product $\boxtimes$). 
	Furthermore, we do not make use of any order structure and Hilbert's Nullstellensatz replaces the Vergleichsstellens\"atze of \textcite{strassen_asymptotic_1988}, see \cite{fritz_unified_2021}.

	In the context of game comonads, cf.\ \cite{dawar_lovasz-type_2021}, \textcite{abramsky_discrete_2022} characterised the homomorphism indistinguishability relations $\equiv_{\mathcal{F}}$ that are \emph{classified by a comonad}, i.e., there exists a comonad $\mathfrak{C}$ on the category of graphs such that $G \equiv_{\mathcal{F}} H$ if, and only if, the cofree $\mathfrak{C}$-coalgebras of $G$ and $H$ are isomorphic.
	The graph classes $\mathcal{F}$ arising here are precisely the \emph{componental} graph classes, i.e., those closed under disjoint union and taking summands. 

	Closest to our axiomatisation of homomorphism indistinguishability relations (\cref{thm:axiomatisation}) is perhaps a result of \textcite{seppelt_logical_2024} characterising the homomorphism indistinguishability relations $\equiv_{\mathcal{F}}$ for which $\mathcal{F}$ can be assumed to have certain closure properties such as being closed under taking (induced) subgraphs or minors.
	In particular, a homomorphism indistinguishability relation $\equiv_{\mathcal{F}}$ is a homomorphism indistinguishability relation over a minor-closed graph class if, and only if, $\equiv_{\mathcal{F}}$ is \emph{preserved under complements}, i.e., $G \equiv_{\mathcal{F}} H \implies \overline G \equiv_{\mathcal{F}} \overline H$ for all graphs $G, H$.
    Combined with our \cref{thm:axiomatisation-simple}, this gives an axiomatisation of homomorphism indistinguishability relations over minor-closed graph classes.

	\section{Preliminaries}

	\paragraph{Graphs.} All graphs are finite, undirected, and without multiple edges, i.e.\ they are \emph{looped-simple}.
	Write $\lG$ for the class of (isomorphism types of) looped-simple graphs.
	A \emph{simple} graph is a graph without loops.
    Write $\G$ for the class of (isomorphism types of) simple graphs.
	A \emph{homomorphism} between graphs $F$ and $G$ is a map $h \colon V(F) \to V(G)$ such that $uv \in E(F)$ implies $h(u)h(v) \in E(G)$;
	in particular, homomorphisms map loops to loops.
    Write $\hom(F, G)$ for the number of homomorphisms $F \to G$.
	The \emph{categorical product} $G \otimes H$ of two graphs $G$ and $H$ has vertex set $V(G \otimes H) = V(G) \times V(H)$ and edges between $gh$ and  $ g'h'$
	if $gg' \in E(G)$ and $hh' \in E(H)$.
	We write $G + H$ for the \emph{disjoint union} of $G$ and $H$.
	A graph is \emph{connected} if it has exactly one connected component; in particular, the vertex-less graph $\boldsymbol 0$ is not connected.
    Write $\Z \subseteq \G$ and $\lZ \subseteq \lG$ for the classes of connected graphs in $\G$ and $\lG$, respectively (\emph{zusammenhängend} means connected).

	\paragraph{Semirings.}
	A \emph{semiring} is a tuple $(S, +, \times, 0, 1)$
	where $S$ is a set equipped with two binary operations $+$ and $\times$ and distinguished elements $0,1 \in S$ satisfying that:
	\begin{assertions}
		\item $(S, +, 0)$ is a commutative monoid,
		\item $(S, \times,1)$ is a monoid,
		\item multiplication distributes over addition: $a \times (b + c) =
		a \times b + a \times c$ and $(a + b) \times c = a \times c + b \times c$
		for all $a, b, c \in S$,
		\item $0 \times a = a \times 0= 0$ for all $a \in S$.
	\end{assertions}
	The semiring is \emph{commutative} if $(S, \times, 1)$ is a commutative monoid, i.e., $a \times b = b \times a$ for all $a, b\in S$.
	The semiring is \emph{cancellative} 
	if the additive monoid $(S, +, 0)$ is \emph{cancellative}, i.e., for all $a, b, c \in S$, if $a + c = b + c$, then $a = b$.

	\begin{lemma}\label{lem:graphs-form-semiring}
		The class of looped-simple graphs $\lG$ with disjoint union $+$ as addition, categorical product $\otimes$ as multiplication, the vertex-less graph $\boldsymbol{0}$ and the one-vertex looped graph $\boldsymbol{1}$ as neutral elements forms a countable cancellative commutative semiring.
	\end{lemma}

	\section{Which relations are homomorphism indistinguishability relations?}
	\label{sec:which-relations}

	Towards answering this question, 
	we build up a collection of necessary conditions that in conjunction will turn out to be sufficient.
	Full proofs are given in \cref{app:which-relations}.
	In particular, 
	we develop an algebraic machinery that is inspired by the following three basic identities 
	satisfied by homomorphism counts.
	For graphs $G, G_1, G_2, F, F_1, F_2 \in \lG$, and every connected graph $Z \in \lZ$, the following hold
	\cite[(5.28)--(5.30)]{lovasz_large_2012}:
	\begin{align}
		\hom(F_1 + F_2, G) & = \hom(F_1, G) \cdot \hom(F_2, G)\label{eq:coprod},\\
		\hom(F, G_1 \otimes G_2) &= \hom(F, G_1) \cdot \hom(F, G_2) \label{eq:prod}, \\
		\hom(Z, G_1 + G_2) &= \hom(Z, G_1) + \hom(Z, G_2).\label{eq:connected} 
	\end{align}
	These basic identities are not hard to show and demonstrate the categorical nature of homomorphism counts.
	\Cref{eq:prod} implies that every homomorphism indistinguishability $\approx$ is \emph{preserved under categorical products}, i.e.\
	for all graphs $G, H, K \in \lG$,
	\begin{align}
		\label{eq:h1}
		G \approx H \implies K \otimes G \approx K \otimes H.
	\end{align}
	
	\paragraph{Componental graph classes.}
	\Cref{eq:coprod,eq:connected} imply a similar identity for disjoint union instead of $\otimes$ with a small caveat:
	\Cref{eq:connected} only applies to connected graphs.
	We therefore restrict our attention to homomorphism indistinguishability relations over componental graph classes.
	A graph class $\mathcal{F}$ is \emph{componental} if, for all graphs $F_1$, $F_2$, it holds that $F_1, F_2 \in \mathcal{F}$ iff $F_1 + F_2 \in \mathcal{F}$.
	Restricting to componental graph classes is standard in the context of homomorphism indistinguishability \cite{roberson_oddomorphisms_2022,abramsky_discrete_2022,seppelt_logical_2024}
	and arguably does not constitute a significant loss of generality.
	In particular, we are not aware of any natural homomorphism indistinguishability relation over a non-componental graph class.
    Every homomorphism indistinguishability relation $\approx$ over a componental graph class 
	satisfies the following identity\footnote{It was shown in \cite[Theorem~7]{seppelt_logical_2024} that a homomorphism indistinguishability relation $\approx$ satisfies \eqref{eq:h2} if, and only if, it is a homomorphism indistinguishability relation over a componental graph class.}
	 for all graphs $G, H, K \in \lG$,
	\begin{align}
		\label{eq:h2}
		G \approx H &\implies K + G \approx K + H.
	\end{align}

	\paragraph{Semiring congruences.}
	Before adding further necessary conditions, we take a step back and phrase \cref{eq:h1,eq:h2} more abstractly.
	A graph isomorphism relaxation $\approx$ satisfying these conditions is in fact a congruence for a semiring on the class $\lG$ of looped-simple graphs with disjoint union as addition, whose neutral element is the vertex-less graph $\boldsymbol{0}$,
	and with the categorical product as multiplication, whose neutral element is the single-vertex looped graph~$\boldsymbol{1}$ (\cref{lem:graphs-form-semiring}).
	\cref{eq:h1,eq:h2} assert that $\approx$ is a congruence for $\lG$.

	\paragraph{Ring congruences.}
	As a next step, we add further necessary conditions which allow us to work with a ring instead of a semiring.
	To that end, we introduce additive inverses to~$\lG$.
	The \emph{Grothendieck ring} $\mathbb{Z}\lG$ of the semiring $\lG$
	is the standard construction in this context.
	Its elements are of the form $G-H$ for $G, H \in \lG$.
	Here, $G - H = G' - H'$ holds in $\mathbb{Z}\lG$ iff $G + H' = G' + H$ holds in $\lG$.
	A relation $\approx$ is a congruence for $\mathbb{Z}\lG$ iff it satisfies \cref{eq:h1,eq:h2} 
	and the following condition for all graphs $G, H, K \in \lG$
	\begin{align}
		K + G \approx K + H &\implies G \approx H. \label{eq:h3}
	\end{align}
	Congruences of a ring are in one-to-one correspondence with the ring's ideals.
	That is, we may identify relations $\approx$ satisfying \cref{eq:h1,eq:h2,eq:h3} with ideals of $\mathbb{Z}\lG$ (\cref{prop:ideals-to-relations}).
	
	\paragraph{Extension of scalars.}
	Towards applying Hilbert's Nullstellensatz, we extend scalars in order to turn $\mathbb{Z}\lG$ into an algebra over an algebraically closed field.
	Here, the only non-trivial step is from $\mathbb{Z}\lG$ to $\mathbb{Q}\lG \coloneqq \mathbb{Q} \otimes_{\mathbb{Z}} \mathbb{Z}\lG$.
	The elements of $\mathbb{Q}\lG$ are of the form $\frac{G - H}{n}$ for graphs $G, H \in \lG$ and integers $n \geq 1$.
	It holds that $\frac{G - H}{n} = \frac{G'-H'}{n'}$ in $\mathbb{Q}\lG$ iff $n'(G-H) = n(G'-H')$ in $\mathbb{Z}\lG$, which holds iff $n'G + nH' = n'H + nG'$ in $\lG$.
	To that end, we add the necessary condition that, for all graphs $G, H \in \lG$, and integers $n \geq 1$,
	\begin{align}
		\underbrace{G + \dots + G}_{n \text{ times}}
		\approx
		\underbrace{H + \dots + H}_{n \text{ times}}
		&\implies G \approx H. \label{eq:h4}
	\end{align}
	It is immediate from \cref{eq:connected} that every componental homomorphism indistinguishability relation $\approx$ satisfies \cref{eq:h4}.
	Furthermore, it holds that $\approx$ is a congruence of $\mathbb{Q}\lG$ iff 
	it satisfies \cref{eq:h1,eq:h2,eq:h3,eq:h4}.
	Such congruences are in one-to-one correspondence with ideals of $\mathbb{Q}\lG$ (\cref{prop:ideals-to-relations-rational}).
	More specifically, for a relation $\approx$ satisfying \cref{eq:h1,eq:h2,eq:h3,eq:h4}, let
	\[
		\mathfrak{a}_{\approx} \coloneqq \left\{ \frac{G - H}{n} \ \middle|\ G, H \in \lG, G\approx H, n \geq 1\right\}.
	\]
	Then $\mathfrak{a}_{\approx}$ is an ideal of $\mathbb{Q}\lG$.
	Conversely, for an ideal $\mathfrak{b}$ of $\mathbb{Q}\lG$, define a relation $\approx_{\mathfrak{b}}$ via
	\[
		G \approx_{\mathfrak{b}} H \iff G - H \in \mathfrak{b}.
	\]
	Then $\approx_{\mathfrak{b}}$ satisfies \cref{eq:h1,eq:h2,eq:h3,eq:h4}.
	Furthermore, the constructions ${\approx} \mapsto \mathfrak{a}_{\approx}$
	and $\mathfrak{b} \mapsto {\approx_{\mathfrak{b}}}$ are mutually converse.
	It remains to understand the ideals $\mathfrak{a}$ of $\mathbb{Q}\lG$
	that encode homomorphism indistinguishability relations.
	For algebraic convenience, we extend scalars further to $\mathbb{C}\lG \coloneqq \mathbb{C} \otimes_{\mathbb{Q}} \mathbb{Q}\lG$
	without incurring additional axioms (\cref{lem:c-ideals}).

	\paragraph{Hilbert's Nullstellensatz.}
	How should we think of the algebra $\mathbb{C}\lG$?
	It looks much like the polynomial algebra $\mathbb{C}[G_1, G_2, \dots]$ in countably many variables enumerating the graphs in $\lG$.
	In fact, it is a quotient of this polynomial algebra imposing the identities that hold in $\lG$.
	For example, it holds in $\lG$ that $K_2 \otimes K_2 = K_2 + K_2$.
	Nevertheless, we think of the elements of $\mathbb{C}\lG$ as polynomials whose indeterminates are the elements of $\lG$.
	
	Polynomials can be evaluated by substituting field elements for the variables. 
	Similarly, we may \enquote{evaluate} the elements of $\mathbb{C}\lG$ by substituting field elements for the graphs in $\lG$.
	However, such a substitution must respect the intrinsic identities that hold in $\mathbb{C}\lG$.
	This observation is formalised by the notion of a \emph{semiring homomorphism} $\phi \colon \lG \to \mathbb{C}$, i.e.\ maps satisfying, for all graphs $G, H \in \lG$,
	\[
		\phi(G + H) = \phi(G) + \phi(H), \qquad
		\phi(G \otimes H) = \phi(G) \cdot \phi(H), \qquad
		\phi(\boldsymbol{1}) = 1.
	\]
	We write $\Hom(\lG, \mathbb{C})$ for the set of such semiring homomorphisms.
	Crucially, by \cref{eq:prod,eq:connected}, the map $\hom(F, \star) \colon \lG \to \mathbb{C}$
	is a semiring homomorphism for every connected graph~$F \in \lZ$.
	Thus counting homomorphisms from a connected graph amounts to substituting field elements for the \enquote{indeterminates} of $\mathbb{C}\lG$.

	The ideal corresponding to a componental homomorphism indistinguishability relation $\equiv_{\mathcal{F}}$ is the vanishing ideal of the semiring homomorphisms $\hom(F, \star)$ for $F \in \mathcal{F}$.
	More generally, for a set $\Phi \subseteq \Hom(\lG, \mathbb{C})$ of semiring homomorphisms and an ideal $\mathfrak{a}$ of $\mathbb{C}\lG$,
	we define the \emph{vanishing ideal} and the \emph{zero locus}
    \begin{align*}
        I(\Phi) &\coloneqq \{p \in \mathbb{C}\lG \mid \forall \phi \in \Phi.\ \phi(p) = 0\}
        \\
        V(\mathfrak{a}) &\coloneqq \{\phi \in \Hom(\lG, \mathbb{C}) \mid \forall p \in \mathfrak{a}.\ \phi(p) = 0\}.
    \end{align*}
	Since homomorphism indistinguishability relations are vanishing ideals, we first characterise such ideals in terms of the relation $\approx$.
	The key ingredient here is a variant of Hilbert's Nullstellensatz.
	We show in \cref{thm:hilbert-nullstellensatz-semiring} that
    \[
		I(V(\mathfrak{a})) = \sqrt{\mathfrak{a}}
        \qquad
        \text{for every ideal } \mathfrak{a} \subseteq \C\lG.
	\]
	Here, $\sqrt{\mathfrak{a}} \coloneqq \{p \in \mathbb{C}\lG \mid \exists n \geq 1.\ p^n \in \mathfrak{a}\}$ is the \emph{radical} of $\mathfrak{a}$.
	It follows that an ideal  $\mathfrak{a}$ of $\mathbb{C}\lG$ is the vanishing ideal of some $\Phi \subseteq \Hom(\lG, \mathbb{C})$
	if, and only if, it is \emph{radical}, i.e.\ $\mathfrak{a} = \sqrt{\mathfrak{a}}$.
	This condition can be phrased in terms of the relation $\approx$ by observing that elements of $\mathbb{C}\lG$ are up to scalars of the form $G - H$ and applying the Binomial Theorem.
	For all graphs $G, H \in \lG$, and $n \geq 1$, we assert that
	\begin{equation}\label{eq:h5}
		\sum_{\substack{k=0,\ k \text{ even}}}^n \binom{n}{k} G^{\otimes k} \otimes H^{\otimes (n-k)}
		\approx 
		\sum_{\substack{k=1,\ k \text{ odd}}}^n \binom{n}{k} G^{\otimes k} \otimes H^{\otimes (n-k)}
		\implies
		G \approx H.
	\end{equation}

	Let us summarise which class of relations we have axiomatised so far.

	\begin{theorem}[restate=thmPhi, label=thm:semiring-hom-congruence]
		Let $\approx$ be a graph isomorphism relaxation.
		The relation $\approx$ satisfies  \cref{eq:h1,eq:h2,eq:h3,eq:h4,eq:h5}
		if, and only if, 
		there exists a set $\Phi \subseteq \Hom(\lG, \mathbb{C})$
		such that, for all graphs $G, H \in \lG$,
		\[
		G \approx H \iff \forall \phi \in \Phi.\ \phi(G) = \phi(H).
		\]
	\end{theorem}

	So far, even though we have concerned ourselves with the semiring $\lG$,
	we have in fact not used any specific properties of this semiring.
	More precisely, \cref{thm:semiring-hom-congruence} holds for countable commutative cancellative semirings in general (\cref{thm:phi-abstract}).
	This already indicates that \cref{thm:semiring-hom-congruence} is most likely not the final axiomatisation of componental homomorphism indistinguishability relations.
	Indeed, we give the following counterexample, see \cref{sec:examples}.
	Moreover, one might wonder whether considering the complex numbers in \cref{thm:semiring-hom-congruence} is necessary.
	We show that it is in \cref{cor:non-bipartite-quotient} by constructing a $\phi \in \Hom(\lG, \C)$ with imaginary values.

	\begin{example}[restate=exRank,label=ex:rank]
		The adjacency matrix rank $\rk \colon \lG \to \mathbb{C}$ is a semiring homomorphism.
		The relation $G \approx_{\rk} H \iff \rk(G) = \rk(H)$ satisfies  \cref{eq:h1,eq:h2,eq:h3,eq:h4,eq:h5}
		but is not a homomorphism indistinguishability relation.
	\end{example}

	We must thus reintroduce the combinatorial structure of $\lG$ to the purely algebraic conditions of \cref{thm:semiring-hom-congruence}.
	More precisely, we must axiomatise $\{\hom(F, \star) \mid F \in \lZ \}$
	inside $\Hom(\lG, \mathbb{C})$.
	The difficulty, however, is that we do not have access to the values of the $\phi$ in \cref{thm:semiring-hom-congruence}
	but merely to the relation $\approx$.
	This is also why the characterisation by \textcite{lovasz_dual_2010}
	does not appear to be applicable.
	As a first step, we prove an alternative characterisation more suitable for our needs, which may be of independent interest.

	\paragraph{Subgraph series.}
	Using M\"obius inversion, we show that every function $p \colon \lG \to \mathbb{C}$ can be written as
	\[
		p(\star) = \sum_{F \in \lG} \alpha_F \sub(F, \star),
        \qquad 
        \text{for } \alpha_F \in \mathbb{C}.
	\]
	The function $\sub(F, \star)$ counts the number of subgraphs isomorphic to~$F$.
	Hence, while the above sum is infinite, only finitely many summands contribute to the value of $p(G)$ for every graph $G \in \lG$.
	We call the expression $\sum_{F} \alpha_F \sub(F, \star)$ a \emph{subgraph series}.
	Conveniently, the coefficients $\alpha_F$ are uniquely determined by the values of $p(\star)$.
	More precisely, as shown in \cref{lem:unique-subgraph-series}, it holds that
	\begin{equation}\label{eq:subgraph-series-coefficients}
		\alpha_F = \sum_{S \subseteq V(F)} \ \sum_{E' \subseteq E(F)} (-1)^{|V(F) \setminus S| + |E(F) \setminus E'|} p(F[S, E'])
	\end{equation}
	for every graph $F\in \lG$.
	Here, $F[S, E']$ denotes the subgraph of $F$ with vertex set $S$ whose edges are those edges of $E'$ all of whose endpoints lie in $S$.
	We may thus define the \emph{support} of a graph parameter $p \colon \lG \to \mathbb{C}$ as the graph class
	\[
		\supp(p) \coloneqq \{F \in \lG \mid \alpha_F \neq 0\},
        \qquad
        \text{for } \alpha_F \in \C \text{ as given by \cref{eq:subgraph-series-coefficients}.}
	\]
	We prove the following new characterisation of homomorphism counts.

	\begin{theorem}[restate=thmSemiringHomChar,label=thm:characterisation-semiring-homs]
		A semiring homomorphism $\phi \in \Hom(\lG, \mathbb{C})$ is of the form $\phi(\star) = \hom(F, \star)$ for some connected graph $F \in \lG$
		if, and only if, its support $\supp(\phi)$ is finite.
	\end{theorem}

    The graph parameters $p$ with finite support $\supp(p)$ are precisely the so-called \emph{graph motif parameters} and intensively studied in counting complexity \cite{roth_parameterised_2026,curticapean_homomorphisms_2017,chen_counting_2016}.
	The fact that $\supp(\phi)$ is finite can be phrased in terms of vanishing ideals yielding an axiomatization of graph motif parameters.
	For semiring homomorphisms $\phi \in \Hom(\lG, \C)$, 
	the sum in \cref{eq:subgraph-series-coefficients}
	may be written as the evaluation of $\phi$ at a single element of $\mathbb{Z}\lG$.
	For a graph $F$, let
	\begin{equation}\label{eq:c-F}
		c_F \coloneqq \sum_{S \subseteq V(F)} \ \sum_{E' \subseteq E(F)} (-1)^{|V(F) \setminus S| + |E(F) \setminus E'|} F[S, E'] \in \mathbb{Z}\lG.
	\end{equation}
	It follows that $\phi \in \Hom(\lG, \mathbb{C})$ has finite support if, and only if,
	$\phi(c_F) = 0$ for all but finitely many $F \in \lG$.
	For an integer $\ell \geq 1$, write $\mathfrak{c}_\ell$ for the ideal of $\mathbb{C}\lG$ 
	generated by the $c_F$ for graphs $F \in \lG$ on more than $\ell$ vertices.
    Using \cref{thm:characterisation-semiring-homs}, we obtain the following characterisation.

	\begin{theorem}[restate=thmRadInt,label=thm:radical-intersection]
		Let $\approx$ be a graph isomorphism relaxation satisfying \cref{eq:h1,eq:h2,eq:h3,eq:h4}
		and let $\mathfrak{a}_{\approx}$ denote the corresponding ideal of $\mathbb{C}\lG$
		given by \cref{lem:c-ideals}.
		Then $\approx$ is a homomorphism indistinguishability relation over a componental graph class
		if, and only if,
		\begin{equation}\label{eq:radical-intersection}
			\mathfrak{a}_{\approx} = \bigcap_{ \ell \geq 1} (\mathfrak{a}_{\approx} + \mathfrak{c}_\ell).
		\end{equation}
	\end{theorem}

	In order to rephrase \eqref{eq:radical-intersection}
	in terms of the relation $\approx$, we characterise the homomorphism indistinguishability relation corresponding to the ideal $\mathfrak{c}_\ell$.
	We write $G \simeq_\ell H$ if $G - H \in \mathfrak{c}_\ell$.

	\begin{lemma}[restate=lemCl, label=lem:ideal-cl]
		Let $\ell \geq 1$. For graphs $G, H \in \lG$, the following are equivalent:
		\begin{assertions}
			\item $G \simeq_\ell H$, i.e., $G - H \in \mathfrak{c}_\ell$,
			\item $\hom(F, G) = \hom(F, H)$ for all $F\in \lG$ on at most $\ell$ vertices,
			\item there exists a bijection $\pi$ between the subsets of $V(G)$ and of $V(H)$ of size at most $\ell$ such that $|S| = |\pi(S)|$ and $G[S] \cong H[\pi(S)]$ for all such $S$.
		\end{assertions}
 	\end{lemma}

	Using \cref{lem:ideal-cl},
	we rephrase \cref{eq:radical-intersection} in terms of the relation $\approx$:
	For all $G, H \in \lG$,
	\begin{equation}\label{eq:h6}
		\left( \forall \ell \geq 1.\ \exists k \geq 1.\  \exists G' \simeq_\ell H'.\  k(G-H) \approx G' - H' \right) \implies G \approx H.
	\end{equation}
	Intuitively, \eqref{eq:h6} asserts that $G$ and $H$ must be in $\approx$ if they can be arbitrarily well approximated by taking graphs $G'$ and $H'$ satisfying $G' \simeq_\ell H'$. 
	
	\paragraph{Putting the pieces together, } 
	we obtain the desired axiomatisation of homomorphism indistinguishability relations.
	For convenience, we restate \cref{eq:h1,eq:h2,eq:h3,eq:h4,eq:h6}.
	\begin{theorem}[restate=thmAxiom,label=thm:axiomatisation]
		A graph isomorphism relaxation $\approx$ is a homomorphism indistinguishability relation over a componental graph class if, and only if, it satisfies, for all graphs $G, H, K \in \lG$ and $n \geq 1$,
		\begin{assertions}
			\item\label{it:a1} $G \approx H \implies K \otimes G \approx K \otimes H$,
			\item\label{it:a2} $G \approx H \implies K + G \approx K + H$,
			\item\label{it:a3} $K + G \approx K + H \implies G \approx H$,
			\item\label{it:a4} $nG \approx nH \implies G \approx H$, 
\item\label{it:a5} if, for all $\ell \geq 1$, there exist $k \geq 1$ and graphs $G', H' \in \lG$ such that $G' \simeq_\ell H'$ and $k(G-H) \approx G' - H'$, then $G \approx H$.
		\end{assertions}
	\end{theorem}

    \paragraph{Discarding loops.}
    Most of the homomorphism indistinguishability literature concerns relations between simple graphs, cf.\ \cite[Section~7]{roberson_oddomorphisms_2022} and \cite[p.~32]{seppelt_homomorphism_2024}. 
    For algebraic convenience, we chose to work with $\lG$ rather than $\G$ as $\G$ lacks a multiplicative unit. 
    Nevertheless, we deduce a characterisation of graph isomorphism relaxations between simple graphs that are homomorphism indistinguishability relations from \cref{thm:axiomatisation}.
    Restricted to $\G$, the axioms are the same, except for \cref{it:s6}, which  mitigates the lack of a unit in $\G$.

	\begin{theorem}[label=thm:axiomatisation-simple,restate=thmAxiomSimp]
		A simple graph isomorphism relaxation $\approx$ is a homomorphism indistinguishability relation
		over a componental class of simple graphs if, and only if,
		it satisfies, for all $G, H, K \in \G$,
		\begin{assertions}
			\item\label{it:s1} $G \approx H \implies K \otimes G \approx K \otimes H$,
			\item\label{it:s2} $G \approx H \implies K + G \approx K + H$,
			\item\label{it:s3} $K + G \approx K + H \implies G \approx H$,
			\item\label{it:s4} if $nG \approx nH $ for some $n\geq1$, then $ G \approx H$,
			\item\label{it:s5} if, for all $\ell \geq 1$, there exist $k \geq 1$ and simple graphs $G', H' \in \G$
			such that $G' \simeq_\ell H'$ and $k(G-H) \approx G' - H'$, then $G \approx H$,
			\item\label{it:s6} if $G \otimes K_n \approx H \otimes K_n$ for all $n \geq 1$ and the $n$-vertex complete graph $K_n$, then $G \approx H$.
		\end{assertions}
	\end{theorem}

    \paragraph{Applications.}
	Admittedly, \cref{it:a5} in \cref{thm:axiomatisation} is quite hard to relate to. 
	However, as we will see in the following example, already the first conditions can be as hard to establish as resolving the notorious Ulam--Kelly Reconstruction Conjecture \cite{ulam_collection_1960,kelly_congruence_1957}.
	This conjecture asserts that any two hypomorphic simple graphs are isomorphic.
	Two graphs $G, H \in \G$ are \emph{hypomorphic} if they have the same number of edges and there exists a bijection $\pi \colon V(G) \to V(H)$
	such that $G - v \cong H- \pi(v)$ for all $v \in V(G)$.\footnote{The Reconstruction Conjecture is commonly stated for graphs on at least three vertices,
	with hypomorphicity defined by the existence of $\pi$ alone.
	Both formulations are equivalent.
	Indeed, on at least three vertices, the number of edges is determined by the multiset of the $G - v$ \cite{kelly_congruence_1957},
	so the additional requirement is vacuous there.
	On at most two vertices, the only non-isomorphic graphs admitting such a bijection $\pi$ are $K_2$ and $\overline{K_2}$,
	which are ruled out as they differ in their number of edges.} 

	\begin{example}[restate=exUlam, label=ex:ulam]
		The following are equivalent:
		\begin{assertions}
			\item\label{it:ulam1} The Ulam--Kelly Reconstruction Conjecture holds.
			\item\label{it:ulam2} Hypomorphicity is a homomorphism indistinguishability relation over a componental graph class.
			\item\label{it:ulam3} The hypomorphicity relation satisfies \cref{it:s1} of \cref{thm:axiomatisation-simple}.
			\item\label{it:ulam4} The hypomorphicity relation satisfies \cref{it:s2} of \cref{thm:axiomatisation-simple}.
		\end{assertions}
	\end{example}

    As another application, we answer an open question raised by Botteron and Weber \cite[Open Question 2]{botteron_communication_2024}, who asked for a homomorphism indistinguishability characterisation of a graph isomorphism relaxation called \emph{$D$\nobreakdash-fractional isomorphism}.
    In \cref{xmpl:D-frac}, we show that this isomorphism relaxation fails condition \ref{it:a1} in \cref{thm:axiomatisation}, so it does not admit a homomorphism indistinguishability characterisation.
	
	We also use \cref{thm:axiomatisation} to significantly simplify the proofs of the negative results by \textcite{atserias_expressive_2021}. In \cref{subsec:FO-equiv}, we show that chromatic equivalence violates condition \ref{it:a1}, while $\operatorname{FO}^k$-equivalence and $\operatorname{FO}_q$-equivalence violate \ref{it:a3}.
    Here, $\operatorname{FO}^k$ and $\operatorname{FO}_q$ denote the $k$-variable and quantifier-rank-$q$ fragment of first-order logic, respectively.

    Finally, we answer an open question of \textcite{atserias_expressive_2021} in the negative: Having the same fractional clique number is not a homomorphism indistinguishability relation (\cref{ex:fractional-clique-number}).
    
	\section{Homomorphism distinguishing topology}
    \label{sec:topology}

	Having demonstrated that homomorphism indistinguishability relations can be regarded as certain ideals of a semiring algebra, we now turn to the dual perspective and discuss the topological structure of homomorphism indistinguishability.
	Here, we make explicit that the homomorphism distinguishing closed componental graph classes form the closed sets of a topology. 
	In order to be aligned with the literature, we focus on simple graphs (the connection to the looped-simple case is discussed in \cref{sec:loops}).

	\begin{definition}[\textcite{roberson_oddomorphisms_2022}]
    \label{def:closure}
		Let $\mathcal{F}$ be a class of simple graphs.
		The \emph{homomorphism distinguishing closure} of $\mathcal{F}$ is
		the graph class
		\[
		\cl(\mathcal{F}) \coloneqq \{F \in \G \mid \forall G, H \in \G.\ G \equiv_{\mathcal{F}} H \implies \hom(F, G) = \hom(F, H) \}.
 		\]
		The class $\mathcal{F}$ is \emph{homomorphism distinguishing closed}
		if $\cl(\mathcal{F}) = \mathcal{F}$.
	\end{definition}

	Roberson observed that $\cl(\star)$ is a closure operator, cf.\ \cref{eq:cl-extensive,eq:cl-idempotent,eq:cl-increasing}.
	Furthermore, Roberson showed the following:

	\begin{lemma}[\protect{\cite[Lemma~6.1]{roberson_oddomorphisms_2022}}]\label{lem:intersections}
    For an arbitrary index set $\Lambda$,
		let $\mathcal{F}_\lambda$ be homomorphism distinguishing closed graph classes for $\lambda \in \Lambda$.
		Then $\bigcap_{\lambda \in \Lambda} \mathcal{F}_\lambda$ is homomorphism distinguishing closed.
	\end{lemma}
	
	Given these properties, it is natural to ask whether $\cl(\star)$ induces a topology by declaring the sets invariant under $\cl(\star)$ as closed.
	This fails since the union of two closed sets is not necessarily closed.
	More precisely,
	it is not hard to see that homomorphism distinguishing 
	closed sets are always closed under disjoint unions as implied by \cref{eq:coprod}.
	Hence, the union of two homomorphism distinguishing closed graph classes is typically not homomorphism distinguishing closed.
	Moreover, in \cite[Example~34]{seppelt_logical_2024}, homomorphism distinguishing closed graph classes $\mathcal{F}_1$ and $\mathcal{F}_2$ are constructed such that the disjoint union closure of $\mathcal{F}_1 \cup \mathcal{F}_2$ is not homomorphism distinguishing closed.
	Thus, $\cl(\star)$ does not interact nicely with (disjoint union closures of) unions.
	We overcome this issue by restricting our attention to componental graph classes.
	For graph classes $\mathcal{F}_1$ and $\mathcal{F}_2$,
	we write $\mathcal{F}_1 + \mathcal{F}_2$ for the smallest componental graph class containing $\mathcal{F}_1$ and $\mathcal{F}_2$.
	Our key observation is the following lemma, which makes use of the algebra structure exhibited in \cref{sec:which-relations}.

    \begin{lemma}[label=lem:componental-closed,restate=lemComponentalClosed]
        If $\mathcal{F}_1, \mathcal{F}_2 \subseteq \G$
		are homomorphism distinguishing closed componental graph classes,
		then $\mathcal{F}_1 + \mathcal{F}_2$ is homomorphism distinguishing closed.
    \end{lemma}

	\Cref{lem:componental-closed} allows us to construct new homomorphism distinguishing closed graph classes by applying $+$.
	For example, invoking \cite{neuen_homomorphism-distinguishing_2024,roberson_oddomorphisms_2022}, we have the following.
	\begin{corollary}
		For $k, d \geq 1$,
		the graph class $\{F \in \G \mid \tw(F) \leq k \lor \Delta(F) \leq d \}$
		is homomorphism distinguishing closed.
		Here, $\tw(F)$ denotes the treewidth and $\Delta(F)$ the maximum degree of~$F$.
	\end{corollary}

	We restrict our attention to componental graph classes $\mathcal{F}$.
	Note that componental graph classes are in one-to-one correspondence with classes of connected graphs.
	Recall that $\mathcal{Z}$ denotes the class of all non-empty connected graphs.

	\begin{definition}\label{def:hom-dist-topology}
		The \emph{homomorphism distinguishing topology} is the topology on $\mathcal{Z}$ whose closed sets are the $\mathcal{F} \subseteq \mathcal{Z}$ such that $\cl(\mathcal{F}) \cap \Z = \mathcal{F}$.
	\end{definition}

	That this is indeed a topological space follows from \cref{lem:intersections,lem:componental-closed}.
	Finally, we list the following properties of the homomorphism distinguishing topology.
	In the following proposition, 
	a graph class $\mathcal{F}$ is \emph{essentially finite} if its componental closure contains finitely many non-isomorphic connected graphs.
	The proposition simplifies previous proofs of \cite{kwiecien_determinacy_2022,seppelt_logical_2024}.

	\begin{proposition}[label=prop:frechet,restate=propFrechet]
		The homomorphism distinguishing topology is Fr\'echet, i.e., all singletons are closed.
		In particular, every essentially finite componental graph class is homomorphism distinguishing closed.
	\end{proposition}

	\begin{proposition}[label=prop:disconnected-noncompact,restate=propDisconnected]
		The homomorphism distinguishing topology is disconnected and non-compact.
	\end{proposition}

	\section{Conclusion}

    By exhibiting geometric and topological structure associated with homomorphism indistinguishability, we gave an axiomatisation of homomorphism indistinguishability relations in \cref{thm:axiomatisation-simple,thm:axiomatisation} resolving the \enquote{ultimate goal} of \textcite{atserias_expressive_2021}.
    A crucial ingredient is a novel characterisation of functions $p \colon \lG \to \C$ which are of the form $p(\star) = \hom(F, \star)$ for some $F \in \lZ$ in \cref{thm:characterisation-semiring-homs,cor:fls-alternative} alternative to \cite{lovasz_dual_2010}.
    In our eyes, the central contribution of this paper, however, is to introduce a new algebraic language to reason about homomorphism indistinguishability.
    We propose to further develop this tool in the following directions.

    Most of this paper is not specific to the semiring $(\lG, +, \otimes, \boldsymbol{0}, \boldsymbol{1})$ but applies to every countable commutative cancellative semiring (\cref{thm:phi-abstract}).
    It would be interesting to apply the theorem to graph semirings with respect to other graph products \cite{imrich_associative_1975,imrich_monoid_2025,campanelli_graph_2017,zuiddam_asymptotic_2019}.
    Concretely, can the relations satisfying the axioms of \cref{thm:phi-abstract} for the strong or Cartesian product be naturally described?
    Moreover, given the few combinatorial properties our \cref{thm:axiomatisation} required,
    we expect a categorical generalisation as in \cite{lovasz_semidefinite_2009} to hold.
    
    A remaining open problem from \cite{atserias_expressive_2021} is a characterisation of right profiles, i.e., relations $\approx$ of the form $G \approx H \iff \hom(G, F) = \hom(H, F)$ for all $F$ in some graph class $\mathcal{F}$.
    Homomorphism indistinguishability to the right is much less studied \cite{atserias_expressive_2021,dawar_homomorphism_2025} than homomorphism indistinguishability to the left due to the small number of known interesting example relations \cite[p.~32]{seppelt_homomorphism_2024}.
    Since $\hom(\star, F)$ lacks an additive identity analogous to \cref{eq:connected}, our \cref{thm:phi-abstract} does not seem to be applicable.
    A similar problem arises when attempting to axiomatise homomorphism indistinguishability relations (to the left) of not necessarily componental graph classes.
    To our knowledge, no interesting examples of such relations are known.
    
    Further, we propose to connect our insights to the two other strands of research on homomorphism indistinguishability.
    On the one hand, it is plausible that the algebraic structure exhibited here may be useful when designing computational reductions between homomorphism indistinguishability problems, as implicitly done in \cite[Theorem~20]{cerny_homomorphism_2026}.
    Such reductions would serve as a tool towards resolving \cite[Conjecture~25]{seppelt_algorithmic_2024}.
    On the other hand, the homomorphism distinguishing topology (\cref{def:hom-dist-topology}) gives additional structure to homomorphism distinguishing closed graph classes in light of the Kristj\'ansson's refutation of Roberson's strong conjecture \cite{kristjansson_counterexamples_2026}.
    We propose to further investigate the topology: Is it Hausdorff? What are convergent sequences, continuous maps?

    Finally, the characterisation of left homomorphism profiles in terms of positive semidefiniteness of connection matrices (\textcite{lovasz_dual_2010}) suggests a connection to sums of squares certificates and real algebraic geometry. We do not know whether this connection can be harnessed to further simplify the conditions in \cref{thm:axiomatisation,thm:axiomatisation-simple}.
    
    \newpage 
	\section*{Acknowledgements}
    JvDdB and DER were supported by the Carlsberg Foundation Young Researcher Fellowship CF21-0682 --- ``Quantum Graph Theory''. Additionally, JvDdB was supported by GAČR grant 25-17377S. Part of this paper was written while JvDdB, DR, GLS, and PZ were in residence at Institut Mittag-Leffler in Djursholm, Sweden during the research semester ``Operator Algebras and Quantum Information'' in the spring of 2026, supported by the Swedish Research Council under grant no.~2021-06594.
    PZ was funded by the European Union (ERC, POCOCOP, 101071674). 
    GLS acknowledges support from the CNRS 80Prime grant ``QuantGraphe''.
    JM acknowledges the financial support from the European Research Council under the European Union's Horizon 2020 research and innovation programme (Synergy Project Malinca, ERC Grant Agreement No 670624).
    TS acknowledges funding by the European Union (ERC, CountHom, 101077083).

    Views and opinions expressed are however those of the author(s) only and do not necessarily reflect those of the European Union or the European Research Council Executive Agency. Neither the European Union nor the granting authority can be held responsible for them.

    Initial discussions between TS and JM  took place at   Dagstuhl Seminar 25141 \enquote{Categories for Automata and Language Theory}.
    TS would like to thank Bartek Klim,
	Filip Ku\v{c}er\'ak, Cornelius Brand, Radu Curticapean, Arnar Kristjánsson, Ben Moore, Ben Young, Daniel Neuen, and Moritz Seppelt for fruitful conversations.

    \section*{\textsmaller{AI} Statement}

    Claude Code Opus 5 was used for writing, proofreading, and literature review. This applies mostly to the commutative algebra techniques. All \textsmaller{AI}-generated content was revised/rewritten by the authors. We assume full responsibility for the entire material.

	\section{Full version of \texorpdfstring{\cref{sec:which-relations}}{Axiomatisation}}
	\label{app:which-relations}

	In this section, we give a full proof of \cref{thm:axiomatisation}.
	The bulk of the section  is concerned with introducing notions from commutative algebra and elementary algebraic geometry to deal with graph isomorphism relaxations.
	Starting with a commutative cancellative semiring~$S$,
	we construct the Grothendieck ring $\mathbb{Z}S$ in \cref{sec:zs},
	extend scalars in $\mathbb{Q}S = \mathbb{Q} \otimes_{\mathbb{Z}} \mathbb{Z}S$ in \cref{sec:extension-of-scalars-Q},
	and further to $\mathbb{C}S = \mathbb{C} \otimes_{\mathbb{Q}} \mathbb{Q}S$ in \cref{sec:extension-of-scalars-C}.
	In each step, we identify the conditions on equivalence relations $\approx$ on $S$ making them correspond to ideals of the respective semiring algebras.
	
	In \cref{sec:alg-geo-semiring-algebra}, we prove a Nullstellensatz for semiring algebras $\mathbb{C}S$.
    In \cref{sec:radical-ideals}, we characterise semiring congruences $\approx$ corresponding to radical ideals of $\mathbb{C}S$.
	The Nullstellensatz is then instantiated to graphs in \cref{thm:characterisation-semiring-homs}.

	So far, none of the results are specific to the graph semiring $\lG = (\lG, +, \otimes, \boldsymbol{0}, \boldsymbol{1})$.
	This gap is closed in \cref{sec:subgraph-series} by studying subgraph series and proving a characterisation of homomorphism counts alternative to previous work of \textcite{lovasz_dual_2010}.

	In \cref{sec:pieces}, our main theorem \ref{thm:axiomatisation} is derived from the aforementioned ingredients.

	\subsection{Semiring congruences and ideals of Grothendieck ring \texorpdfstring{$\mathbb{Z}S$}{ZS}}
	\label{sec:zs}

	Let $S = (S, +, \cdot, 0, 1)$ be a commutative cancellative semiring.
	We write multiplication as juxtaposition.
	The \emph{Grothendieck ring} $\mathbb{Z}S$
	is defined on the set of equivalence classes of the equivalence relation $\sim$ on $S \times S$ given by
	\[
		(s_1, t_1) \sim (s_2, t_2) \iff s_1 + t_2 = t_1 + s_2.
	\]
	Multiplication and addition are defined via
	\begin{align*}
		(s_1, t_1) + (s_2, t_2) &\coloneqq (s_1 + s_2, t_1 + t_2)\\
		(s_1, t_1) \cdot (s_2, t_2) &\coloneqq (s_1s_2 + t_1 t_2 , s_1 t_2 + t_1 s_2)
	\end{align*}
	for all $s_1, s_2, t_1, t_2 \in S$.
	The map $S \to \mathbb{Z}S$ given by $s \mapsto (s, 0) \eqqcolon [s]$ is a semiring homomorphism.
	Generic elements of $\mathbb{Z}S$ are of the form $[s] - [t] = (s, t)$ for $s, t\in S$.

	Recalling \cref{eq:h1,eq:h2,eq:h3}, we show the following correspondence between ideals of $\mathbb{Z}S$ and congruences on $S$.
	For an ideal $\mathfrak{b} \subseteq \mathbb{Z}S$,
	define an equivalence relation $\approx_{\mathfrak{b}}$ on $S$ via
	\[
		s \approx_{\mathfrak{b}} t \iff [s] - [t] \in \mathfrak{b}.
	\]
	Conversely, given an equivalence relation $\approx$ on $S$,
	define the set
	\[
		\mathfrak{a}_{\approx} \coloneqq \{[s] - [t] \in \mathbb{Z}S \mid s, t \in S, \ s \approx t\}.
	\]
	We shall see that $\mathfrak{a}_{\approx}$ is an ideal and that these constructions are mutually inverse.
	
	\begin{proposition}\label{prop:ideals-to-relations}
		Let $S$ denote a commutative cancellative semiring.
		The maps ${\approx} \mapsto \mathfrak{a}_{\approx }$
		and  $\mathfrak{b} \mapsto {\approx_{\mathfrak{b}}}$
		establish a bijection between the set of ideals $\mathfrak{b}$ of the Grothendieck ring $\mathbb{Z}S$ 
		and the set of equivalence relations $\approx$ on $S$
		such that, for all $s, t, r \in S$, it holds that 
		\begin{align}
			s \approx t &\iff s + r \approx t + r, \label{eq:c1}\\
			s \approx t &\implies s \times r \approx t \times r. \label{eq:c2}
		\end{align}
	\end{proposition}
	We call an equivalence relation $\approx$ on $S$ satisfying~\eqref{eq:c1} and~\eqref{eq:c2} a \emph{cancellative congruence}.
	Towards \cref{prop:ideals-to-relations}, we make the following observations.

	\begin{lemma}\label{lem:cancellation}
		Let $\approx$ be a cancellative congruence on $S$.
		For all $s, s', t, t' \in S$,
		\begin{assertions}
			\item if $s \approx t$ and $s' \approx t'$, then $s + s' \approx t+t'$,
			\item if $s \approx t$ and $s + s' \approx t+t'$, then $s' \approx t'$, and
			\item if $[s] - [t] = [s'] - [t']$ in $\mathbb{Z}S$,
			then $s \approx t$ if, and only if, $s' \approx t'$.
		\end{assertions}
	\end{lemma}
	\begin{proof}
		For the first assertion, by the forward direction of \eqref{eq:c1},
		$s \approx t$ implies $s + s' \approx t + s'$.
		Analogously,
		$s' \approx t'$ implies $t + s' \approx t + t'$.
		By transitivity, $s + s' \approx t + t'$.
For the second assertion,
		by the first assertion,
		 $s \approx t$ and $s + s' \approx t+t'$ imply that
		 $s + t + s' \approx t + s + t'$.
		 Now the term $s + t$ cancels from both sides by the backward implication of \eqref{eq:c1} yielding $s' \approx t'$.
For the final assertion,
		 the equality $[s] - [t] = [s'] - [t']$ rewrites as $s + t' = s' + t$ in $S$. 
			In particular, $s + t' \approx s' + t$.
			By the second assertion, $t' \approx s'$, that is, $s' \approx t'$ by symmetry.
			The converse follows by exchanging the roles of $(s,t)$ and $(s',t')$.
	\end{proof}

	\begin{proof}[Proof of \cref{prop:ideals-to-relations}.]
		We must first show that the map  ${\approx} \mapsto \mathfrak{a}_{\approx }$ is well-defined.
		\begin{claim}
			For every cancellative congruence $\approx$ on $S$,
			the set $\mathfrak{a}_{\approx}$ is an ideal of $\mathbb{Z}S$.
		\end{claim}
		\begin{claimproof}
			It holds that $0 \in \mathfrak{a}_{\approx}$ since $\approx$ is reflexive. 

			If $[s_1] - [t_1], [s_2] - [t_2] \in \mathfrak{a}_{\approx}$,
			then, by the third assertion of \cref{lem:cancellation}, 
			$s_1 \approx t_1$ and $s_2 \approx t_2$.
			By \cref{lem:cancellation},
			$s_1 + s_2 \approx t_1 + t_2$.
			Hence, $[s_1] - [t_1]+ [s_2] - [t_2] = [s_1 + s_2] - [t_1 + t_2] \in \mathfrak{a}_{\approx}$.

			If $[s] - [t] \in \mathfrak{a}_{\approx}$
			and $[a] - [b] \in \mathbb{Z}S$,
			then $s \approx t$ by \cref{lem:cancellation}.
			By \eqref{eq:c2},
			$a \times s \approx a \times t$
			and 
			$b \times s \approx b \times t$.
			By \cref{lem:cancellation},
			\(
			a\times s + b \times t \approx a \times t + b \times s.
			\)
			Hence, it holds that
			$([a] - [b]) ([s] - [t]) = [a \times s + b \times t ] -  [a \times t + b \times s ] \in \mathfrak{a}_\approx$.
		\end{claimproof}

		Next we show that the map  $\mathfrak{b} \mapsto {\approx_{\mathfrak{b}}}$ is well-defined.
		
		\begin{claim}\label{claim:ideal-to-relation}
			For every ideal $\mathfrak{b}$ of $\mathbb{Z}S$,
			the relation $\approx_{\mathfrak{b}}$ is a cancellative congruence.
		\end{claim}
		\begin{claimproof}
			Since $[s] - [s] = 0 \in \mathfrak{b}$ for every $s\in S$,
			the relation $\approx_{\mathfrak{b}}$ is reflexive.
			It is symmetric since $[s] - [t] \in \mathfrak{b}$ iff $[t] - [s] \in \mathfrak{b}$ since $\mathfrak{b}$ is closed under multiplication with $-1 = -[1]$.
			It is transitive since $[s] - [t] \in \mathfrak{b}$ and $[t] - [r] \in \mathfrak{b}$ implies that $[s] - [t] + [t] - [r] \in \mathfrak{b}$.

			It remains to verify \eqref{eq:c1} and \eqref{eq:c2}.
			Let $s \approx_{\mathfrak{b}} t$, i.e.\ $[s] - [t] \in \mathfrak{b}$.
			For $r \in S$,
			it holds that
			$\mathfrak{b} \ni [s] - [t] = [s] + [r] - [r] - [t] = [s  +r ] - [r + t]$.
			Hence, $s +r \approx_{\mathfrak{b}} t + r$.
			The converse implication of \eqref{eq:c1} follows analogously.
For \eqref{eq:c2},
			$[r \times s] - [r \times t] = [r] ([s] - [t]) \in \mathfrak{b}$ since $\mathfrak{b}$ is an ideal.
			Hence, $r \times s \approx_{\mathfrak{b}} r \times t$.
		\end{claimproof}

		It remains to argue that the maps are mutually inverse.

		\begin{claim}
			For every ideal $\mathfrak{b}$ of $\mathbb{Z}S$,
			it holds that
			$\mathfrak{b} = \mathfrak{a}_{\approx_{\mathfrak{b}}}$.
		\end{claim}
		\begin{claimproof}
			Let $[s] - [t] \in \mathfrak{b}$.
			Then $s \approx_{\mathfrak{b}} t$
			and hence $[s] - [t] \in \mathfrak{a}_{\approx_{\mathfrak{b}}}$.
			Conversely, 
			if $[s] - [t] \in \mathfrak{a}_{\approx_{\mathfrak{b}}}$,
			then $s \approx_{\mathfrak{b}} t$ by \cref{lem:cancellation,claim:ideal-to-relation}.
			Hence, $[s] - [t] \in \mathfrak{b}$.
		\end{claimproof}

		\begin{claim}
			For every cancellative congruence $\approx$ on $S$,
			it holds that
			${\approx} = {\approx_{\mathfrak{a}_{\approx}}}$.
		\end{claim}
		\begin{claimproof}
			If $s \approx t$,
			then $[s] - [t] \in \mathfrak{a}_{\approx}$,
			hence, $s \approx_{\mathfrak{a}_{\approx}} t$.
			Conversely,
			if $s \approx_{\mathfrak{a}_{\approx}} t$,
			then $[s] - [t] \in \mathfrak{a}_{\approx}$.
			Hence, by \cref{lem:cancellation},
			$s \approx t$.
		\end{claimproof}

		This concludes the proof.
	\end{proof}

	\subsection{Extension of scalars to \texorpdfstring{$\mathbb{Q}S$}{QS}}
	\label{sec:extension-of-scalars-Q}

	We extend scalars to $\mathbb{Q}S \coloneqq \mathbb{Q} \otimes_{\mathbb{Z}} \mathbb{Z}S$.
	Writing $s - t \coloneqq [s] - [t]$ for the elements of $\mathbb{Z}S$, 
	the elements of $\mathbb{Q}S$ are of the form $\frac{s-t}{n}$ for $s,t \in S$ and integers $n \geq 1$.
	For $s_1, s_2, t_1, t_2 \in S$, $n_1, n_2 \geq 1$, 
	it holds that
	\begin{align*}
		\frac{s_1 - t_1}{n_1}
		= \frac{s_2 - t_2}{n_2}
		\in \mathbb{Q}S
		&\iff n_2\cdot (s_1 - t_1) = n_1\cdot(s_2 - t_2) \in \mathbb{Z}S\\
		&\iff n_2\cdot s_1 + n_1 \cdot t_2 = n_2\cdot t_1 + n_1\cdot s_2 \in S.
	\end{align*}
	Here, for $n \geq 1$ and $s \in S$, 
	we write $n\cdot s$ for the $n$-fold sum $s + \dots + s \in S$.
	Note that there is a semiring homomorphism $\mathbb{N} \to S$ mapping $n \mapsto 1 + \dots + 1$.
	Hence, $n \cdot s$ is a product in~$S$.

	For an ideal $\mathfrak{b} \subseteq \mathbb{Q}S$,
	define an equivalence relation $\approx_{\mathfrak{b}}$ on $S$ via
	\[
		s \approx_{\mathfrak{b}} t \iff s - t = \frac{s-t}{1} \in \mathfrak{b}.
	\]
	Conversely, given an equivalence relation $\approx$ on $S$,
	define the set
	\[
		\mathfrak{a}_{\approx} \coloneqq \left\{\frac{s - t}{n}\in \mathbb{Q}S \ \middle|\ n\geq 1,s, t \in S, \ s \approx t \right\}.
	\]
	We show the following extension of \cref{prop:ideals-to-relations} mirroring \eqref{eq:h4}.
    Recall that an equivalence relation $\approx$ on $S$ is called cancellative if it satisfies~\eqref{eq:c1} and~\eqref{eq:c2}.

	\begin{proposition}\label{prop:ideals-to-relations-rational}
		Let $S$ denote a commutative cancellative semiring.
		The maps ${\approx} \mapsto \mathfrak{a}_{\approx }$
		and  $\mathfrak{b} \mapsto {\approx_{\mathfrak{b}}}$
		establish a bijection between the set of ideals $\mathfrak{b}$ of $\mathbb{Q}S$
		and the set of cancellative congruences $\approx$ on $S$
		such that, for all $s,t\in S$ and $n \geq 1$, it holds that 
		\begin{align}\label{eq:c3}
			n\cdot s \approx n\cdot t \implies s \approx t.
		\end{align}
	\end{proposition}

	We call an equivalence relation $\approx$ on a semiring $S$ satisfying \eqref{eq:c3} \emph{torsion-free}.
	In preparation, we first make the following observation analogous to \cref{lem:cancellation}.

	\begin{lemma}\label{lem:cancellation-rationals}
		Let $\approx$ be a torsion-free cancellative congruence on $S$.
			If $\frac{s - t}{p} = \frac{s' - t'}{p'}$ in $\mathbb{Q}S$,
			then $s \approx t$ if, and only if, $s' \approx t'$.
	\end{lemma}
	\begin{proof}
		If $s \approx t$, then $pp' \cdot s \approx pp' \cdot t$ by \cref{lem:cancellation}.
			Furthermore, by assumption, $p' \cdot s + p \cdot t' = p' \cdot t + p \cdot s'$.
			Hence, $p p' \cdot s + p^2 \cdot t' = pp' \cdot t + p^2 \cdot s'$ by \cref{lem:cancellation}.
			In particular,
			 $p p' \cdot s + p^2 \cdot t' \approx pp' \cdot t + p^2 \cdot s'$.
			By \cref{lem:cancellation},
			$p^2 \cdot t' \approx p^2 \cdot s'$.
			By \eqref{eq:c3}, $t' \approx s'$, that is, $s' \approx t'$ by symmetry.
			The converse follows by exchanging the roles of $(s,t,p)$ and $(s',t',p')$.
	\end{proof}

	\begin{proof}[Proof of \cref{prop:ideals-to-relations-rational}]
		We must first show that the map  ${\approx} \mapsto \mathfrak{a}_{\approx }$ is well-defined.
		\begin{claim}
			For every torsion-free cancellative congruence $\approx$ on $S$,
			the set $\mathfrak{a}_{\approx}$ is an ideal of $\mathbb{Q}S$.
		\end{claim}
		\begin{claimproof}
			It holds that $0 \in \mathfrak{a}_{\approx}$ since $\approx$ is reflexive. 

			If $\frac{s_1 - t_1}{n_1},\frac{s_2 - t_2}{n_2} \in \mathfrak{a}_{\approx}$,
			then, by \cref{lem:cancellation-rationals}, 
			$s_1 \approx t_1$ and $s_2 \approx t_2$.
			By \cref{lem:cancellation},
			$n_2 \cdot s_1 \approx n_2 \cdot t_1$
			and $n_1 \cdot s_2 \approx n_1 \cdot t_2$.
			By \cref{lem:cancellation},
			$n_2 \cdot s_1 + n_1 \cdot s_2 \approx n_2 \cdot t_1 + n_1\cdot t_2$.
			Hence, 
			\[
			\frac{s_1 - t_1}{n_1} + \frac{s_2 - t_2}{n_2}
			= \frac{n_2 \cdot s_1 - n_2 \cdot t_1 + n_1 \cdot s_2 - n_1 \cdot t_2}{n_1n_2}
			\in \mathfrak{a}_{\approx}.
			\]

			If $\frac{s - t}{n} \in \mathfrak{a}_{\approx}$
			and $\frac{a - b}{m} \in \mathbb{Q}S$,
			then $s \approx t$ by \cref{lem:cancellation-rationals}.
			By \eqref{eq:c2},
			$a \cdot s \approx a \cdot t$
			and 
			$b \cdot s \approx b \cdot t$.
			By \cref{lem:cancellation},
			\(
			a\cdot s + b \cdot t \approx a \cdot t + b \cdot s.
			\)
			Hence, it holds that
			\[
			\frac{a - b}{m} \cdot 
			\frac{s - t}{n} =
			\frac{(a-b)\cdot (s - t)}{nm}
			=
			\frac{a \cdot s + b \cdot t - a \cdot t - b \cdot s}{nm}
			\in \mathfrak{a}_{\approx}. \qedhere
			\]\end{claimproof}

		Next we show that the map  $\mathfrak{b} \mapsto {\approx_{\mathfrak{b}}}$ is well-defined.
		
		\begin{claim}\label{claim:ideal-to-relation-rational}
			For every ideal $\mathfrak{b}$ of $\mathbb{Q}S$,
			the relation $\approx_{\mathfrak{b}}$ is a torsion-free cancellative congruence.
		\end{claim}
		\begin{claimproof}
			By the proof of \cref{claim:ideal-to-relation},
			it remains to verify that $\approx_{\mathfrak{b}}$ is torsion-free.
			Let $n \cdot s \approx_{\mathfrak{b}} n \cdot t$
			for $s, t\in S$ and $n \geq 1$.
			By definition, 
			$n \cdot (s - t) = n \cdot s - n \cdot t \in \mathfrak{b}$.
			Since $\mathfrak{b}$ is an ideal,
			$s - t = \frac{n \cdot (s-t)}n = \frac1n n (s-t) \in \mathfrak{b}$.
		\end{claimproof}

		It remains to argue that the maps are mutually inverse.

		\begin{claim}
			For every ideal $\mathfrak{b}$ of $\mathbb{Q}S$,
			it holds that
			$\mathfrak{b} = \mathfrak{a}_{\approx_{\mathfrak{b}}}$.
		\end{claim}
		\begin{claimproof}
			Let $\frac{s - t}{n} \in \mathfrak{b}$.
			Then $s \approx_{\mathfrak{b}} t$
			and hence $\frac{s - t}n \in \mathfrak{a}_{\approx_{\mathfrak{b}}}$.
			Conversely, 
			if $\frac{s - t}{n} \in \mathfrak{a}_{\approx_{\mathfrak{b}}}$,
			then $\frac{s - t}{n} = \frac{s' -t'}{n'}$ for some $s' \approx_{\mathfrak{b}} t'$ and $n' \geq 1$.
			By \cref{lem:cancellation-rationals,claim:ideal-to-relation-rational},
			 $s \approx_{\mathfrak{b}} t$.
			By definition, $ s - t \in \mathfrak{b}$.
			By multiplication with $\frac1n \in \mathbb{Q}S$,
			it follows that
			$\frac{s-t}{n} \in \mathfrak{b}$.
		\end{claimproof}

		\begin{claim}
			For every torsion-free cancellative congruence $\approx$ on $S$,
			it holds that
			${\approx} = {\approx_{\mathfrak{a}_{\approx}}}$.
		\end{claim}
		\begin{claimproof}
			If $s \approx t$,
			then $s - t \in \mathfrak{a}_{\approx}$,
			hence, $s \approx_{\mathfrak{a}_{\approx}} t$.
			Conversely,
			if $s \approx_{\mathfrak{a}_{\approx}} t$,
			then $s - t \in \mathfrak{a}_{\approx}$.
			By definition, $s - t = \frac{s' - t'}{n'}$ for some $s' \approx t'$ and $n'\geq 1$.
			By \cref{lem:cancellation-rationals},
			$s \approx t$.
		\end{claimproof}

		This concludes the proof.
	\end{proof}

	\subsection{Extension of scalars to \texorpdfstring{$\mathbb{C}S$}{CS}}
	\label{sec:extension-of-scalars-C}

	Hilbert's Nullstellensatz requires an algebraically closed field.
	Hence, we extend scalars further to $\mathbb{C}S \coloneqq \mathbb{C} \otimes_{\mathbb{Q}} \mathbb{Q}S$.
	Generic elements of $\mathbb{C}S$ are finite sums $\sum \alpha_i \frac{s_i - t_i}{n_i}$
	where $\alpha_i \in \mathbb{C}$,
	$s_i, t_i \in S$, and $n_i \geq 1$.
	Furthermore, the $\alpha_i$ can be assumed to be linearly independent as elements of the $\mathbb{Q}$-vector space~$\mathbb{C}$.

	For an ideal $\mathfrak{b} \subseteq \mathbb{C}S$,
	define an equivalence relation $\approx_{\mathfrak{b}}$ on $S$ via
	\[
		s \approx_{\mathfrak{b}} t \defiff s - t \in \mathfrak{b}.
	\]
	Conversely, given an equivalence relation $\approx$ on $S$,
	define the set
	\begin{equation}\label{eq:a-approx-complex}
		\mathfrak{a}_{\approx} \coloneqq \left\{\sum \alpha_i \frac{s_i - t_i}{n_i} \in \mathbb{C}S \ \middle|\ \alpha_i \in \mathbb{C}, n_i\geq 1, s_i, t_i
		 \in S, \ s_i \approx t_i \right\}.
	\end{equation}
	
	\begin{proposition}\label{prop:ideals-to-relations-complex}
		Let $S$ denote a commutative cancellative semiring.
		\begin{assertions}
			\item\label{it:complex-ideal-to-relation}
			For every ideal $\mathfrak{b}$ of $\mathbb{C}S$,
			the relation $\approx_{\mathfrak{b}}$ is a torsion-free cancellative congruence on $S$.
			\item\label{it:complex-relation-to-ideal}
			For every torsion-free cancellative congruence $\approx$ on $S$,
			the set $\mathfrak{a}_{\approx}$ is an ideal of $\mathbb{C}S$
			and ${\approx_{\mathfrak{a}_{\approx}}} = {\approx}$,
			i.e.\ for all $s,t \in S$, it holds that $s \approx t$ if, and only if, $s - t \in \mathfrak{a}_{\approx}$.
		\end{assertions}
		In particular, ${\approx} \mapsto \mathfrak{a}_{\approx}$ is injective
		and establishes a bijection between the torsion-free cancellative congruences on $S$
		and the ideals of $\mathbb{C}S$ extended from $\mathbb{Q}S$.
	\end{proposition}
	\begin{proof}
		Write $\mathfrak{a}^{\mathbb{Q}}_{\approx} \subseteq \mathbb{Q}S$
		for the ideal corresponding to $\approx$ under \cref{prop:ideals-to-relations-rational}.
		Comparing the definitions in \cref{sec:extension-of-scalars-Q,sec:extension-of-scalars-C},
		the set $\mathfrak{a}_{\approx}$ consists of the finite sums $\sum_i \alpha_i x_i$
		with $\alpha_i \in \mathbb{C}$ and $x_i \in \mathfrak{a}^{\mathbb{Q}}_{\approx}$.
		That is,
		\begin{equation}\label{eq:extended-ideal}
			\mathfrak{a}_{\approx}
			= \mathbb{C} \otimes_{\mathbb{Q}} \mathfrak{a}^{\mathbb{Q}}_{\approx}
			\subseteq \mathbb{C} \otimes_{\mathbb{Q}} \mathbb{Q}S = \mathbb{C}S
		\end{equation}
		is the ideal of $\mathbb{C}S$ extended from $\mathfrak{a}^{\mathbb{Q}}_{\approx}$;
		in particular, it is indeed an ideal.

		\Cref{it:complex-ideal-to-relation} follows as in the proof of \cref{prop:ideals-to-relations-rational}.

		For \cref{it:complex-relation-to-ideal},
		recall that the field extension $\mathbb{Q} \to \mathbb{C}$ is \emph{faithfully flat}.
		Hence, for every $\mathbb{Q}$-algebra $A$ and every ideal $\mathfrak{a} \subseteq A$,
		the quotient $A / \mathfrak{a}$ embeds into
		$\mathbb{C} \otimes_{\mathbb{Q}} (A/\mathfrak{a}) = (\mathbb{C} \otimes_{\mathbb{Q}} A)/(\mathbb{C} \otimes_{\mathbb{Q}} \mathfrak{a})$
		\cite[\href{https://stacks.math.columbia.edu/tag/05CK}{Tag 05CK}]{STACKS},
		where the equality uses that $\mathbb{C}$ is
		\emph{flat} \cite[\href{https://stacks.math.columbia.edu/tag/00HB}{Tag 00HB}]{STACKS} over $\mathbb{Q}$.
		In other words, extending scalars does not enlarge $\mathfrak{a}$ within $A$:
		\begin{equation}\label{eq:extension-conservative}
			a \in \mathfrak{a} \iff a \in \mathbb{C} \otimes_{\mathbb{Q}} \mathfrak{a},
			\qquad \text{for all } a \in A.
		\end{equation}
		Applying this to $A \coloneqq \mathbb{Q}S$ and $\mathfrak{a} \coloneqq \mathfrak{a}^{\mathbb{Q}}_{\approx}$,
		we obtain, for all $s,t\in S$,
		\[
			s - t \in \mathfrak{a}_{\approx}
			\overset{\eqref{eq:extended-ideal}}{\iff} s - t \in \mathbb{C} \otimes_{\mathbb{Q}} \mathfrak{a}^{\mathbb{Q}}_{\approx}
			\overset{\eqref{eq:extension-conservative}}{\iff} s - t \in \mathfrak{a}^{\mathbb{Q}}_{\approx}
			\iff s \approx t,
		\]
		where the final equivalence is \cref{prop:ideals-to-relations-rational}.

		The injectivity of ${\approx} \mapsto \mathfrak{a}_{\approx}$ follows from
		\cref{it:complex-relation-to-ideal}.
		By \eqref{eq:extended-ideal}, its image consists of extended ideals.
		Conversely, if $\mathfrak{b} = \mathbb{C} \otimes_{\mathbb{Q}} \mathfrak{b}'$
		for an ideal $\mathfrak{b}'$ of $\mathbb{Q}S$,
		then ${\approx_{\mathfrak{b}}} = {\approx_{\mathfrak{b}'}}$ by \eqref{eq:extension-conservative},
		and hence $\mathfrak{a}_{\approx_{\mathfrak{b}}} = \mathbb{C} \otimes_{\mathbb{Q}} \mathfrak{b}' = \mathfrak{b}$
		by \eqref{eq:extended-ideal} and \cref{prop:ideals-to-relations-rational}.
	\end{proof}

	\begin{remark}\label{rem:no-bijection-complex}
		Unlike \cref{prop:ideals-to-relations,prop:ideals-to-relations-rational},
		\cref{prop:ideals-to-relations-complex} cannot be strengthened to a bijection with \emph{all} ideals of $\mathbb{C}S$.
		The reason is that the elements $s - t$ for $s, t \in S$ span $\mathbb{C}S$ over $\mathbb{C}$ without exhausting it,
		so that an ideal of $\mathbb{C}S$ is not determined by the differences it contains.
		For example, let $S \coloneqq \mathbb{N}[u]$, so that $\mathbb{Z}S = \mathbb{Z}[u]$ and $\mathbb{C}S = \mathbb{C}[u]$.
		A polynomial in $\mathbb{Z}[u]$ vanishes at $\sqrt{-1}$ if, and only if, it is divisible by $u^2 + 1$.
		Hence, the distinct principal ideals $(u - \sqrt{-1})$ and $(u^2+1)$ of $\mathbb{C}[u]$
		induce the same relation on $S$,
		namely the one corresponding to the extended ideal $(u^2+1)$.

		Nevertheless, among all ideals inducing a given congruence $\approx$
		there is a least one, namely $\mathfrak{a}_{\approx}$:
		if $\mathfrak{b}$ is an ideal of $\mathbb{C}S$ with ${\approx_{\mathfrak{b}}} = {\approx}$,
		then $s_i - t_i \in \mathfrak{b}$ for all $s_i \approx t_i$,
		and $\mathfrak{b}$ is closed under $\mathbb{C}$-linear combinations,
		whence $\mathfrak{a}_{\approx} \subseteq \mathfrak{b}$.
	\end{remark}

	This concludes the preparations for the proof of \cref{lem:c-ideals},
	which instantiates \cref{prop:ideals-to-relations-complex} to graph isomorphism relaxations.

	\begin{lemma}[restate=lemCideals, label=lem:c-ideals]
		Let $\approx$ be a graph isomorphism relaxation.
		The relation $\approx$ satisfies \cref{eq:h1,eq:h2,eq:h3,eq:h4} if, and only if, 
		there exists an ideal $\mathfrak{a}$ of $\mathbb{C}\lG$ such that, for all graphs $G, H \in \lG$,
		\[
			G \approx H \iff G - H \in \mathfrak{a}.
		\]
	\end{lemma}
	\begin{proof}
		Recall from \cref{lem:graphs-form-semiring} that $\lG$,
		equipped with disjoint union as addition and the categorical product as multiplication,
		is a commutative cancellative semiring.
		Being a graph isomorphism relaxation,
		$\approx$ is an equivalence relation on $\lG$.

		We first observe that $\approx$ satisfies \cref{eq:h1,eq:h2,eq:h3,eq:h4}
		if, and only if, it is a torsion-free cancellative congruence on $\lG$.
		Indeed, \cref{eq:h1} is \eqref{eq:c2},
		while the forward and the backward direction of \eqref{eq:c1}
		are \cref{eq:h2,eq:h3}, respectively.
		Furthermore, \cref{eq:h4} asserts that $\approx$ is torsion-free,
		since $n \cdot G$ is the $n$-fold disjoint union of $G$ with itself.

		Now, if $\approx$ satisfies \cref{eq:h1,eq:h2,eq:h3,eq:h4},
		then, by \cref{it:complex-relation-to-ideal} of \cref{prop:ideals-to-relations-complex},
		the set $\mathfrak{a} \coloneqq \mathfrak{a}_{\approx}$ is an ideal of $\mathbb{C}\lG$
		such that $G \approx H$ if, and only if, $G - H \in \mathfrak{a}$,
		for all graphs $G, H \in \lG$.

		Conversely, if $\mathfrak{b}$ is an ideal of $\mathbb{C}\lG$
		such that $G \approx H$ if, and only if, $G - H \in \mathfrak{b}$,
		then ${\approx} = {\approx_{\mathfrak{b}}}$.
		By \cref{it:complex-ideal-to-relation} of \cref{prop:ideals-to-relations-complex},
		this is a torsion-free cancellative congruence on $\lG$,
		so that $\approx$ satisfies \cref{eq:h1,eq:h2,eq:h3,eq:h4}.
	\end{proof}

	\subsection{Algebraic geometry of semiring algebras}
	\label{sec:alg-geo-semiring-algebra}
	
	In order to prove \cref{thm:semiring-hom-congruence}, we adapt Hilbert's Nullstellensatz to semiring algebras in \cref{thm:nullstellensatz}.

	Let $S$ be a commutative cancellative semiring.
	Write $\Hom(S, \mathbb{C})$ for the set of \emph{semiring homomorphisms} $S \to \mathbb{C}$, i.e., maps $\phi \colon S\to \mathbb{C}$ such that $\phi(1) = 1$ and, for all $s, t\in S$,
	\[
		\phi(s + t) = \phi(s) + \phi(t), \qquad \phi(st) = \phi(s) \phi(t).
	\]

    We require the following lemma whose proof is routine and therefore omitted.

	\begin{lemma}\label{lem:unique-extension}
		For every semiring homomorphism $\phi \in \Hom(S, \mathbb{C})$,
		there exists a unique $\mathbb{C}$-algebra homomorphism $\widehat{\phi} \colon \mathbb{C}S\to \mathbb{C}$ such that $\widehat{\phi}(s) = \phi(s)$ for all $s\in S$.
		It is given by
		\begin{equation}\label{eq:phi-hat}
			\widehat{\phi}\left(\sum_i \alpha_i \frac{s_i - t_i}{n_i}\right) = \sum_i \frac{\alpha_i}{n_i}\left(\phi(s_i) - \phi(t_i)\right).
		\end{equation}
	\end{lemma}

	In light of \cref{lem:unique-extension},
	we henceforth write $\phi$ for the extension $\widehat\phi \colon \C S \to \C$ of a semiring homomorphism $\phi \in \Hom(S, \mathbb{C})$.

	\begin{definition}
		For an ideal $\mathfrak{a} \subseteq \mathbb{C}S$ and a subset $\Phi \subseteq \Hom(S, \mathbb{C})$, let
		\begin{align*}
			V(\mathfrak{a}) & \coloneqq \left\{\phi \in \Hom(S, \mathbb{C}) \mid \forall p \in \mathfrak{a}.\ \phi(p) = 0 \right\}, \\
			I(\Phi) & \coloneqq \left\{p \in \mathbb{C}S \mid \forall \phi \in \Phi.\ \phi(p) = 0 \right\}.
		\end{align*}
	\end{definition}

	We record that vanishing ideals of rational-valued semiring homomorphisms
	are extended from $\mathbb{Q}S$.

	\begin{lemma}\label{lem:vanishing-ideal-extended}
		Let $S$ denote a commutative cancellative semiring
		and let $\Phi \subseteq \Hom(S, \mathbb{Q})$.
		Then it holds that
		$I(\Phi) = \mathbb{C} \otimes_{\mathbb{Q}} \left(I(\Phi) \cap \mathbb{Q}S\right)$.
	\end{lemma}
	\begin{proof}
		By \cref{sec:extension-of-scalars-C},
		every $p \in \mathbb{C}S$ can be written as $p = \sum_i \alpha_i x_i$
		with $x_i \in \mathbb{Q}S$
		and $\alpha_i \in \mathbb{C}$ linearly independent over $\mathbb{Q}$.
		Let $\phi \in \Phi$.
		Since $\phi$ takes rational values on $S$,
		it holds that $\phi(x) \in \mathbb{Q}$ for every $x \in \mathbb{Q}S$ by \eqref{eq:phi-hat}.
		Hence, $\phi(p) = \sum_i \alpha_i \phi(x_i) = 0$
		if, and only if, $\phi(x_i) = 0$ for all $i$.
		Therefore, $p \in I(\Phi)$ if, and only if, $x_i \in I(\Phi) \cap \mathbb{Q}S$ for all $i$,
		which is the assertion.
	\end{proof}

	For an ideal $\mathfrak{a}$ in a ring $R$, write $\sqrt{\mathfrak{a}} \coloneqq \{r \in R \mid \exists n \geq 1.\ r^n \in \mathfrak{a}\}$ for the \emph{radical} of $\mathfrak{a}$.
	An ideal $\mathfrak{a}$ is called \emph{radical} if $\mathfrak{a} = \sqrt{\mathfrak{a}}$.
	The following Nullstellensatz for semiring algebras is the main result of this section.

	\begin{theorem}[Hilbert's Nullstellensatz for Semiring Algebras]
		\label{thm:hilbert-nullstellensatz-semiring}\label{thm:nullstellensatz}
		Let $S$ be a countable commutative cancellative semiring and $\mathfrak{a}$ an ideal of $\mathbb{C}S$.
		Then
		\[
			I(V(\mathfrak{a})) = \sqrt{\mathfrak{a}}.
		\]
	\end{theorem}

	We deduce \cref{thm:nullstellensatz} from the following \cref{thm:zeidler},
	a version of Hilbert's Nullstellensatz for polynomial algebras in infinitely many variables.
	Note that the difference between \cref{thm:nullstellensatz,thm:zeidler} is that in the definition of $V(\mathfrak{a})$ only semiring homomorphisms $S \to \mathbb{C}$ are considered while $V_A(\mathfrak{a})$ contains arbitrary maps $S \to \mathbb{C}$.
    \Cref{thm:zeidler} is apparently due to \textcite{Lang_1952}; see also \cite[\href{https://stacks.math.columbia.edu/tag/00FU}{Tag 00FU}]{STACKS} and \cite{zeidler_infinite_2026}.

	\begin{theorem}[\textcite{Lang_1952}]\label{thm:zeidler}
		Let $F$ be a non-zero ring,
		$\Lambda$ a non-empty index set,
		and write $A$ for the polynomial algebra $F[t_\lambda \mid \lambda \in \Lambda]$.
		Suppose that $F$ is algebraically closed and that $|\Lambda| < |F|$.
		Then, for every ideal $\mathfrak{b} \subseteq A$, it holds that
		\[
			I_A (V_A(\mathfrak{b})) = \sqrt{\mathfrak{b}}.
		\]
		Here, every $\chi \in F^{\Lambda}$ gives rise to the evaluation homomorphism
		$\ev_{\chi} \colon A \to F$,
		i.e.\ the unique $F$-algebra homomorphism with $\ev_{\chi}(t_\lambda) = \chi(\lambda)$ for all $\lambda \in \Lambda$,
		and, for an ideal $\mathfrak{b} \subseteq A$ and a set $X \subseteq F^{\Lambda}$,
		\begin{align*}
			V_A(\mathfrak{b}) & \coloneqq \{\chi \in F^{\Lambda} \mid \forall r \in \mathfrak{b}.\ \ev_{\chi}(r) = 0 \},\\
			I_A(X) & \coloneqq \{r\in A \mid \forall \chi \in X.\ \ev_{\chi}(r) = 0\}.
		\end{align*}
	\end{theorem}

	\begin{proof}[Proof of \cref{thm:nullstellensatz}]
		Since $S$ is countable and non-empty,
		there exists a surjection $\mathbb{N} \to S$, $i \mapsto s_i$.
		Write $A \coloneqq \mathbb{C}[t_i \mid i \in \mathbb{N}]$
		for the polynomial algebra over $\mathbb{C}$ in countably many variables
		and let
		\begin{align*}
			\pi \colon A &\to \mathbb{C}S \\
						t_i &\mapsto s_i
		\end{align*}
		denote the unique $\mathbb{C}$-algebra homomorphism mapping $t_i$ to the image of $s_i$
		under the canonical semiring homomorphism $S \to \mathbb{C}S$.
		Its kernel $\ker \pi$ encodes the relations holding in $S$.

		\begin{claim}\label{cl:pi-surjective}
			The homomorphism $\pi$ is surjective.
		\end{claim}
		\begin{claimproof}
			By \cref{sec:extension-of-scalars-C},
			every element of $\mathbb{C}S$ is a finite sum $\sum_j \alpha_j \frac{u_j - v_j}{n_j}$
			with $\alpha_j \in \mathbb{C}$, $u_j, v_j \in S$, and $n_j \geq 1$.
			Since $i \mapsto s_i$ is surjective,
			there are $k_j, \ell_j \in \mathbb{N}$ with $u_j = s_{k_j}$ and $v_j = s_{\ell_j}$.
			As $\frac{\alpha_j}{n_j} \in \mathbb{C}$ and $\pi$ is $\mathbb{C}$-linear,
			\[
				\pi\left(\sum_j \frac{\alpha_j}{n_j}\left(t_{k_j} - t_{\ell_j}\right)\right)
				= \sum_j \frac{\alpha_j}{n_j}\left(u_j - v_j\right)
				= \sum_j \alpha_j \frac{u_j - v_j}{n_j}. \qedhere
			\]
		\end{claimproof}

		Write $\widetilde{\mathfrak{a}} \coloneqq \pi^{-1}(\mathfrak{a})$ for the preimage of $\mathfrak{a}$ in $A$.
		It is an ideal of $A$ containing $\ker \pi$.
		Since $\mathbb{C}$ is a non-zero algebraically closed field
		and $|\mathbb{N}| = \aleph_0 < 2^{\aleph_0} = |\mathbb{C}|$,
		\cref{thm:zeidler} applies to $\widetilde{\mathfrak{a}}$ and yields
		\begin{equation}\label{eq:zeidler-applied}
			I_A(V_A(\widetilde{\mathfrak{a}})) = \sqrt{\widetilde{\mathfrak{a}}}.
		\end{equation}

		By \cref{lem:unique-extension},
		every $\phi \in \Hom(S, \mathbb{C})$ is identified with its unique extension to a $\mathbb{C}$-algebra homomorphism $\mathbb{C}S \to \mathbb{C}$.
		Write $\widetilde{\phi} \in \mathbb{C}^{\mathbb{N}}$ for the map $i \mapsto \phi(s_i)$.
		Then
		\begin{equation}\label{eq:ev-factors}
			\ev_{\widetilde\phi} = \phi \circ \pi,
		\end{equation}
		since both sides are $\mathbb{C}$-algebra homomorphisms $A \to \mathbb{C}$
		agreeing on the generators $t_i$, for $i \in \mathbb{N}$.

		\begin{claim}\label{cl:zero-sets}
			The map $\phi \mapsto \widetilde\phi$ is injective on $\Hom(S, \mathbb{C})$
			and restricts to a bijection between $V(\mathfrak{a})$ and $V_A(\widetilde{\mathfrak{a}})$,
			i.e.\ the following square commutes and its horizontal maps are as indicated:
			\[
			\begin{tikzcd}
				V(\mathfrak{a}) \arrow[d, hook] \arrow[r, "\cong"] & V_A(\widetilde{\mathfrak{a}}) \arrow[d, hook] \\
				{\Hom(S, \mathbb{C})} \arrow[r, hook, "\phi \mapsto \widetilde\phi"']   & \mathbb{C}^{\mathbb{N}}
			\end{tikzcd}
			\]
		\end{claim}
		\begin{claimproof}
			Injectivity is immediate.

			The map is well-defined on $V(\mathfrak{a})$:
			If $\phi \in V(\mathfrak{a})$ and $r \in \widetilde{\mathfrak{a}}$,
			then $\pi(r) \in \mathfrak{a}$ and hence
			$\ev_{\widetilde\phi}(r) = \phi(\pi(r)) = 0$ by \eqref{eq:ev-factors}.
			Thus $\widetilde\phi \in V_A(\widetilde{\mathfrak{a}})$.

			It remains to show that every $\chi \in V_A(\widetilde{\mathfrak{a}})$
			is of the form $\widetilde\phi$ for some $\phi \in V(\mathfrak{a})$.
			Since $0 \in \mathfrak{a}$, it holds that $\ker \pi \subseteq \widetilde{\mathfrak{a}}$,
			so $\ev_{\chi}$ vanishes on $\ker\pi$.
			By \cref{cl:pi-surjective},
			$\pi$ induces an isomorphism $A / \ker \pi \cong \mathbb{C}S$.
			Hence, there is a unique $\mathbb{C}$-algebra homomorphism
			$\theta \colon \mathbb{C}S \to \mathbb{C}$ such that $\theta \circ \pi = \ev_{\chi}$.
			Let $\phi \colon S \to \mathbb{C}$ denote the composition of $\theta$
			with the canonical semiring homomorphism $S \to \mathbb{C}S$.
			As a composition of semiring homomorphisms,
			$\phi$ is a semiring homomorphism,
			and $\phi(1) = \theta(1) = 1$,
			i.e.\ $\phi \in \Hom(S, \mathbb{C})$.
			By the uniqueness assertion in \cref{lem:unique-extension},
			$\theta$ is the extension of $\phi$ to $\mathbb{C}S$.
			For all $i \in \mathbb{N}$,
			\[
				\widetilde\phi(i) = \phi(s_i) = \theta(\pi(t_i)) = \ev_{\chi}(t_i) = \chi(i),
			\]
			so $\widetilde\phi = \chi$.
			Finally, $\phi \in V(\mathfrak{a})$:
			Given $p \in \mathfrak{a}$,
			by \cref{cl:pi-surjective} there is $r \in A$ with $\pi(r) = p$.
			Then $r \in \widetilde{\mathfrak{a}}$ and thus
			$\phi(p) = \theta(\pi(r)) = \ev_{\chi}(r) = 0$.
		\end{claimproof}

		\begin{claim}\label{cl:vanishing-ideals}
			It holds that $\pi^{-1}\left(I(V(\mathfrak{a}))\right) = I_A(V_A(\widetilde{\mathfrak{a}}))$.
		\end{claim}
		\begin{claimproof}
			Let $r \in A$.
			Using \cref{cl:zero-sets} in the second and \eqref{eq:ev-factors} in the third step,
			\begin{align*}
				r \in I_A(V_A(\widetilde{\mathfrak{a}}))
				&\iff \forall \chi \in V_A(\widetilde{\mathfrak{a}}).\ \ev_{\chi}(r) = 0\\
				&\iff \forall \phi \in V(\mathfrak{a}).\ \ev_{\widetilde\phi}(r) = 0\\
				&\iff \forall \phi \in V(\mathfrak{a}).\ \phi(\pi(r)) = 0\\
				&\iff \pi(r) \in I(V(\mathfrak{a}))\\
				&\iff r \in \pi^{-1}\left(I(V(\mathfrak{a}))\right). \qedhere
			\end{align*}
		\end{claimproof}

		\begin{claim}\label{cl:radical-preimage}
			It holds that $\pi\left(\sqrt{\widetilde{\mathfrak{a}}}\right) = \sqrt{\mathfrak{a}}$.
		\end{claim}
		\begin{claimproof}
			If $r \in \sqrt{\widetilde{\mathfrak{a}}}$,
			say $r^n \in \widetilde{\mathfrak{a}}$ for some $n \geq 1$,
			then $\pi(r)^n = \pi(r^n) \in \mathfrak{a}$,
			i.e.\ $\pi(r) \in \sqrt{\mathfrak{a}}$.
			Conversely, let $p \in \sqrt{\mathfrak{a}}$, say $p^n \in \mathfrak{a}$.
			By \cref{cl:pi-surjective}, there is $r \in A$ with $\pi(r) = p$.
			Then $\pi(r^n) = p^n \in \mathfrak{a}$, i.e.\ $r^n \in \widetilde{\mathfrak{a}}$.
			Hence $p = \pi(r) \in \pi\left(\sqrt{\widetilde{\mathfrak{a}}}\right)$.
		\end{claimproof}

		Combining these claims yields
		\[
			I(V(\mathfrak{a}))
			\overset{\labelcref{cl:pi-surjective}}{=} \pi\left(\pi^{-1}\left(I(V(\mathfrak{a}))\right)\right)
			\overset{\labelcref{cl:vanishing-ideals}}{=} \pi\left(I_A(V_A(\widetilde{\mathfrak{a}}))\right)
			\overset{\eqref{eq:zeidler-applied}}{=} \pi\left(\sqrt{\widetilde{\mathfrak{a}}}\right)
			\overset{\labelcref{cl:radical-preimage}}{=} \sqrt{\mathfrak{a}}. \qedhere
		\]
	\end{proof}

	\subsection{Radical ideals}
	\label{sec:radical-ideals}

	In order to apply \cref{thm:nullstellensatz},
	it remains to understand the radical $\sqrt{\mathfrak{a}}$ of an ideal $\mathfrak{a} \subseteq \mathbb{C}S$
	in terms of the cancellative congruence corresponding to $\mathfrak{a}$ under \cref{prop:ideals-to-relations-complex}.
	While a generic element of $\mathbb{C}S$ is a finite sum $\sum_i \alpha_i \frac{s_i - t_i}{n_i}$,
	this congruence only records which of the elements $s - t$ for $s, t\in S$ lie in $\mathfrak{a}$.
	Fortunately, these elements are all we need to consider.
	By the Binomial Theorem, it holds in $\mathbb{C}S$ that
	\begin{equation}\label{eq:binomial-split}
	(s-t)^n 
	= \sum_{i=0}^n \binom{n}{i} s^{n-i} (-t)^i 
	= \underbrace{\sum_{\substack{i=0 \\ i \text{ even}}}^n \binom{n}{i} s^{n-i} t^i}_{\eqqcolon (s, t)^n_+}
	- \underbrace{\sum_{\substack{i=1 \\ i \text{ odd}}}^n \binom{n}{i} s^{n-i} t^i}_{\eqqcolon (s,t)^n_-},
	\end{equation}
	where crucially $(s, t)^n_+$ and $(s,t)^n_-$ are again elements of $S$.
	That is, the $n$-th power of $s - t$ is again the difference of two elements of $S$.
	We may thus consider the following condition on a relation $\approx$ on $S$,
	asserted for all $s,t \in S$ and all $n \geq 1$:
	\begin{equation}\label{eq:radical}
		(s, t)^n_+ \approx (s,t)^n_- \implies s \approx t.
	\end{equation}

	\begin{lemma}\label{lem:radical-differences}
		Let $\approx$ be a torsion-free cancellative congruence on $S$
		satisfying \eqref{eq:radical},
		and let $\mathfrak{a}_{\approx} \subseteq \mathbb{C}S$ denote the corresponding ideal
		given by \cref{eq:a-approx-complex}.
		Then, for all $s, t \in S$,
		\[
			s - t \in \sqrt{\mathfrak{a}_{\approx}} \iff s \approx t.
		\]
	\end{lemma}
	\begin{proof}
		By \cref{prop:ideals-to-relations-complex},
		it holds that ${\approx} = {\approx_{\mathfrak{a}_{\approx}}}$,
		i.e.\ $s \approx t$ if, and only if, $s - t \in \mathfrak{a}_{\approx}$, for all $s, t \in S$.
		In the backward direction,
		if $s \approx t$, then $s - t \in \mathfrak{a}_{\approx} \subseteq \sqrt{\mathfrak{a}_{\approx}}$.
		In the forward direction,
		suppose that $s - t \in \sqrt{\mathfrak{a}_{\approx}}$,
		i.e.\ $(s-t)^n \in \mathfrak{a}_{\approx}$ for some $n \geq 1$.
		By \eqref{eq:binomial-split},
		\[
			(s, t)^n_+ - (s,t)^n_- = (s-t)^n \in \mathfrak{a}_{\approx}.
		\]
		As $(s, t)^n_+$ and $(s,t)^n_-$ are elements of $S$,
		this means that $(s, t)^n_+ \approx (s,t)^n_-$.
		By \eqref{eq:radical}, it follows that $s \approx t$.
	\end{proof}

	\Cref{lem:radical-differences} is all we shall need,
		but in fact more is true:

	\begin{lemma}\label{lem:radical-ideal}
		Let $\approx$ be a torsion-free cancellative congruence on $S$.
		The relation $\approx$ satisfies \eqref{eq:radical}
		if, and only if, the ideal $\mathfrak{a}_{\approx}$ of $\mathbb{C}S$ is radical.
	\end{lemma}
	\begin{proof}
		Suppose that $\mathfrak{a}_{\approx}$ is radical
		and that $(s, t)^n_+ \approx (s,t)^n_-$ for some $s, t \in S$ and $n \geq 1$.
		By \eqref{eq:binomial-split}, $(s-t)^n = (s, t)^n_+ - (s,t)^n_- \in \mathfrak{a}_{\approx}$,
		whence $s - t \in \mathfrak{a}_{\approx}$, i.e.\ $s \approx t$ by \cref{prop:ideals-to-relations-complex}.

		The converse requires the intermediate extension of scalars to $\mathbb{Q}S$.
		Write $\mathfrak{a}^{\mathbb{Q}}_{\approx} \subseteq \mathbb{Q}S$ for the ideal
		corresponding to $\approx$ under \cref{prop:ideals-to-relations-rational},
		and suppose that $\approx$ satisfies \eqref{eq:radical}.
		Every element of $\mathbb{Q}S$ is of the form $x = \frac{s-t}{n}$ for $s,t \in S$ and $n \geq 1$.
		If $x^m \in \mathfrak{a}^{\mathbb{Q}}_{\approx}$ for some $m \geq 1$,
		then, as $\mathfrak{a}^{\mathbb{Q}}_{\approx}$ is an ideal of the $\mathbb{Q}$-algebra $\mathbb{Q}S$,
		\[
			(s, t)^m_+ - (s,t)^m_- = (s-t)^m = n^m x^m \in \mathfrak{a}^{\mathbb{Q}}_{\approx},
		\]
		i.e.\ $(s, t)^m_+ \approx (s,t)^m_-$.
		By \eqref{eq:radical}, $s \approx t$, so that $x = \frac1n (s-t) \in \mathfrak{a}^{\mathbb{Q}}_{\approx}$.
		Hence, $\mathfrak{a}^{\mathbb{Q}}_{\approx}$ is radical,
		equivalently, the quotient $\mathbb{Q}S / \mathfrak{a}^{\mathbb{Q}}_{\approx}$ is
		\emph{reduced} \cite[\href{https://stacks.math.columbia.edu/tag/01IZ}{Tag 01IZ}]{STACKS},
		i.e.\ has no non-zero nilpotents.
		As $\mathbb{Q}$ has characteristic zero, it is
		\emph{perfect} \cite[\href{https://stacks.math.columbia.edu/tag/030Y}{Tag 030Y}]{STACKS},
		so that $\mathbb{Q}S / \mathfrak{a}^{\mathbb{Q}}_{\approx}$ is
		\emph{geometrically reduced} \cite[\href{https://stacks.math.columbia.edu/tag/030S}{Tag 030S}]{STACKS}
		over $\mathbb{Q}$ by \cite[\href{https://stacks.math.columbia.edu/tag/00I4}{Tag 00I4}]{STACKS};
		in particular, $\mathbb{C} \otimes_{\mathbb{Q}} \left(\mathbb{Q}S / \mathfrak{a}^{\mathbb{Q}}_{\approx}\right)$ is reduced.
		By the definitions in \cref{sec:extension-of-scalars-Q,sec:extension-of-scalars-C},
		the ideal $\mathfrak{a}_{\approx}$ consists of the finite sums $\sum_i \alpha_i x_i$
		with $\alpha_i \in \mathbb{C}$ and $x_i \in \mathfrak{a}^{\mathbb{Q}}_{\approx}$,
		i.e.\ $\mathfrak{a}_{\approx} = \mathbb{C} \otimes_{\mathbb{Q}} \mathfrak{a}^{\mathbb{Q}}_{\approx}$
		as $\mathbb{C}$ is \emph{flat} \cite[\href{https://stacks.math.columbia.edu/tag/00HB}{Tag 00HB}]{STACKS} over $\mathbb{Q}$.
		The same flatness identifies the reduced ring above with $\mathbb{C}S / \mathfrak{a}_{\approx}$,
		whence $\mathfrak{a}_{\approx}$ is radical.
	\end{proof}

	\subsection{Semiring congruences from semiring homomorphism indistinguishability}
	\label{sec:semiring-congruence-indistingishability}
	In this section, we prove \cref{thm:semiring-hom-congruence} which characterises the graph isomorphism relaxations that are of the form $G \approx H \iff \forall \phi \in \Phi.\ \phi(G) = \phi(H)$ for some set $\Phi \subseteq \Hom(\lG, \mathbb{C})$ of semiring homomorphisms.
	Since this statement does not involve the combinatorics of~$\lG$, we prove the following general statement, which holds for any countable commutative cancellative semiring.

	\begin{theorem}\label{thm:phi-abstract}
		Let $S$ be a countable commutative cancellative semiring
		and let $\approx$ be an equivalence relation on~$S$.
		There exists a set $\Phi \subseteq \Hom(S, \mathbb{C})$ such that, for all $s, t \in S$,
		\[
			s \approx t \iff \forall \phi \in \Phi.\ \phi(s) = \phi(t)
		\]
		if, and only if,
		for all $s, t,r \in S$ and $n \geq 1$,
		\begin{assertions}
			\item\label{it:phi-add} $s \approx t \implies s + r \approx t +r$,
			\item\label{it:phi-mult} $s \approx t \implies s \cdot r \approx t \cdot r$,
			\item\label{it:phi-cancel} $s +r \approx t +r \implies s \approx t$,
			\item\label{it:phi-torsion} $n s \approx nt \implies s \approx t$,
			\item\label{it:phi-radical} $(s, t)^n_+ \approx (s, t)^n_- \implies s\approx t$.
		\end{assertions}
	\end{theorem}
	\begin{proof}

		Suppose first that there exists a set $\Phi \subseteq \Hom(S, \mathbb{C})$
		such that $s \approx t$ if, and only if, $\phi(s) = \phi(t)$ for all $\phi \in \Phi$.
		Let $s, t, r \in S$, let $n \geq 1$, and let $\phi \in \Phi$.
		\Cref{it:phi-add,it:phi-mult} hold since
		$\phi(s + r) = \phi(s) + \phi(r)$ and $\phi(s \cdot r) = \phi(s) \cdot \phi(r)$,
		and analogously with~$t$ in place of~$s$.
		\Cref{it:phi-cancel} holds since
		$\phi(s) + \phi(r) = \phi(t) + \phi(r)$ implies $\phi(s) = \phi(t)$.
		\Cref{it:phi-torsion} holds since $\phi(n s) = n\phi(s)$ and $\phi(n t) = n\phi(t)$,
		and $n\phi(s) = n\phi(t)$ implies $\phi(s) = \phi(t)$,
		as $\mathbb{C}$ is of characteristic zero.
		For \cref{it:phi-radical}, observe that, by the Binomial Theorem,
		\[
			\phi\left((s, t)^n_+\right) - \phi\left((s,t)^n_-\right)
			= \sum_{i=0}^n \binom{n}{i} \phi(s)^{n-i} \left(-\phi(t)\right)^i
			= \left(\phi(s) - \phi(t)\right)^n.
		\]
		Hence, if $(s, t)^n_+ \approx (s,t)^n_-$,
		then $\left(\phi(s) - \phi(t)\right)^n = 0$ and thus $\phi(s) = \phi(t)$,
		for every $\phi \in \Phi$.
		That is, $s \approx t$.

		Conversely, suppose that $\approx$ satisfies \cref{it:phi-add,it:phi-mult,it:phi-cancel,it:phi-torsion,it:phi-radical}.
		\Cref{it:phi-add,it:phi-cancel} are the two directions of \eqref{eq:c1},
		\cref{it:phi-mult} is \eqref{eq:c2},
		and \cref{it:phi-torsion} states that $\approx$ is torsion-free.
		Hence, $\approx$ is a torsion-free cancellative congruence on $S$.
		By \cref{prop:ideals-to-relations-complex},
		\[
			\mathfrak{a} \coloneqq \mathfrak{a}_{\approx}
			= \left\{ \sum_i \alpha_i \frac{s_i - t_i}{n_i} \in \mathbb{C}S
			\ \middle|\ \alpha_i \in \mathbb{C},\ n_i \geq 1,\ s_i \approx t_i \right\}
		\]
		is an ideal of $\mathbb{C}S$ such that, for all $s, t \in S$,
		it holds that $s \approx t$ if, and only if, $s - t \in \mathfrak{a}$.
		Furthermore, \cref{it:phi-radical} is \eqref{eq:radical}.

		Since $S$ is countable, \cref{thm:nullstellensatz} applies.
		Writing $\Phi \coloneqq V(\mathfrak{a}) \subseteq \Hom(S, \mathbb{C})$, it yields
		\[
			I(\Phi) = I(V(\mathfrak{a})) = \sqrt{\mathfrak{a}}.
		\]
		Therefore, for all $s, t \in S$,
		\[
			s \approx t
			\iff s - t \in \sqrt{\mathfrak{a}} = I(\Phi)
			\iff \forall \phi \in \Phi.\ \phi(s) - \phi(t) = 0,
		\]
		where the first equivalence is \cref{lem:radical-differences}.
		This is the desired characterisation of $\approx$.
	\end{proof}

	\Cref{thm:phi-abstract} readily implies the following theorem via \cref{lem:graphs-form-semiring}.
	\thmPhi*

	\subsection{Subgraph series}
	\label{sec:subgraph-series}

	So far we have not made use of the combinatorics of $\lG$.
	Given \cref{thm:semiring-hom-congruence},
	it remains to understand the semiring homomorphisms $\phi \in \Hom(\lG, \mathbb{C})$ that are of the form $\hom(F, \star)$ for some connected graph $F$.
	Towards this goal,
	we observe that every graph parameter $\lG \to \mathbb{C}$
	can be uniquely expressed as an infinite linear combination of subgraph counts $\sub(F, \star)$.

	\subsubsection{Unique coefficients and support}
	\label{sec:unique-coefficients}

	For a graph $F$ and sets $S \subseteq V(F)$ and $E' \subseteq E(F)$,
	write $F[S, E']$ for the subgraph of $F$ with vertex set $S$
	whose edges are precisely those edges of $E'$ all of whose endpoints lie in $S$.
	The subgraphs of $F$ are exactly the graphs of this form.
	We abbreviate the induced subgraph $F[S] \coloneqq F[S, E(F)]$.

	\begin{lemma}\label{lem:unique-subgraph-series}
		Let $p \colon \lG \to \mathbb{C}$ be a graph parameter.
		For a graph $F \in \lG$, let
		\[
			\alpha_F \coloneqq \sum_{S \subseteq V(F)} \ \sum_{E' \subseteq E(F)} (-1)^{|V(F) \setminus S| + |E(F) \setminus E'|}  p(F[S, E']) \in \mathbb{C}.
		\]
		Then $p(\star) = \sum_F \alpha_F \sub(F, \star)$,
		where for every input graph $G$ only finitely many summands are non-zero.
		Moreover, the $\alpha_F$ are the unique complex numbers with this property.
	\end{lemma}

	We call $\sum_F \alpha_F \sub(F, \star)$ the \emph{subgraph series} of the graph parameter $p(\star)$.
	Here, we adopt the convention that the indices $F \in \lG$ are pairwise non-isomorphic.

	\begin{proof}

		\textbf{Uniqueness.}
		Suppose that $\sum_F \alpha_F \sub(F, \star) = \sum_F \alpha'_F \sub(F, \star)$
		and write $\gamma_F \coloneqq \alpha_F - \alpha'_F$,
		so that $\sum_F \gamma_F \sub(F, \star)$ vanishes identically.
		We show that $\gamma_{F_0} = 0$ for every $F_0 \in \lG$
		by induction along the lexicographic order on the pairs $(|V(F_0)|, |E(F_0)|) \in \mathbb{N}^2$.
		So let $F_0 \in \lG$ and suppose that $\gamma_F = 0$
		for every graph $F$ whose pair is lexicographically smaller than that of~$F_0$.
		Every $F \not\cong F_0$ with $\sub(F, F_0) \neq 0$ is such a graph,
		as $|V(F)| \leq |V(F_0)|$ and $|E(F)| \leq |E(F_0)|$ with equality in both only if $F \cong F_0$.
		Evaluating at $F_0$ therefore yields
		\[
			0 = \sum_F \gamma_F \sub(F, F_0) = \gamma_{F_0} \sub(F_0, F_0) = \gamma_{F_0}.
		\]

		\textbf{Existence.}
        We apply M\"obius inversion over the Boolean lattice (\cref{lem:mobius-inversion}) twice.
		Define the graph parameter
		\[
			q(K) \coloneqq \sum_{S \subseteq V(K)} (-1)^{|V(K) \setminus S|} p(K[S])
			\qquad \text{for } K \in \lG.
		\]
		The graph $F[V(F), E']$ has vertex set $V(F)$ and satisfies $F[V(F), E'][S] = F[S, E']$,
		so that $q(F[V(F), E']) = \sum_{S \subseteq V(F)} (-1)^{|V(F) \setminus S|} p(F[S, E'])$.
		Exchanging the order of summation in the definition of $\alpha_F$ therefore gives
		\begin{equation}\label{eq:alpha-def}
			\alpha_F = \sum_{E' \subseteq E(F)} (-1)^{|E(F) \setminus E'|} q(F[V(F), E']).
		\end{equation}
		We claim that, for all graphs $G, K \in \lG$,
		\begin{align}
			p(G) &= \sum_{S \subseteq V(G)} q(G[S]), \label{eq:invert-vertices}\\
			q(K) &= \sum_{E' \subseteq E(K)} \alpha_{K[V(K), E']}. \label{eq:invert-edges}
		\end{align}
		Both are instances of \cref{lem:mobius-inversion}, i.e.\ of Möbius inversion over a Boolean lattice.
		For \eqref{eq:invert-vertices}, apply it to $X \coloneqq V(G)$
		and to $f(S) \coloneqq p(G[S])$ and $g(S) \coloneqq q(G[S])$:
		as $G[S][T] = G[T]$ for $T \subseteq S$,
		the definition of $q$ states precisely that $g(S) = \sum_{T \subseteq S} (-1)^{|S \setminus T|} f(T)$,
		so that $f(S) = \sum_{T \subseteq S} g(T)$; now take $S \coloneqq V(G)$.
		For \eqref{eq:invert-edges}, apply it to $X \coloneqq E(K)$
		and to $f(E') \coloneqq q(K[V(K), E'])$ and $g(E') \coloneqq \alpha_{K[V(K), E']}$:
		as $K[V(K), E'][V(K), E''] = K[V(K), E'']$ for $E'' \subseteq E'$,
		\eqref{eq:alpha-def} states precisely that $g(E') = \sum_{E'' \subseteq E'} (-1)^{|E' \setminus E''|} f(E'')$,
		so that $f(E') = \sum_{E'' \subseteq E'} g(E'')$; now take $E' \coloneqq E(K)$.

		Finally, let $G \in \lG$ be arbitrary.
		Every subgraph of $G$ is of the form $G[S, E']$
		for a unique pair of sets $S \subseteq V(G)$ and $E' \subseteq E(G[S])$,
		and $G[S][V(G[S]), E'] = G[S, E']$ for such a pair.
		Hence, grouping the subgraphs of $G$ by their isomorphism types,
		\[
			\sum_F \alpha_F \sub(F, G)
			= \sum_{G' \subseteq G} \alpha_{G'}
			= \sum_{S \subseteq V(G)} \ \sum_{E' \subseteq E(G[S])} \alpha_{G[S, E']}
			\overset{\eqref{eq:invert-edges}}{=} \sum_{S \subseteq V(G)} q(G[S])
			\overset{\eqref{eq:invert-vertices}}{=} p(G),
		\]
		where all sums are finite.
	\end{proof}

	By \cref{lem:unique-subgraph-series}, the coefficients of the subgraph series are uniquely determined by $p$.
	Hence, the following notion is well-defined.
	\begin{definition}\label{def:support}
		For a graph parameter $p \colon \lG \to \mathbb{C}$, let $\alpha_F \in \C$ for $F \in \lG$ be as in  \cref{lem:unique-subgraph-series} and define the \emph{support} of $p$ as the graph class
		\(
		\supp(p) \coloneqq \{F \in \lG \mid \alpha_F \neq 0 \}.
	\)
	\end{definition}
	
	Recalling $c_F \in \mathbb{Z}\lG$ from \cref{eq:c-F}, we prove the following lemma, which allows us to recast the support of a semiring homomorphism~$\phi$ in terms of its values.
    Note that, in contrast to \cref{def:support,lem:unique-subgraph-series},
    this requires $p$ to be a semiring homomorphism.

	\begin{lemma}\label{lem:c-F-support}
		Let $\phi \in \Hom(\lG, \mathbb{C})$ and let $F \in \lG$.
		Then $\phi(c_F) = \alpha_F$,
		where $\alpha_F$ denotes the coefficient of $F$ in the subgraph series of $\phi$.
		In particular, $\phi(c_F) = 0$ if, and only if, $F \notin \supp(\phi)$.
	\end{lemma}
	\begin{proof}
		By \cref{lem:unique-extension},
		$\phi$ extends to a $\mathbb{C}$-algebra homomorphism $\mathbb{C}\lG \to \mathbb{C}$,
		which we again denote by $\phi$.
		In particular, $\phi$ is $\mathbb{C}$-linear, and hence, recalling  \cref{eq:c-F,lem:unique-subgraph-series},
		\[
			\phi(c_F)
			= \sum_{S \subseteq V(F)} \ \sum_{E' \subseteq E(F)} (-1)^{|V(F) \setminus S| + |E(F) \setminus E'|} \phi(F[S, E'])
			= \alpha_F.
		\]
		Hence, $\phi(c_F) = 0$ if, and only if, $\alpha_F = 0$,
		that is, if, and only if, $F \notin \supp(\phi)$.
	\end{proof}

	Finally, note that the double M\"obius inversion in \cref{lem:unique-subgraph-series} is essential.
	One may be tempted to invert over edge subsets only as done in \cite[Section 4.1]{lovasz_large_2012},
	that is, to expand $p$ into counts of \emph{spanning} subgraphs.
	The following remark explains why this is not sufficient for our purposes.

	\begin{remark}\label{rem:spanning-insufficient}
		For graphs $F, G \in \lG$, let $\spsub(F, G)$ denote the number of spanning subgraphs of $G$ isomorphic to $F$,
		i.e.\ the number of sets $E' \subseteq E(G)$ with $G[V(G), E'] \cong F$,
		and put
		\[
			\beta_F \coloneqq \sum_{E' \subseteq E(F)} (-1)^{|E(F) \setminus E'|} p(F[V(F), E']),
		\]
		i.e.\ retain only the summands with $S = V(F)$ in the definition of $\alpha_F$.
		A single application of \cref{lem:mobius-inversion} to $X \coloneqq E(G)$ yields
		$p(G) = \sum_{E' \subseteq E(G)} \beta_{G[V(G), E']} = \sum_F \beta_F \spsub(F, G)$,
		so this expansion is valid as well.
		It is, however, useless for our purposes.

		Indeed, $\spsub(F, G) = 0$ unless $|V(F)| = |V(G)|$,
		so such an expansion amounts to one separate expansion for each vertex count,
		and a parameter with finitely many non-zero $\beta_F$ vanishes on all but finitely many vertex counts.
		In particular, no semiring homomorphism $\phi \in \Hom(\lG, \mathbb{C})$ has finitely many non-zero $\beta_F$,
		as $\phi$ is additive and hence maps the disjoint union of $n$ copies of $\boldsymbol{1}$ to $n$.
		The notion of finite support underlying \cref{lem:c-F-support,lem:intersection-vanishing} would thus be vacuous.
		For example, the parameter $p(\star) = \hom(K_1, \star)$ counting vertices
		satisfies $\alpha_{K_1} = 1$ and $\alpha_F = 0$ for all $F \not\cong K_1$,
		whereas $\beta_{\overline{K_n}} = n$ for every $n \geq 0$, writing $\overline{K_n}$ for the edgeless graph on $n$ vertices.
		Moreover, $\spsub(F, \star)$ is not itself a graph motif parameter, see \cref{sec:gmp}.
\end{remark}
	
	\subsubsection{Graph motif parameters}
	\label{sec:gmp}

	Of particular importance are subgraph series with finite support.
	These were studied by \textcite{curticapean_homomorphisms_2017} under the term \emph{graph motif parameter}, see also the survey \cite{roth_parameterised_2026}.

	\begin{definition}\label{def:gmp}
		A \emph{graph motif parameter} is a function $p \colon \lG \to \mathbb{C}$ with finite support $\supp(p)$.
	\end{definition}

	By definition, the functions $\sub(F, \star)$ for $F \in \lG$ span the $\mathbb{C}$-vector space of graph motif parameters. 
	It is implicit in the work of \textcite{lovasz_operations_1967} that these functions form a basis of this space.
	In the following, it will be more convenient to work with an alternative basis of the space of graph motif parameters, the \emph{homomorphism basis}.
    The basis change is given by 
	\cite[(5.18)]{lovasz_large_2012}, i.e., for every $F \in \lG$, it holds that
	\begin{equation}\label{eq:sub-to-hom}
		\sub(F, \star) = \frac{1}{\aut(F)} \sum_{\pi} \mu_\pi \hom(F/\pi, \star),
	\end{equation}
	where $\aut(F)$ is the number of automorphisms of $F$, $\pi$ ranges over partitions of $V(F)$, $F/\pi$ is a suitably defined graph quotient, and the $\mu_\pi$ are coefficients arising from M\"obius inversion over the set-partition lattice.
	Hence, the homomorphism counts $\hom(F, \star)$ for $F \in \lG$ span the space of graph motif parameters. In fact, they form a basis.

	\begin{theorem}[\textcite{lovasz_operations_1967}, see also {\cite[Proposition~5.44b]{lovasz_large_2012}}]\label{thm:lovasz-linear-independent}
		For pairwise non-isomorphic graphs $F_1, \dots, F_n \in \lG$,
		the functions $\hom(F_1, \star), \dots, \hom(F_n, \star) \colon \lG \to \mathbb{C}$ are linearly independent.
	\end{theorem}

	We use \cref{thm:lovasz-linear-independent} to characterise the graph motif parameters that are of the form $\hom(F, \star)$ for some $F \in \lG$, i.e., they are comprised only of a single homomorphism count.
	We write graph motif parameters in this basis as $\sum \eta_F \hom(F, \star)$, where, in contrast to the definition of subgraph series, all but finitely many coefficients $\eta_F$ vanish.

	\begin{lemma}\label{lem:gmp-multiplicative}
		For every non-zero graph motif parameter $p(\star) = \sum \eta_F \hom(F, \star)$,
		the following are equivalent:
		\begin{assertions}
			\item $p$ is \emph{multiplicative}, i.e., $p(G \otimes H) = p(G) p(H)$ for all $G, H \in \lG$,
			\item $p$ is of the form $p(\star) = \hom(F, \star)$ for some graph $F \in \lG$.
		\end{assertions}
	\end{lemma}
	\begin{proof}
		The backward direction is immediate from \cref{eq:prod}.
		Conversely, 
		let $G, H \in \lG$ be arbitrary.
		By assumption and \cref{eq:prod},
		\begin{align*}
			\sum_{F, F'} \eta_F \eta_{F'} \hom(F, G) \hom(F',H)
			&= p(G) p(H)  \\
			&= p(G \otimes H)\\
			&= \sum_F \eta_F \hom(F, G \otimes H)\\
			&\overset{\eqref{eq:prod}}{=} \sum_F \eta_F \hom(F, G) \hom(F,H).
		\end{align*}
		For fixed $H$, we first view both sides as graph motif parameters in $G$. That is,
		\[
		\sum_{F, F'} \eta_F \eta_{F'} \hom(F, \star) \hom(F',H)  =  \sum_F \eta_F \hom(F, \star) \hom(F,H).
		\]
		Since the functions $\hom(F, \star)$ are linearly independent by \cref{thm:lovasz-linear-independent}, i.e.\ they form a basis of the space of graph motif parameters, we may compare coefficients and obtain
		\[
			\eta_{F} \sum_{F'} \eta_{F'} \hom(F', H)
			= 
			\eta_F \hom(F, H).
		\]
		Fix $F$ with $\eta_F \neq 0$, which exists since $p$ is non-zero, and divide both sides by $\eta_F$.
		Since $H$ was chosen arbitrarily, we can view this identity as an identity of graph motif parameters 
		\[
			\sum_{F'} \eta_{F'} \hom(F', \star)
			= 
			\hom(F, \star)
		\]
		and compare coefficients.
		Hence, $\eta_{F'} = 0$ for all $F' \not\cong F$ and $\eta_F = 1$.
		In particular, $p(\star) = \hom(F, \star)$ for some $F \in \lG$, as desired.
	\end{proof}

	We deduce the following characterisation of homomorphism counts alternative to \textcite{lovasz_dual_2010}.

	\begin{corollary}
		\label{cor:fls-alternative}
		A non-zero function $p \colon \lG \to \mathbb{C}$ is of the form $p(\star) = \hom(F, \star)$ for some graph $F \in \lG$ if, and only if, it is multiplicative and has finite support.
	\end{corollary}

	Finally, we characterise the homomorphism counts $\hom(F, \star)$ 
	which are semiring homomorphisms $\lG \to \mathbb{C}$.
	For a more general statement, compare \cite[Exercise~4.2]{lovasz_large_2012}.

	\begin{lemma}\label{lem:gmp-additive}
		For every graph $F \in \lG$, the following are equivalent:
		\begin{assertions}
			\item $\hom(F, \star)$ is \emph{additive}, i.e., $\hom(F, G + H) = \hom(F, G) + \hom(F, H)$ for all $G, H \in \lG$,
			\item $F$ is connected.
		\end{assertions}
	\end{lemma}
	\begin{proof}
		The backward direction is immediate from \cref{eq:connected}.
		Conversely, let $F_1 + \dots + F_n = F$ be the connected components of $F$.
		Let $G \in \lG$ be arbitrary.
		By \cref{eq:connected,eq:coprod}, we have
		\begin{align*}
		\hom(F, G + \boldsymbol{1}) 
		&= \prod_{i\in [n]} \hom(F_i, G+ \boldsymbol{1}) \\
		&= \prod_{i\in [n]} \left( \hom(F_i, G) + \hom(F_i, \boldsymbol{1}) \right)\\
		&= \sum_{S \subseteq [n]} \hom(\sum_{s \in S} F_s, G) \hom(\sum_{t \in [n] \setminus S} F_t, \boldsymbol{1}).
		\\
		&= \sum_{S \subseteq [n]} \hom(\sum_{s \in S} F_s, G).
		\end{align*}
		By the assumption that $\hom(F, G + \boldsymbol{1}) = \hom(F, G) + \hom(F,  \boldsymbol{1})$, 
		\[
			\sum_{S \subseteq [n]} \hom(\sum_{s \in S} F_s, G)
			= 
			\hom(F, G) + 1.
		\]
		In particular,
		\[
			\sum_{\emptyset \neq S \subseteq [n]} \hom(\sum_{s \in S} F_s, G)
			= 
			\hom(F, G).
		\]
		We view these expressions as a graph motif parameter in $G$, i.e.
		\[
			\sum_{\emptyset \neq S \subseteq [n]} \hom(\sum_{s \in S} F_s, \star)
			= 
			\hom(F, \star).
		\]
		All coefficients on the left-hand side are non-negative, which rules out cancellations.
		There is only one pattern on the right-hand side.
		By comparing coefficients by virtue of \cref{thm:lovasz-linear-independent}, 
		there is only one pattern on the left-hand side.
		Hence, $n = 1$, as desired.
	\end{proof}

	Combining the lemmas above, we obtain our characterisation.

	\thmSemiringHomChar*
	\begin{proof}
		If $\phi(\star) = \hom(F, \star)$ for some connected graph $F$,
		then it is a graph motif parameter, i.e.\ it has finite support.
		It is additive and multiplicative by \cref{eq:connected,eq:prod}.
		Furthermore, $\phi(\boldsymbol{1}) = 1$.
		Conversely,
		since $\phi$ is multiplicative and of finite support,
		it is of the form $\phi(\star) = \hom(F, \star)$ for some graph $F\in \lG$ by \cref{lem:gmp-multiplicative}.
		Since $\phi$ is additive,
		the graph $F$ is connected by \cref{lem:gmp-additive}.
	\end{proof}

	\subsection{Putting the pieces together}

	\label{sec:pieces}

	Towards \cref{thm:radical-intersection},
	we first identify the right-hand side of \eqref{eq:radical-intersection} as a vanishing ideal.

	\begin{lemma}\label{lem:intersection-vanishing}
        \begin{assertions}
            \item\label{it:intersection-vanishing1} For $\ell \geq 1$, it holds that $V(\mathfrak{c}_\ell) = \{\hom(F, \star) \mid F \in \lZ, |V(F)| \leq \ell\}$.
            \item\label{it:intersection-vanishing2} For $\ell \geq 1$, every ideal of  $\C\lG$ containing $\mathfrak{c}_\ell$ is radical.
            \item For an ideal $\mathfrak{a}$ of $\C\lG$, it holds that $\bigcap_{\ell \geq 1} (\mathfrak{a} + \mathfrak{c}_\ell)
			= I(\{\hom(F,\star) \mid F \in \lZ \} \cap V(\mathfrak{a})).$
        \end{assertions}
	\end{lemma}
	\begin{proof}
		By \cref{lem:c-F-support}, it holds that $\phi(c_F) = 0$ if, and only if, $F \notin \supp(\phi)$,
		for every $\phi \in \Hom(\lG, \mathbb{C})$ and every graph $F$.
		Since $\mathfrak{c}_\ell$ is generated by the $c_F$ for graphs $F$ on more than $\ell$ vertices,
		\[
			V(\mathfrak{c}_\ell) = \{\phi \in \Hom(\lG, \mathbb{C}) \mid \supp(\phi) \text{ contains only graphs on at most } \ell \text{ vertices}\}.
		\]
        In particular, each $\phi \in V(\mathfrak{c}_\ell)$ is of finite support
        and thus of the form $\phi(\star) = \hom(F, \star)$ for some $F \in \lZ$ by \cref{thm:characterisation-semiring-homs}.
        By applying M\"obius inversion to \cref{eq:sub-to-hom},
        it follows that $F \in \supp(\hom(F, \star))$ for every $F \in \lG$.
        Hence, $|V(F)| \leq \ell$.
        Conversely, if $F \in \lZ$ has at most $\ell$ vertices,
        then by \cref{eq:sub-to-hom} every graph in $\supp(\hom(F, \star))$ is a quotient of $F$
        and hence has at most $\ell$ vertices, so that $\hom(F, \star) \in V(\mathfrak{c}_\ell)$.

        For the second assertion, observe that the elements of the form $c_F$ where $F$ is connected form a basis of $\C\lG$ over $\C$ since $c_F - F$ is a linear combination of connected graphs with fewer vertices than $F$, or with the same number of vertices and fewer edges. It follows that the dimension of $\C\lG/\mathfrak{c}_\ell$ over $\C$ is at most the	number of connected graphs on $\ell$ vertices or less. By \cref{thm:lovasz-linear-independent}, the functions $\hom(F, \star)$ for $F \in \lZ$ with $|V(F)| \leq \ell$ are linearly independent, so the $\C$-linear evaluation map
		\[ 
		\C\lG \longrightarrow \prod_{\substack{F\in\lZ\\|V(F)|\leq\ell}}\C,
		\qquad
		p \longmapsto (\hom(F,p))_F
		\]
		is surjective. This map factors through $\C\lG/\mathfrak{c}_\ell$ by \cref{it:intersection-vanishing1}, and by the bound on the dimension of $\C\lG/\mathfrak{c}_\ell$, we obtain an isomorphism
		\[
		\C\lG/\mathfrak{c}_\ell
		\cong
		\prod_{\substack{H\in\lZ\\ |V(H)|\leq\ell}}\C.
		\]
		Every ideal of this ring is radical, hence every ideal of $\C\lG$ containing $\mathfrak{c}_\ell$ is radical.
        
        For the third assertion, observe that $V(\mathfrak{a} + \mathfrak{c}_\ell) = V(\mathfrak{a}) \cap V(\mathfrak{c}_\ell)$.
		Therefore, using \cref{thm:nullstellensatz} in the first step,
		\[
			\bigcap_{\ell \geq 1} \sqrt{\mathfrak{a} + \mathfrak{c}_\ell}
			= \bigcap_{\ell \geq 1} I\left(V(\mathfrak{a}) \cap V(\mathfrak{c}_\ell)\right)
			= I\left(\bigcup_{\ell \geq 1} \left(V(\mathfrak{a}) \cap V(\mathfrak{c}_\ell)\right)\right)
            = I\left(V(\mathfrak{a}) \cap \bigcup_{\ell \geq 1} V(\mathfrak{c}_\ell)\right).
		\]
        The assertion now follows by applying \cref{it:intersection-vanishing1} and \cref{it:intersection-vanishing2}.
	\end{proof}

    This concludes the preparations for the proof of \cref{thm:radical-intersection}.
    Note that \eqref{eq:h5} is not among the hypotheses of \cref{thm:radical-intersection}.
	It is in fact implied by \eqref{eq:radical-intersection}:
	the right-hand side of \eqref{eq:radical-intersection} is an intersection of radical ideals,
	so that \eqref{eq:radical-intersection} forces $\mathfrak{a}_{\approx}$ to be radical,
	which is equivalent to \eqref{eq:h5} by \cref{lem:radical-ideal}.
    
	\thmRadInt*
	\begin{proof}
		Write $\mathfrak{a} \coloneqq \mathfrak{a}_{\approx}$,
		$\mathcal{F} \coloneqq \{F \in \lZ \mid \hom(F, \star) \in V(\mathfrak{a})\}$,
		and $\Psi \coloneqq \{\hom(F, \star) \mid F \in \mathcal{F}\}$.
		By \cref{lem:intersection-vanishing}, \eqref{eq:radical-intersection} reads $\mathfrak{a} = I(\Psi)$.
		By \cref{lem:c-ideals}, $G \approx H$ if, and only if, $G - H \in \mathfrak{a}$.

        In the backward direction,
		suppose that $\mathfrak{a} = I(\Psi)$.
		Then $G \approx H$ iff $G - H \in I(\Psi)$ iff $\hom(F, G) = \hom(F, H)$ for all $F \in \mathcal{F}$,
		i.e.\ ${\approx} = {\equiv_{\mathcal{F}}}$.
		As $\mathcal{F}$ consists of connected graphs, ${\equiv_{\mathcal{F}}}$ coincides with homomorphism indistinguishability over the componental closure of $\mathcal{F}$ by \eqref{eq:coprod}.

		Conversely, invoking \cref{eq:coprod}, suppose that ${\approx} = {\equiv_{\mathcal{F}_0}}$ for a class $\mathcal{F}_0$ of connected graphs.
		Let $\Psi_0 \coloneqq \{\hom(F, \star) \mid F \in \mathcal{F}_0\}$.
		Every $\hom(F, \star) \in \Psi_0$ vanishes on the generators $G - H$ of $\mathfrak{a}$,
		so $\mathcal{F}_0 \subseteq \mathcal{F}$ and $\Psi_0 \subseteq \Psi$.
		Moreover, $G - H \in I(\Psi_0)$ iff $G \equiv_{\mathcal{F}_0} H$ iff $G \approx H$,
		i.e.\ ${\approx_{I(\Psi_0)}} = {\approx} = {\approx_{\mathfrak{a}}}$.
		Both $I(\Psi_0)$ and $\mathfrak{a}$ are extended from $\mathbb{Q}\lG$
		by \cref{lem:vanishing-ideal-extended} and \eqref{eq:extended-ideal}, respectively,
		so $I(\Psi_0) = \mathfrak{a}$ by \cref{prop:ideals-to-relations-complex}.
		Hence,
		$\mathfrak{a} \subseteq I(V(\mathfrak{a})) \subseteq I(\Psi) \subseteq I(\Psi_0) = \mathfrak{a}$,
		using $\Psi \subseteq V(\mathfrak{a})$ and $\Psi_0 \subseteq \Psi$.
	\end{proof}

	This concludes the preparations for the proof of \cref{thm:axiomatisation}.

	\thmAxiom*
	\begin{proof}
		\Cref{it:a1,it:a2,it:a3,it:a4} are \cref{eq:h1,eq:h2,eq:h3,eq:h4},
		which hold for every homomorphism indistinguishability relation over a componental graph class.
		Write $\mathfrak{a} \coloneqq \mathfrak{a}_{\approx}$.
		By \cref{lem:c-ideals}, $k(G-H) \approx G'-H'$ if, and only if, $k(G-H) - (G'-H') \in \mathfrak{a}$.
		By \cref{thm:radical-intersection}, it suffices to show that \cref{it:a5} is equivalent to
		$\mathfrak{a} = \bigcap_{\ell} (\mathfrak{a} + \mathfrak{c}_\ell)$.
		Recall that $G' \simeq_\ell H'$ iff $G' - H' \in \mathfrak{c}_\ell$.

		Suppose that $\mathfrak{a} = \bigcap_{\ell \geq 1} (\mathfrak{a} + \mathfrak{c}_\ell)$
		and that $G, H$ satisfy the premise of \cref{it:a5}.
		For every $\ell \geq 1$, it holds that  $k(G-H) \in \mathfrak{a} + \mathfrak{c}_\ell$ and, as $\frac1k \in \C$, $G - H \in \mathfrak{a} + \mathfrak{c}_\ell$.
		Hence, $G - H \in \mathfrak{a}$, i.e.\ $G \approx H$.

		Conversely, assume \cref{it:a5}.
		By \cref{lem:intersection-vanishing,lem:vanishing-ideal-extended},
		the ideals $\mathfrak{b} \coloneqq \bigcap_{\ell} (\mathfrak{a} + \mathfrak{c}_\ell) \supseteq \mathfrak{a}$
		and $\mathfrak{c}_\ell = I(V(\mathfrak{c}_\ell))$ are vanishing ideals of integer-valued homomorphisms and thus extended from $\mathbb{Q}\lG$,
		as is $\mathfrak{a}$ by \eqref{eq:extended-ideal}.
		By \cref{prop:ideals-to-relations-complex}, it suffices to show that $G - H \in \mathfrak{b}$ implies $G \approx H$.
		
		Let $\ell \geq 1$.
		Since $\mathfrak{a}$ and $\mathfrak{c}_\ell$ are extended from $\mathbb{Q}\lG$
		and $G-H \in \mathbb{Q}\lG$, we may write $G-H = a + c$
		with $a \in \mathfrak{a} \cap \mathbb{Q}\lG$ and $c \in \mathfrak{c}_\ell \cap \mathbb{Q}\lG$.
		Every element of $\mathbb{Q}\lG$ is of the form $\frac{G' - H'}{k}$ for some $G', H' \in \lG$ and $k \geq 1$,
		so $c = \frac{G' - H'}{k}$ with $G' \simeq_\ell H'$, and $k(G-H) - (G' - H') = ka \in \mathfrak{a}$,
		i.e.\ $k(G-H) \approx G' - H'$.
		As $\ell \geq 1$ was arbitrary, \cref{it:a5} yields
		$G \approx H$ and thus $G - H\in \mathfrak{a}$.
		It follows that $\mathfrak{a} = \mathfrak{b}$ and applying \cref{prop:ideals-to-relations-complex} concludes the proof.
	\end{proof}

    Finally, we prove \cref{lem:ideal-cl}.
    The lemma shows in particular that $\simeq_\ell$ is a homomorphism indistinguishability relation.
    
	\lemCl*
	\begin{proof}
		Assume that $G - H \in \mathfrak{c}_\ell$.
		Let $F$ be a connected graph on at most $\ell$ vertices.
		By \cref{lem:intersection-vanishing}, it holds that $\hom(F, \star) \in V(\mathfrak{c}_\ell)$.
		In particular, $\hom(F, G) = \hom(F, H)$.
		The case of disconnected $F$ follows by \cref{eq:coprod}, as each of its components has at most $\ell$ vertices.

		Conversely, suppose that $G$ and $H$ are homomorphism indistinguishable over all graphs $F \in \lG$ on at most $\ell$ vertices.
		Let $\phi \in V(\mathfrak{c}_\ell)$.
		By \cref{lem:intersection-vanishing}, $\phi = \hom(F, \star)$ for some graph $F$ on at most $\ell$ vertices.
		It follows that $\phi(G) = \phi(H)$.
		Hence, $G -H \in I(V(\mathfrak{c}_\ell)) = \mathfrak{c}_\ell$ by  \cref{thm:nullstellensatz,lem:intersection-vanishing}.

        For the final assertion, consider the following additional equivalent assertions:
        \begin{assertions}[start=4]
            \item $\sub(F, G) = \sub(F, H)$ for all $F\in \lG$ on at most $\ell$ vertices,
			\item $\indsub(F, G) = \indsub(F, H)$ for all $F\in \lG$ on at most $\ell$ vertices.
        \end{assertions}    
        Here, $\indsub(F, \star)$ counts induced subgraphs isomorphic to $F$.
        For the equivalence of the second and the fifth assertion, 
		we observe that, by \cite{curticapean_homomorphisms_2017},
		the functions $\hom(F, \star)$ for graphs $F$ on at most $\ell$ vertices
		span the same space as the functions $\indsub(F, \star)$ that count induced subgraphs from graphs $F$ on at most $\ell$ vertices.
		The fourth assertion is equivalent to these as well,
		since the functions $\sub(F, \star)$ for graphs $F$ on at most $\ell$ vertices
		span that same space \cite{curticapean_homomorphisms_2017}.
		Finally, it holds that $\indsub(F, G) = \indsub(F, H)$ for every $F$ on at most $\ell$ vertices
		if, and only if, the bijection from the third assertion exists.
	\end{proof}

	\subsection{Applications}

	\label{sec:applications}

    We give three applications of \cref{thm:axiomatisation,thm:axiomatisation-simple}.

    \subsubsection{Ulam--Kelly Reconstruction Conjecture}

    The following \cref{ex:ulam} shows that the involved \cref{it:a5} is not the bottleneck when applying \cref{thm:axiomatisation-simple} to the Ulam--Kelly Reconstruction Conjecture.

	\exUlam*

    In order to establish the equivalence with \cref{it:ulam3}, we require the following lemma.
    Recall the definition of graph motif parameters from \cref{def:gmp}.
    \begin{lemma}\label{lem:ulam-gmp}
		For every graph $F \in \G$, the graph parameter $r_F \colon \G \to \C$
		\[
			G \mapsto \sum_{v \in V(G)} \hom(F, G - v)
		\]
		is a graph motif parameter $\sum_{L \in \G} \eta_L \hom(L, \star)$
        such that the unique coefficient $\eta_{F + K_1}$ is non-zero.
	\end{lemma}

	\begin{proof}
		First note that, for every $v \in V(G)$,
		\begin{equation}\label{eq:two-types}
			\hom(F, G) = \hom(F, G - v) + \hom(F, G; \text{ hits } v).
		\end{equation}
		Moreover, using an observation of \textcite{lovasz_operations_1967}, see \cite[Lemma~3.1.2]{seppelt_homomorphism_2024},
		\begin{align*}
			\sum_{v \in V(G)} \hom(F, G; \text{ hits } v)
			&= \sum_{h \colon F\to G} |h(F)|  \\
			&= \sum_{K \in \G}  \frac{|V(K)| \cdot \surj(F, K)}{\aut(K)} \cdot \emb(K, G) \\
			&=   \sum_{K \in \G}  \frac{|V(K)| \cdot \surj(F, K)}{\aut(K)} \sum_{\pi \vdash V(K)} \mu_\pi \hom(K/\pi, G).
		\end{align*}
		Here, $\hom(F, G; \text{ hits } v)$ denotes the number of homomorphisms $F \to G$ whose image contains~$v$,
		$\surj(F, K)$ the number of surjective homomorphisms $F \to K$,
		and $\emb(K, G)$ the number of injective homomorphisms $K \to G$, so that $\emb(K, G) = \aut(K) \sub(K, G)$.
		The first sum ranges over all homomorphisms $h \colon F\to G$.
		In the last two lines, the (outer) sum ranges over all graphs $K$ on at most $|V(F)|$ vertices.
		In the last line, the inner sum ranges over all partitions of $V(K)$.
		By summing \cref{eq:two-types} over $v \in V(G)$ and observing that $|V(G)| \hom(F, G) = \hom(F + K_1, G)$, it follows that
		\begin{equation}\label{eq:ulam}
			\sum_{v\in V(G)} \hom(F, G  - v) = \hom(F + K_1, G)
			-  \sum_{K \in \G}  \frac{|V(K)| \cdot \surj(F, K)}{\aut(K)} \sum_{\pi \vdash V(K)} \mu_\pi \hom(K/\pi, G).
		\end{equation}
		This is a graph motif parameter evaluated at $G$ since the patterns and coefficients do not depend on~$G$. 
		Furthermore, after grouping terms for isomorphic patterns it follows that the coefficient of $F + K_1$ is non-zero.
		To that end, observe that the $K/\pi$ all have at most $|V(F)|$ vertices.
	\end{proof}

    The second ingredient we require is the following lemma.

    \begin{lemma}[{\cite[Lemma 6]{seppelt_logical_2024}, cf.\ \cite[Lemma 3.6]{curticapean_homomorphisms_2017}}]\label{lem:gmp-interpolation}
        Let $\approx$ be a simple graph isomorphism relaxation satisfying \cref{it:s1} of \cref{thm:axiomatisation-simple}.
        Let $p(\star) = \sum_{F \in \G} \eta_F \hom(F, \star)$ be a graph motif parameter.
        Suppose that, for all $G, H \in \G$, it holds that $G \approx H \implies p(G) = p(H)$.
        Then, for all $G, H \in \G$, it holds that $G \approx H \implies \hom(F, G) = \hom(F, H)$ for all $F \in \G$ with $\eta_F \neq 0$.
    \end{lemma}
    
    \begin{proof}[Proof of \cref{ex:ulam}]
		Given \cite{lovasz_operations_1967} and \cref{thm:axiomatisation-simple},
		only the implications from \cref{it:ulam3,it:ulam4} to \cref{it:ulam1} require an argument.

        For \cref{it:ulam4},
		let $G, H \in \G$ be hypomorphic via a bijection $\pi \colon V(G) \to V(H)$.
		As $G$ and $H$ have the same number of edges and $G - v \cong H - \pi(v)$,
		every $v \in V(G)$ satisfies
		$\deg_G(v) = |E(G)| - |E(G-v)| = |E(H)| - |E(H - \pi(v))| = \deg_H(\pi(v))$.
		Hence, $G$ and $H$ have the same number $\ell \geq 0$ of isolated vertices.
		By \cref{it:s2}, the graphs $G + K_1$ and $H + K_1$ are hypomorphic,
		say via a bijection $\sigma$.
		Both have $\ell + 1$ isolated vertices,
		and removing any one of them yields $G$ and $H$, respectively.
		Let $v$ denote the vertex of the added $K_1$ in $G + K_1$,
		so that $(G + K_1) - v = G$ has $\ell$ isolated vertices.
		Removing a non-isolated vertex of $H + K_1$ leaves the $\ell+1$ isolated vertices isolated.
		Hence, $\sigma(v)$ is isolated and $G \cong (H + K_1) - \sigma(v) = H$, as asserted by the Reconstruction Conjecture.

        For \cref{it:ulam3},
        we apply \cref{lem:ulam-gmp,lem:gmp-interpolation}.
        By definition, if $G, H \in \G$ are hypomorphic, 
        then there exists a bijection $\pi \colon V(G) \to V(H)$ such that $G - v \cong H- \pi(v)$ for all $v \in V(G)$.
        Hence, $r_F(G) = r_F(H)$ for all $F \in \G$ and $r_F$ the graph motif parameter from \cref{lem:ulam-gmp}.
        By \cref{lem:ulam-gmp,lem:gmp-interpolation},
        it follows that
        \[
            \hom(F + K_1, G) =\hom(F+K_1, H)
        \]
        for all $F \in \G$ whenever $G, H \in \G$ are hypomorphic.
        We may assume that $|V(G)| = |V(H)| \geq 1$, as otherwise $G$ and $H$ are both empty.
        By \eqref{eq:coprod},
        $\hom(F + K_1, G) = \hom(F, G) \cdot |V(G)|$ and likewise for $H$,
        so that $\hom(F, G) = \hom(F, H)$.
        Hence, $G$ and $H$ are homomorphism indistinguishable over all graphs and thus isomorphic by \cite{lovasz_operations_1967}.
	\end{proof}

	\subsubsection{\texorpdfstring{$D$}{D}-fractional isomorphism}
    
    For an integer $D \geq 1$, two graphs $G$ and $H$ are said to be \emph{$D$-fractionally isomorphic}, written $G \cong_\opfrac^D H,$ if there exists a bistochastic matrix $u$ such that $A_G^{(t)} u = u A_H^{(t)}$ for all $t \leq D$, where
	\[ (A_G^{(t)})_{ij} = \begin{cases}
		1, \quad \text{if $\dist_G(i,j) = t$}; \\
		0, \quad \text{otherwise}.
	\end{cases} \]
    This isomorphism relaxation was introduced by Botteron and Weber in \cite{botteron_communication_2024}, motivated by nonlocal games and communication complexity.

	In \cite[Open Question 2]{botteron_communication_2024}, Botteron and Weber ask for a homomorphism indistinguishability characterisation of $D$\nobreakdash-fractional isomorphism.
	We use \cref{thm:axiomatisation-simple} to show that neither $2$\nobreakdash-fractional isomorphism, nor $D$\nobreakdash-fractional isomorphism for some fixed $D \geq 2$, nor $D$-fractional isomorphism for all $D$, can be characterised as homomorphism indistinguishability over some graph class.
    
    We show this by exhibiting two graphs that violate \cref{it:s1}. A priori, this only shows that $\cong_\opfrac^D$ is not a homomorphism indistinguishability relation over a \emph{componental} graph class. However, it is easy to see that $\cong_\opfrac^D$ satisfies \cref{eq:h2}, that is
    \begin{equation*}
        G \cong_\opfrac^D H \implies K + G \cong_\opfrac^D K + H.
    \end{equation*}
    As was shown in \cite[Theorem~7]{seppelt_logical_2024}, a homomorphism indistinguishability relation satisfying \cref{eq:h2} is necessarily a homomorphism indistinguishability relation over a componental graph class. Our negative result thus follows from the contrapositive.
	
	\begin{plainexample}
		\label{xmpl:D-frac}
		Let $G = K_{3,3}$ and let $H$ be the triangular prism.
		Then both $A_G^{(1)}$ and $A_H^{(1)}$ are $3$-regular, both $A_G^{(2)}$ and $A_H^{(2)}$ are $2$-regular, and $A_G^{(t)} = A_H^{(t)} = 0$ for all $t > 2$, so the normalised all-ones matrix $u = \frac{1}{6} J$ gives a $D$-fractional isomorphism between $G$ and $H$ for all $D \geq 1$.
		However, we claim that $G \otimes K_2 \not \cong_\opfrac^D H \otimes K_2$ whenever $D \geq 2$, so $\cong_\opfrac^D$ fails \cref{it:s1} in \cref{thm:axiomatisation-simple}.
		Indeed, $G \otimes K_2$ is the disjoint union of two copies of $K_{3,3}$ and $H \otimes K_2$ is the hexagonal prism.
		The former has $12$ pairs of vertices at distance $2$, while the latter has $24$, so they cannot be $2$-fractionally isomorphic.
		We conclude that $\cong_\opfrac^D$ does not admit a homomorphism indistinguishability characterisation when $D \geq 2$.
	\end{plainexample}

    \subsubsection{Chromatic- and \texorpdfstring{$\operatorname{FO}$-equivalence}{FO-equivalence}}
    \label{subsec:FO-equiv}

    In \cite{atserias_expressive_2021}, the authors proved the following negative results.

    \begin{proposition}[{\cite[Theorem 17, Proposition 1, Proposition 2]{atserias_expressive_2021}}]
        The following equivalence relations are not homomorphism indistinguishability relations. 
        \begin{enumerate}
            \item Chromatic equivalence, having the same chromatic polynomial.
            \item $\operatorname{FO}^k$-equivalence, being indistinguishable by a sentence of first-order logic using at most $k$ distinct variables.
            \item $\operatorname{FO}_q$-equivalence, being indistinguishable by a sentence of first-order logic of quantifier-rank at most $q$.
        \end{enumerate}
    \end{proposition}

    We can recover these statements as corollaries of \cref{thm:axiomatisation-simple}. Clearly, the three relations satisfy \cref{eq:h2}, so it again suffices to show that they violate one of the conditions of the theorem.

    Indeed, chromatic equivalence violates \cref{it:s1}. Consider the graphs $G_1 = K_1 + P_3$ and $G_2 = 2 K_2$. As noted in \cite{atserias_expressive_2021}, it is easy to verify that they both have chromatic polynomial $x^2 (x-1)^2$ and hence are chromatically equivalent. On the other hand, we have $P_3 \otimes G_1 = 3 K_1 + C_4 + K_{1, 4}$ and $P_3 \otimes G_2 = 4 P_3$. The chromatic polynomial of the former is given by $x^3[(x-1)^4 + (x-1)]x(x-1)^4 = x^4(x-1)^8 + x^4(x-1)^5$ while the latter has chromatic polynomial $(x(x-1)^2)^4 = x^4(x-1)^8$.

    On the other hand, equivalence in first-order logic $\operatorname{FO}$---both over the bounded variable- and the bounded quantifier-rank fragment---violates \cref{it:s3}. Let $m \in \mathbb{N}$ be arbitrary and consider the graphs $G_1 = \overline{K_m}$ and $G_2 = \overline{K_{m-1}}$. Clearly, they are distinguished by the sentence
    \begin{equation*}
        \exists x_1 \dots \exists x_m \bigwedge_{1 \leq i < j \leq m} x_i \neq x_j
    \end{equation*}
    which is in $\operatorname{FO}^m \cap \operatorname{FO}_m$.
    On the other hand, we have $K_1 + G_1 = \overline{K_{m+1}}$ and $K_1 + G_2 = \overline{K_m}$. These graphs are not distinguishable by any $\operatorname{FO}^m$ or $\operatorname{FO}_m$ sentence, as a routine EF-game argument shows.

	\subsubsection{Fractional clique number}

    Finally, we give an example of a graph isomorphism relaxation satisfying \cref{it:s1,it:s2} but failing \cref{it:s3}. We answer a further question of \textcite{atserias_expressive_2021}.

    In \cite[Section~VII]{atserias_expressive_2021}, \citeauthor{atserias_expressive_2021}
    asked whether the relation $G \approx_{\omega_f} H \defiff \omega_f(G) = \omega_f(H)$ is a homomorphism indistinguishability relation (for simple graphs $G, H \in \G$).
    Here, $\omega_f(G)$ denotes the \emph{fractional clique number} defined as the optimum of the linear program
    \begin{align*}
        \max &\sum_{v\in V(G)} y_v \\
        \text{such that}& \sum_{u \in U} y_u \leq 1 \qquad \forall U \in I(G) \\
        & y_v \geq 0 \qquad \forall v \in V(G)
    \end{align*}
    Here, $I(G)$ denotes the set of independent sets in $G$. By the duality theorem of linear programming, $\omega_f(G)$ equals the \emph{fractional chromatic number} $\chi_f(G)$.
    We show that $\approx_{\omega_f}$ is not a homomorphism indistinguishability relation.

    \begin{example}\label{ex:fractional-clique-number}
        The relation $\approx_{\omega_f}$ satisfies \cref{it:s1,it:s2} but fails to satisfy \cref{it:s3} of \cref{thm:axiomatisation-simple}.
        In particular, it is not a homomorphism indistinguishability relation.
    \end{example}
    \begin{proof}
        By \cite[Theorem~2]{zhu_fractional_2011} and $\omega_f = \chi_f$ \cite{atserias_expressive_2021}, for all $G, H \in \G$,
        \[
            \omega_f(G \otimes H) = \min\left\{ \omega_f(G), \omega_f(H) \right\}.
        \]
        This identity readily implies \cref{it:s1}.
    
        For \cref{it:s2},
        we show that
        $\omega_f(G + H) = \max\left\{\omega_f(G), \omega_f(H)\right\}$ for $G, H \in \G$.
        For any feasible $y$ on $G+H$, put $\alpha = \max_{U \in  I(G)} \sum_{u\in U} y_u$ and $\beta = \max_{U \in  I(H)} \sum_{u\in U} y_u$. Every independent set of $G+H$ splits as the union of independent sets $U_G \cup U_H$, so $\alpha + \beta \le 1$; and $y|_G/\alpha$ is feasible for $G$, giving $y(G) \le \alpha\omega_f(G)$ and likewise for $H$. Hence \[ \sum_{v \in V(G + H)} y_v \le \alpha\omega_f(G) + \beta\omega_f(H) \le \max\{\omega_f(G),\omega_f(H)\},\] and putting all the weight on the better component attains it.

        In order to see that \cref{it:s3} is violated, consider
        $$\omega_f(K_3 + K_1) = \max\{3,1\} = 3 = \max\{3,2\} = \omega_f(K_3 + K_2),$$
        while $\omega_f(K_1) = 1 \neq 2 = \omega_f(K_2)$.

        By \cite[Theorem~7]{seppelt_logical_2024}, the relation $\approx_{\omega_f}$ satisfying \cref{it:s2} has to be a homomorphism indistinguishability  relation over a componental graph class if it is one at all.
        However, as it does not satisfy \cref{it:s3}, it is no homomorphism indistinguishability relation by \cref{thm:axiomatisation-simple}.
    \end{proof}

	\section{Examples of semiring homomorphisms on \texorpdfstring{$\lG$}{Gcirc}}
	\label{sec:examples}

	In this section, we give examples of graph isomorphism relaxations and semiring homomorphisms highlighting the necessity of the axioms in \cref{thm:axiomatisation}.

	\subsection{Necessity of \texorpdfstring{\cref{eq:h6}}{c\_l condition}}

	\exRank*
	\begin{proof}
		In order to verify \cref{eq:h1,eq:h2,eq:h3,eq:h4,eq:h5}, 
		it suffices by \cref{thm:semiring-hom-congruence}
		to show that $\rk \colon \lG \to \mathbb{C}$ is a semiring homomorphism.
		It holds that $\rk(\boldsymbol{1}) = \rk((1)) = 1$
		and for matrices $A, B$,
		\[
			\rk(A \oplus B) = \rk(A) + \rk(B), \qquad
			\rk(A \otimes B) = \rk(A) \cdot \rk(B)
		\]
		where $A \oplus B = \left( \begin{smallmatrix}
			A & 0 \\ 0 & B
		\end{smallmatrix}\right)$ is the matrix direct sum and $\otimes$ denotes the Kronecker product of matrices.
		The direct sum of the adjacency matrices of two graphs is the adjacency matrix of their disjoint union;
		the Kronecker product of the adjacency matrices of two graphs is the adjacency matrix of their categorical product.

		Suppose $\approx_{\rk}$ is a homomorphism indistinguishability relation over some graph class~$\mathcal{F} \subseteq \lG$.
		The adjacency matrices of the graphs $K_2^\circ$ and $K_3^\circ$ 
		are the $2 \times 2$ and the $3 \times 3$ 
		all-ones matrix, respectively.
		They both have rank~$1$.
		However, for every $n$-vertex graph $F \in \lG$, it holds that
		$\hom(F, K_\ell^\circ) = \ell^n$, hence $K_2^\circ$ and $K_3^\circ$ are not homomorphism indistinguishable over any graph except the empty graph $\boldsymbol{0}$, for which $\hom(\boldsymbol{0}, K_\ell^\circ) = 1$.
		This implies that $\mathcal{F} = \{\boldsymbol{0}\}$ which renders $\approx_{\rk}$ trivial, a contradiction.
	\end{proof}

	\cref{thm:characterisation-semiring-homs} implies that the support of  the graph parameter $\rk$ is infinite. We leave computing this support explicitly as an interesting open problem.

    Finally, we note that \cref{ex:rank} is an instance of a more general phenomenon.
    \begin{lemma}\label{lem:rk}
		The functions from Hermitian matrices to $\mathbb{C}$ satisfying, for all Hermitian $A, B$ and unitary~$X$,
		\[
			f(X^{*}AX) = f(A), \quad
			f(A \oplus B) = f(A) + f(B), \quad
			f(A \otimes B) = f(A) f(B)
		\]
		are of the form $f \colon A \mapsto \sum_{\lambda \in \Spec(A)} r(\lambda)$
		where $r$ is a semigroup homomorphism $(\mathbb{R}, \cdot) \to (\mathbb{C}, \cdot)$
        and $\Spec(A)$ denotes the multiset of eigenvalues of $A$.
	\end{lemma}

    For the rank, we have $r(x) = [x \neq 0]$.
    Since permutation matrices are unitary, every function~$f$ satisfying the properties in \cref{lem:rk} induces (by application to the adjacency matrix) a graph parameter $f \colon \lG \to \mathbb{C}$ which is a semiring homomorphism.
     
	\begin{proof}[Proof of \cref{lem:rk}]
        Let $f \colon A \mapsto \sum_{\lambda \in \Spec(A)} r(\lambda)$ for some semiring homomorphism~$r$.
        Since $X^* A X$ and $A$ are cospectral for every unitary~$X$, the first property holds.
        For the second property, observe that $\Spec(A \oplus B) = \Spec(A) \uplus \Spec(B)$ where we take the multiset union on the right hand-side.
        For the final property, $\Spec(A \otimes B) =\multiset{\mu\lambda \mid \mu \in \Spec(A), \lambda \in \Spec(B)}$.
        Hence, $f$ is of the desired form.

        Conversely,
		every Hermitian matrix can be unitarily diagonalised
		into a diagonal matrix of real eigenvalues $\lambda_1, \dots, \lambda_n$.
		Applying the $\oplus$-identity, the only remaining degree of freedom is $f( (\lambda) )$ for $\lambda \in \mathbb{R}$.
		Let $r \colon \mathbb{R} \to \mathbb{C}$ be this map. 
		The $\otimes$ identity implies that $r(ab) = r(a) r(b)$ for all $a,b \in \mathbb{R}$, that is, $r$ is a semigroup homomorphism.
	\end{proof}

\subsection{Semiring homomorphism on looped-simple graphs with imaginary values}

Next, we show that not all semiring homomorphisms $\lG \to \C$ have image contained in $\R$.
	To that end, we use the following lemmas.

	\begin{lemma}[{\cite[Exercise~5.52]{lovasz_large_2012}}]\label{lem:lov5.52}
		Let $G_1, \dots, G_n \in \G$ be pairwise non-isomorphic connected non-bipartite simple graphs.
		Then the functions $\hom(\star, G_1), \dots, \hom(\star, G_n)$ are linearly independent, even when restricted to connected graphs.
	\end{lemma}

    \begin{lemma}[\textcite{lovasz_cancellation_1971}, cf.\ {\cite[Theorem~9.10]{hammack_handbook_2011}, \cite[Theorem 5.37]{lovasz_large_2012}}]\label{lem:lovasz-cancellation}
        Let $A, B, C \in \lG$
        and suppose that $C$ has an odd cycle.
        Then $A \otimes C \cong B \otimes C \implies A \cong B$.
    \end{lemma}

	\begin{lemma}
		\label{lem:non-bipartite}
		Let $G \in \lG$ be connected, simple, and non-bipartite.
		Then:
		\begin{assertions}
			\item\label{itm:nb:trnsc} $G$ is transcendental in $\C\lG$, i.e., it is not the root of any non-zero polynomial $f \in \C[x]$.
			\item\label{itm:nb:non-inv} If $f \in \C[x]$ is any non-constant polynomial, then $f(G) \in \C\lG$ is non-invertible.
		\end{assertions}
	\end{lemma}
	\begin{proof}
		\ref{itm:nb:trnsc}.
		Since $G$ is connected and non-bipartite, the categorical powers $G^k$ are connected and non-bipartite for all $k \geq 0$.
		Note that $G^k \neq G^\ell$ whenever $k \neq \ell$.
		Suppose that $f = a_d x^d + \cdots + a_1 x + a_0 \in \C[x]$ is a polynomial such that $f(G) = \boldsymbol{0}$.
		Then for all connected graphs $K$ we have by \cref{eq:connected}
		\begin{align*}
			0 &= \hom(K,\boldsymbol{0}) \\
			&= \hom(K,a_d G^d + \cdots + a_1 G + a_0 \boldsymbol{1}) \\
			&= a_d \hom(K,G^d) + \cdots + a_1 \hom(K,G) + a_0 \hom(K,\boldsymbol{1}).
		\end{align*}
		By taking $K = \boldsymbol{1}$, it follows that $a_0 = 0$ since the looped graph $\boldsymbol{1}$ does not admit homomorphisms to any of the simple graphs $G^k$.
		Now by applying \cref{lem:lov5.52} to the remaining expression, it follows that $a_1 = \cdots = a_d = 0$.
		This shows that $G$ is transcendental.
		
		\ref{itm:nb:non-inv}.
		Write $f = a_d x^d + \cdots + a_1 x + a_0$ with $a_d \neq 0$ for some $d \geq 1$ (since $f$ is non-constant).
		Let $H \in \C\lG \setminus \{\boldsymbol{0}\}$ be any non-zero element, and write $H = \sum_{i=1}^k \beta_i H_i$ with $H_1,\ldots,H_k \in \lG$ connected and pairwise non-isomorphic, and $\beta_1,\ldots,\beta_k \in \C^\times$ (where $k \geq 1$, since $H \neq \boldsymbol{0}$).
		We show that $f(G)H \neq \boldsymbol{1}$.
		To that end, reorder the $H_i$ by decreasing size, and let $\ell \in [k]$ be such that $|V(H_1)| = \cdots = |V(H_\ell)| > |V(H_{\ell+1})| \geq \cdots \geq |V(H_k)|$.
		Now we write
		\[ f(G)H = \sum_{i=0}^d \sum_{j=1}^k a_i \beta_j G^i H_j = \underbrace{a_d \beta_1 G^d H_1 + \cdots + a_d \beta_\ell G^d H_\ell}_{\substack{\text{a linear combination of}\\\text{graphs of size $|V(G)|^d \cdot |V(H_1)|$}}} + \underbrace{\text{remaining terms}}_{\substack{\text{a linear combination of}\\\text{graphs of strictly smaller size}}}, \]
		and note that the graphs $G^d H_1 , \ldots , G^d H_\ell$ are connected (since $G$ is non-bipartite) and pairwise non-isomorphic by \cref{lem:lovasz-cancellation}.
		Since $|V(G)| \geq 3$ and $d,k,\ell \geq 1$, it follows that $f(G)H$ contains non-constant terms, so we have $f(G)H \neq \boldsymbol{1}$.
		Since $H$ was arbitrary, we conclude that $f(G)$ is non-invertible.
	\end{proof}

	\begin{corollary}
		\label{cor:non-bipartite-quotient}
		Let $G \in \lG$ be connected, simple, and non-bipartite.
		Then for every $\alpha \in \C$, there exists $\phi \in \Hom(\lG , \C)$ such that $\phi(G) = \alpha$.
	\end{corollary}
	\begin{proof}

		Let $\mathfrak{a} \subseteq \C\lG$ be the ideal generated by $(G - \alpha)(G - \overline{\alpha}) = G^2 - 2\Re(\alpha)G + |\alpha|^2 \cdot \boldsymbol{1}$.
		By \cref{lem:non-bipartite}\ref{itm:nb:non-inv}, the generator is not invertible, so $\mathfrak{a}$ is a proper ideal.
		Therefore, $\mathfrak{a}$ is contained in a maximal ideal $\mathfrak m$ \cite[\href{https://stacks.math.columbia.edu/tag/00E0}{Tag 00E0}]{STACKS}.
		By \cite[\href{https://stacks.math.columbia.edu/tag/00FU}{Tag 00FU}]{STACKS}, we have $\C\lG / \mathfrak m \cong \C$.
		Let $\psi \colon \C\lG \to \C\lG / \mathfrak m \cong \C$ be the canonical map.
		We have $(\psi(G) - \alpha)(\psi(G) - \overline{\alpha}) = 0$, so either $\psi(G) = \alpha$ or $\psi(G) = \overline{\alpha}$.
		It follows that either $\psi$ or $\overline{\psi} = \bar{\, }\, \circ \psi$ restricts to the desired homomorphism $\phi \colon \lG \to \C$ such that $\phi(G) = \alpha$.
	\end{proof}
	
	In particular, \cref{cor:non-bipartite-quotient} shows that homomorphisms $\lG \to \C$ are not necessarily real.
	\begin{remark}
		In the proof of \cref{cor:non-bipartite-quotient}, we could have taken the ideal generated by $G - \alpha$.
		Working with the ideal generated by $(G - \alpha)(G - \overline{\alpha})$ for some $\alpha \in \C \setminus \R$ shows something stronger, namely that even $\R\lG$ has a maximal ideal $\mathfrak m$ such that $\R\lG / \mathfrak m \cong \C$.
		So even when working with real coefficients, there still exist semiring homomorphisms onto the complex numbers.
	\end{remark}

    \begin{remark}
    When $\alpha$ is not a non-negative integer, it is clear that the semiring homomorphism $\phi \in \Hom(\lG , \C)$ constructed in \cref{cor:non-bipartite-quotient} is not of the form $\phi(H) = \hom(F, H)$ for some $F \in \lG$ and all $H \in \lG$, so this gives a wealth of \enquote{exotic} semiring homomorphisms $\lG \to \C$.
    Further, the kernel of the associated homomorphism $\psi \colon \C\lG \to \C$ is a maximal ideal, and therefore radical, so this gives a wealth of radical ideals in $\C\lG$ that are not associated to homomorphism indistinguishability relations.
\end{remark}

\section{Discarding loops} \label{sec:loops}

In this section, we discuss the transition from homomorphism indistinguishability relations on $\lG$ to homomorphism indistinguishability between simple graphs from $\G$.
To that end, we first answer an open question of \textcite{roberson_oddomorphisms_2022}.

\begin{question}[{\textcite[Question~1]{roberson_oddomorphisms_2022}}]
    For a family of simple graphs $\mathcal{F} \subseteq \G$,
    is  $\cl(\mathcal{F}) \subseteq \cl^\circ(\mathcal{F})$?\label{question:rob1}
\end{question}

Here, $\cl^\circ(\star)$ is the homomorphism distinguishing closure in $\lG$, i.e., for a graph class $\mathcal{F} \subseteq \lG$,
\[
\cl^\circ(\mathcal{F}) = \{F \in \lG \mid \forall G, H \in \lG.\ G \equiv_{\mathcal{F}} H \implies \hom(F, G) = \hom(F, H)\}.
\]
The difference to \cref{def:closure} is that $F, G, H \in \lG$ range over looped-simple graphs.
\Cref{question:rob1} asks for the relationship between $\cl(\star)$ and $\cl^\circ(\star)$.
We resolve this question in the affirmative. 

\begin{lemma}\label{lem:roberson-question}
	For every family $\mathcal{F} \subseteq \G$ of simple graphs,
	$\cl(\mathcal{F}) \subseteq \cl^\circ(\mathcal{F})$.
\end{lemma}

The proof of \cref{lem:roberson-question} requires the following fact, that we will use throughout this section.
\begin{fact}\label{fact:drop-loops}
    For $G \in \lG$ and $H \in \G$, it holds that $G \otimes H \in \G$.
\end{fact}

\begin{proof}[Proof of \cref{lem:roberson-question}]
	Let $F \in \cl(\mathcal{F})$.
	By definition, $F$ is a simple graph.
	Let $G, H \in \lG$ be not necessarily simple graphs such that $G \equiv_{\mathcal{F}} H$.
	Then $G \otimes F \equiv_{\mathcal{F}} H \otimes F$ by \eqref{eq:h1}.
	Since $F$ is simple, so are $G \otimes F$ and $H \otimes F$ by \cref{fact:drop-loops}.
	Hence, by assumption,
	\[
		\hom(F, G) \cdot \hom(F, F)
		\overset{\eqref{eq:prod}}{=} \hom(F, G\otimes F)
		= \hom(F, H \otimes F)
		\overset{\eqref{eq:prod}}{=} \hom(F, H) \cdot \hom(F, F). 
	\]
	Since $\hom(F, F) \neq 0$, we may divide both sides by this term.
	Hence, $\hom(F, G) = \hom(F, H)$ and $F \in \cl^\circ(\mathcal{F})$.
\end{proof}

\Cref{lem:roberson-question}
implies that the looped-simple homomorphism distinguishing closure of a class of simple graphs collapses to the simple homomorphism distinguishing closure.
That is, for every $\mathcal{F} \subseteq \G$, it holds that
\begin{equation}\label{eq:drop-loops-closure}
    \cl(\mathcal{F}) = \cl^\circ(\mathcal{F}) \cap \G.
\end{equation}

As a step towards \cref{thm:axiomatisation-simple},
we characterise the homomorphism indistinguishability relations between looped-simple graphs that are homomorphism indistinguishability relations over a class of simple graphs.

\begin{lemma}\label{lem:loops-same-relation}
		Let $\mathcal{F} \subseteq \lG$ be a class of looped-simple graphs.
		The relation $\equiv_{\mathcal{F}}$ is a homomorphism indistinguishability relation over a class of simple graphs
		if, and only if, for all graphs $G, H \in \lG$,
		\[
			\left( \forall n \geq 1.\ G \otimes K_n \equiv_{\mathcal{F}} H\otimes K_n \right)
			\implies G \equiv_{\mathcal{F}} H.
		\]
 	\end{lemma}
	\begin{proof}
		In the forward direction,
		suppose that $\equiv_{\mathcal{F}}$ coincides with $\equiv_{\mathcal{F}'}$ for a class $\mathcal{F}'$ of simple graphs
		and let $G, H \in \lG$ satisfy $G \otimes K_n \equiv_{\mathcal{F}} H \otimes K_n$ for all $n \geq 1$.
		Let $F \in \mathcal{F}'$ and put $n \coloneqq \max\{|V(F)|, 1\}$.
		As $F$ is simple, $\hom(F, K_n) \neq 0$.
		By \eqref{eq:prod},
		\[
			\hom(F, G)\hom(F, K_n)
			= \hom(F, G \otimes K_n)
			= \hom(F, H \otimes K_n)
			= \hom(F, H)\hom(F, K_n),
		\]
		and dividing by $\hom(F, K_n) \neq 0$ yields $\hom(F, G) = \hom(F, H)$.
		As $F \in \mathcal{F}'$ was arbitrary, $G \equiv_{\mathcal{F}'} H$ and hence $G \equiv_{\mathcal{F}} H$.

		For the backward direction, let $\mathcal{F}' \subseteq \mathcal{F}$ be the class of simple graphs in $\mathcal{F}$.
		Clearly, if $G \equiv_{\mathcal{F}} H$, then $G \equiv_{\mathcal{F}'} H$.
		Conversely, suppose $G \equiv_{\mathcal{F}'} H$ and let $n \geq 1$ and $F \in \mathcal{F}$.
		If $F \in \mathcal{F}'$, then \eqref{eq:prod} gives
		$\hom(F, G \otimes K_n) = \hom(F, G)\hom(F, K_n) = \hom(F, H)\hom(F, K_n) = \hom(F, H \otimes K_n)$.
		Otherwise, $F$ contains a loop, so $\hom(F, K_n) = 0$ as $K_n$ has none,
		and both sides vanish by \eqref{eq:prod}.
		Hence, $G \otimes K_n \equiv_{\mathcal{F}} H \otimes K_n$ for all $n \geq 1$,
		and the assumption yields $G \equiv_{\mathcal{F}} H$.
		It follows that $\equiv_{\mathcal{F}}$ and $\equiv_{\mathcal{F}'}$ coincide, as desired.
	\end{proof}

	\subsection{An axiomatisation for simple graph isomorphism relaxations}
	\label{sec:simple-axiomatisation}

	\Cref{thm:axiomatisation} axiomatises the homomorphism indistinguishability relations between \emph{looped-simple} graphs.
	In the literature, one is typically interested in relations between \emph{simple} graphs.
	Our machinery does not immediately apply to simple graphs.
	This is because the class of simple graphs $\G$ does not contain the multiplicative unit $\boldsymbol{1}$, i.e., the single-vertex looped graph.
	The complete graphs $K_n$ serve as a substitute.

	In \cref{thm:axiomatisation-simple}, we deduce the simple case from \cref{thm:axiomatisation} at the cost of one additional axiom.
	To that end, we first describe how to extend a \emph{simple graph isomorphism relaxation}, i.e., an equivalence relation $\approx$ on $\G$ such that $G \cong H$ implies $G \approx H$, to a graph isomorphism relaxation between looped-simple graphs.

	\begin{definition}\label{def:circ-extension}
		Let $\approx$ be a simple graph isomorphism relaxation.
		Its \emph{saturation} is the graph isomorphism relaxation $\approx^\circ$ defined,
		for $G, H \in \lG$, by
		\[
			G \approx^\circ H
			\quad \defiff \quad
			\forall L \in \G.\ G \otimes L \approx H \otimes L.
		\]
	\end{definition}

	This is well-defined since $G \otimes L$ and $H \otimes L$ are simple for every simple $L \in \G$ by \cref{fact:drop-loops}.
	We deduce \cref{thm:axiomatisation-simple} from \cref{thm:axiomatisation} after adding the additional \cref{it:s6}.
	Recall that $\simeq_\ell$ is the graph isomorphism relaxation characterised by \cref{lem:ideal-cl}.

	\thmAxiomSimp*

	\begin{proof}
		Necessity follows from \cref{thm:axiomatisation,lem:loops-same-relation}.
		For sufficiency, let $\approx$ satisfy \crefrange{it:s1}{it:s6}.
		We show that its saturation $\approx^\circ$ satisfies the conditions of \cref{thm:axiomatisation}.

		\begin{claim}\label{cl:circ-restricts}
			For all simple graphs $G, H \in \G$, it holds that $G \approx^\circ H$ if, and only if, $G \approx H$.
		\end{claim}
		\begin{claimproof}
			If $G \approx H$, then $G \otimes K \approx H \otimes K$ for every $K \in \G$ by \cref{it:s1},
			so $G \approx^\circ H$.
			Conversely, if $G \approx^\circ H$, then in particular $G \otimes K_n \approx H \otimes K_n$ for all $n \geq 1$,
			so $G \approx H$ by \cref{it:s6}.
		\end{claimproof}

		\begin{claim}\label{cl:circ1}
			The relation $\approx^\circ$ satisfies \cref{it:a1} of \cref{thm:axiomatisation}.
		\end{claim}
		\begin{claimproof}
			Suppose that $G \approx^\circ H$ for $G, H \in \lG$.
			Let $K \in \lG$.
			Let $L \in \G$ be arbitrary.
			Then $G \otimes L \approx H \otimes L$ by \cref{def:circ-extension}.
			Let $n \geq 1$ be arbitrary.
			The graph $K_n \otimes K$ is simple by \cref{fact:drop-loops}.
			Hence, by \cref{it:s1} for $\approx$,
			it holds that
			$G \otimes L \otimes (K \otimes K_n) \approx H \otimes L \otimes (K \otimes K_n)$.
			Since $L$ is simple, $G \otimes L \otimes K$
			and $H \otimes L \otimes K$ are simple by \cref{fact:drop-loops} and \cref{it:s6} applies.
			It follows that $G \otimes L \otimes K \approx H \otimes L \otimes K$.
			Hence, by \cref{def:circ-extension}, as $L$ was arbitrary,
			$G \otimes K \approx^\circ H \otimes K$.
		\end{claimproof}

		\begin{claim}\label{cl:circ2}
			The relation $\approx^\circ$ satisfies \cref{it:a2} of \cref{thm:axiomatisation}.
		\end{claim}
		\begin{claimproof}
			Suppose that $G \approx^\circ H$ for $G, H \in \lG$.
			Let $K \in \lG$.
			Let $L \in \G$ be arbitrary.
			Then $G \otimes L \approx H \otimes L$ by \cref{def:circ-extension}.
			Since $K \otimes L$ is simple by \cref{fact:drop-loops}, \cref{it:s2} applies and 
			$(G + K) \otimes L = G \otimes L + K \otimes L \approx H \otimes L + K \otimes L = (H + K) \otimes L$.
			As $L$ was arbitrary, by \cref{def:circ-extension},
			$G + K \approx^\circ H + K$.
		\end{claimproof}

		\begin{claim}\label{cl:circ3}
			The relation $\approx^\circ$ satisfies \cref{it:a3} of \cref{thm:axiomatisation}.
		\end{claim}
		\begin{claimproof}
			Suppose that $G + K \approx^\circ H + K$ for $G, H, K \in \lG$.
			Let $L \in \G$ be arbitrary.
			By \cref{def:circ-extension},
			$G \otimes L + K \otimes L = (G + K) \otimes L \approx (H +K) \otimes L = H \otimes L + K \otimes L$.
			Since $K \otimes L$ is simple by \cref{fact:drop-loops}, it cancels from both sides by \cref{it:s3}, as desired.
		\end{claimproof}

		\begin{claim}\label{cl:circ4}
			The relation $\approx^\circ$ satisfies \cref{it:a4} of \cref{thm:axiomatisation}.
		\end{claim}
		\begin{claimproof}
			Suppose that $n G \approx^\circ n H $ for $G, H\in \lG$ and $n\geq 1$.
			Let $L \in \G$ be arbitrary.
			By \cref{def:circ-extension},
			$n (G \otimes L) = (nG)\otimes L \approx (nH)\otimes L  = n(H \otimes L)$.
			Since $G \otimes L$ is simple by \cref{fact:drop-loops}, \cref{it:s4} applies and $n$ cancels from both sides, as desired.
		\end{claimproof}

		\begin{claim}\label{cl:circ5}
			The relation $\approx^\circ$ satisfies \cref{it:a5} of \cref{thm:axiomatisation}.
		\end{claim}
		\begin{claimproof}
			Let $G, H \in \lG$.
			Suppose that, for all $\ell \geq 1$, there exist $k \geq 1$ and graphs $G', H' \in \lG$ such that $G' \simeq_\ell H'$ and $k(G-H) \approx^\circ G'-H'$.
			Let $L \in \G$ be arbitrary.
			Let $\ell \geq 1$ and choose $k$, $G'$, and $H'$ as above.
By \cref{def:circ-extension} applied with $L^{\otimes m}$ and distributivity,
			\[
			k(L \otimes G - L \otimes H) = L \otimes k(G-H) \approx L \otimes ( G'-H') = L \otimes G' - L \otimes H'.
			\]
			By \cref{lem:ideal-cl}, $\simeq_\ell$ is a homomorphism indistinguishability relation.
			Hence, by \cref{eq:h1}, $G' \simeq_\ell H'$ implies that $L \otimes G' \simeq_\ell L \otimes H'$.
			The graphs $L \otimes G$, $L \otimes H$, $L \otimes G'$, and $L \otimes H'$ are simple by \cref{fact:drop-loops}.
			As $\ell$ was arbitrary, \cref{it:s5} yields $L \otimes G \approx L \otimes H$.
			Since $L$ was arbitrary, $G \approx^\circ H$ by \cref{def:circ-extension}, establishing the conclusion of \cref{it:a5} of \cref{thm:axiomatisation}.
		\end{claimproof}

		By \cref{cl:circ1,cl:circ2,cl:circ3,cl:circ4,cl:circ5} and \cref{thm:axiomatisation},
		$\approx^\circ$ is a homomorphism indistinguishability relation over a componental class $\mathcal{F}$ of looped-simple graphs.
		Moreover, $\approx^\circ$ satisfies the condition of \cref{lem:loops-same-relation}:
		if $G \otimes K_n \approx^\circ H \otimes K_n$ for all $n \geq 1$,
		then, for every $L \in \G$, applying \cref{it:s6} to the graphs $G \otimes L$ and $H \otimes L$, which are simple by \cref{fact:drop-loops},
		yields $G \otimes L \approx H \otimes L$, i.e.\ $G \approx^\circ H$.
		By \cref{lem:loops-same-relation},
		$\approx^\circ$ is a homomorphism indistinguishability relation over the class of simple graphs in $\mathcal{F}$, cf.\ its proof.
		It remains to observe that via \cref{cl:circ-restricts} the relations $\approx$ and $\approx^\circ$ coincide on simple graphs.
	\end{proof}

    \section{Material omitted in \texorpdfstring{\cref{sec:topology}}{Topology}}

    \subsection{Homomorphism distinguishing topology as Zariski topology}

    Before turning to the elementary perspective of \cref{sec:topology},
    we relate the homomorphism distinguishing topology to the Zariski topology on the semiring algebras of \cref{sec:which-relations}.

    For a subset $P$ of $\mathbb{Z}\lG$ or of $\mathbb{C}\lG$, let
    \[
        \mathcal{V}(P) \coloneqq \{F \in \lG \mid \forall G, H \in \lG.\ G - H \in P \implies \hom(F,G) = \hom(F,H)\}.
    \]
    
    \begin{lemma}\label{lem:closed-to-vanishing-looped}
        For a componental graph class $\mathcal{F} \subseteq \lG$, the following are equivalent.
        \begin{assertions}
            \item\label{it:ctv1} $\mathcal{F}$ is homomorphism distinguishing closed, i.e., $\cl^\circ(\mathcal{F}) = \mathcal{F}$,
            \item\label{it:ctv2} there exists an ideal $\mathfrak{a} \subseteq \mathbb{Z}\lG$ such that $\mathcal{F} = \mathcal{V}(\mathfrak{a})$,
            \item\label{it:ctv3} there exists an ideal $\mathfrak{a} \subseteq \mathbb{C}\lG$ such that $\mathcal{F} = \mathcal{V}(\mathfrak{a})$.
			\item\label{it:ctv4} there exists a set $P \subseteq \mathbb{Z}\lG$ such that $\mathcal{F} = \mathcal{V}(P)$,
            \item\label{it:ctv5} there exists a set $P \subseteq \mathbb{C}\lG$ such that $\mathcal{F} = \mathcal{V}(P)$.
        
		\end{assertions}
		Furthermore, for a componental class $\mathcal{F} \subseteq \G$ of simple graphs, the following are equivalent:
		\begin{assertions}[resume]
			\item\label{it:ctv6} $\mathcal{F}$ is homomorphism distinguishing closed, i.e., $\cl(\mathcal{F}) = \mathcal{F}$,
			\item\label{it:ctv7} there exists a set $P \subseteq \mathbb{Z}\lG$ such that $\mathcal{F} = \mathcal{V}(P) \cap \G$,
			\item\label{it:ctv8} there exists a set $P \subseteq \mathbb{C}\lG$ such that $\mathcal{F} = \mathcal{V}(P) \cap \G$.
		\end{assertions}
    \end{lemma}
    \begin{proof}
        \ref{it:ctv2} $\Rightarrow$ \ref{it:ctv4} and \ref{it:ctv3} $\Rightarrow$ \ref{it:ctv5}
        hold as every ideal is in particular a subset.

        \ref{it:ctv4} $\Rightarrow$ \ref{it:ctv1} and \ref{it:ctv5} $\Rightarrow$ \ref{it:ctv1}.
        Let $\mathcal{F} = \mathcal{V}(P)$.
        By definition, $G \equiv_{\mathcal{F}} H$ for all graphs $G, H \in \lG$ with $G - H \in P$.
        Hence, every $F \in \cl^\circ(\mathcal{F})$ satisfies $\hom(F,G) = \hom(F,H)$ for all such $G, H$,
        that is, $\cl^\circ(\mathcal{F}) \subseteq \mathcal{V}(P) = \mathcal{F}$.
        The reverse inclusion holds by the definition of $\cl^\circ$.

        \ref{it:ctv1} $\Rightarrow$ \ref{it:ctv2} and \ref{it:ctv1} $\Rightarrow$ \ref{it:ctv3}.
        Since $\mathcal{F}$ is componental, ${\equiv_{\mathcal{F}}} = {\equiv_{\mathcal{F} \cap \lZ}}$ by \eqref{eq:coprod}.
        Hence, \eqref{eq:connected} yields both directions of \eqref{eq:c1}
        and \eqref{eq:prod} yields \eqref{eq:c2},
        so that $\equiv_{\mathcal{F}}$ is a torsion-free cancellative congruence on $\lG$.
        By \cref{prop:ideals-to-relations}, the set $\mathfrak{a} \coloneqq \{G - H \mid G \equiv_{\mathcal{F}} H\}$
        is an ideal of $\mathbb{Z}\lG$ with
        \begin{equation}\label{eq:V-differences}
            G - H \in \mathfrak{a} \iff G \equiv_{\mathcal{F}} H
            \qquad \text{for all } G, H \in \lG.
        \end{equation}
        Therefore, $\mathcal{V}(\mathfrak{a}) = \cl^\circ(\mathcal{F}) = \mathcal{F}$ by the definition of $\cl^\circ$.

        For \ref{it:ctv3}, note that $\equiv_{\mathcal{F}}$ is a graph isomorphism relaxation
        satisfying \cref{eq:h1,eq:h2,eq:h3,eq:h4}, the last by \eqref{eq:connected}.
        By \cref{lem:c-ideals}, there is an ideal $\mathfrak{a}^{\mathbb{C}}$ of $\mathbb{C}\lG$
        such that $G - H \in \mathfrak{a}^{\mathbb{C}}$ if, and only if, $G \equiv_{\mathcal{F}} H$,
        so that $\mathcal{V}(\mathfrak{a}^{\mathbb{C}}) = \cl^\circ(\mathcal{F}) = \mathcal{F}$ as before.

        Finally, let $\mathcal{F} \subseteq \G$ be componental.
        Note that the ideals constructed above satisfy $\mathcal{V}(\mathfrak{a}) = \cl^\circ(\mathcal{F})$
        and $\mathcal{V}(\mathfrak{a}^{\mathbb{C}}) = \cl^\circ(\mathcal{F})$
        without any assumption on $\mathcal{F}$,
        since this only used the definitions of $\mathfrak{a}$, $\mathfrak{a}^{\mathbb{C}}$, and $\cl^\circ$.
        Hence, by \eqref{eq:drop-loops-closure},
        \[
            \cl(\mathcal{F}) = \cl^\circ(\mathcal{F}) \cap \G = \mathcal{V}(\mathfrak{a}) \cap \G = \mathcal{V}(\mathfrak{a}^{\mathbb{C}}) \cap \G,
        \]
        so \ref{it:ctv6} implies \ref{it:ctv7} and \ref{it:ctv8}.
        Conversely, let $\mathcal{F} = \mathcal{V}(P) \cap \G$ for a set $P \subseteq \mathbb{Z}\lG$ or $P \subseteq \mathbb{C}\lG$.
        As above, $G \equiv_{\mathcal{F}} H$ for all $G, H \in \lG$ with $G - H \in P$,
        so every $F \in \cl^\circ(\mathcal{F})$ satisfies $\hom(F,G) = \hom(F,H)$ for all such $G,H$.
        By \eqref{eq:drop-loops-closure},
        $\cl(\mathcal{F}) = \cl^\circ(\mathcal{F}) \cap \G \subseteq \mathcal{V}(P) \cap \G = \mathcal{F}$,
        and the reverse inclusion holds by the definition of $\cl$.
    \end{proof}

	\Cref{lem:closed-to-vanishing-looped} shows that the homomorphism distinguishing topology (on $\lZ$ or $\Z$) is the subspace topology of the Zariski topology of $\C\lG$ or $\mathbb{Z}\lG$ on the subspace $\{\hom(F, \star) \mid F \in \lZ\}$.
	In particular, it is indeed a topology.

    \subsection{The homomorphism distinguishing topology is a topology}

    In this section, we give an elementary proof of the fact that the homomorphism distinguishing topology on $\Z$ forms a topology. To that end,
    \textcite{roberson_oddomorphisms_2022}
    observed that $\cl(\star)$ is a closure operator, i.e., for all graph classes $\mathcal{F},\mathcal{F}_1, \mathcal{F}_2 \subseteq \G$ with $\mathcal{F}_1 \subseteq \mathcal{F}_2$,
    \begin{align}
        \mathcal{F} &\subseteq \cl(\mathcal{F}), \label{eq:cl-extensive}\\
        \cl(\mathcal{F}_1) &\subseteq \cl(\mathcal{F}_2), \label{eq:cl-increasing} \\
        \cl(\cl(\mathcal{F})) &= \cl(\mathcal{F}). \label{eq:cl-idempotent}
    \end{align}

    We require the following lemma.

	\begin{lemma}[{\cite[Theorem~7]{seppelt_logical_2024}}]\label{lem:cl-componental}
		If $\mathcal{F} \subseteq \G$ is componental, then $\cl(\mathcal{F})$ is componental.
	\end{lemma}
    
    \lemComponentalClosed*
    \begin{proof}
        Let $\mathcal{F} \coloneqq \mathcal{F}_1 + \mathcal{F}_2$.
		By \cref{lem:cl-componental}, cf.\ \cite[Corollary~7.1.5]{seppelt_homomorphism_2024},
		it suffices to construct, for every connected graph $F \not\in \mathcal{F}$,
		two graphs $G$ and $H$ such that $G \equiv_{\mathcal{F}} H$
		and $\hom(F, G) \neq \hom(F, H)$.
        
		Let  $F \not\in \mathcal{F}$ be connected.
        This implies that $F\not\in \mathcal{F}_1$ and $F\not\in \mathcal{F}_2$.
		By assumption,
		there exist graphs $G_1, G_2, H_1, H_2 \in \G$,
		such that $G_1 \equiv_{\mathcal{F}_1} H_1$,
		$G_2 \equiv_{\mathcal{F}_2} H_2$,
		$\hom(F, G_1) \neq \hom(F, H_1)$,
		and $\hom(F, G_2) \neq \hom(F, H_2)$.
		
		We combine these two pairs of graphs into one pair of graphs with the desired properties.
		Let $G \coloneqq G_1 \otimes G_2 + H_1 \otimes H_2$
		and $H \coloneqq G_1 \otimes H_2 + G_2 \otimes H_1$.
        By \cref{eq:connected,eq:prod},
		for every connected graph $K$,
		it holds that $\hom(K, G) = \hom(K, H)$
		if, and only if,
		\[
			\left( \hom(K, G_1) - \hom(K, H_1) \right)
			\left( \hom(K, G_2) - \hom(K, H_2) \right)
			= 0.
		\]
		It follows that $\hom(F, G) \neq \hom(F, H)$
        and $G \equiv_{\mathcal{F}_1 + \mathcal{F}_2} H$.
	\end{proof}

    The following corollary allows us to compute the homomorphism distinguishing closure of a sum of two componental graph classes as the sum of their closures.
	
	\begin{corollary}\label{cor:componental-closure}
		For componental graph classes $\mathcal{F}_1, \mathcal{F}_2 \subseteq \G$, it holds that
		\[ \cl(\mathcal{F}_1 + \mathcal{F}_2) = \cl(\mathcal{F}_1) + \cl(\mathcal{F}_2).\]
	\end{corollary}
	\begin{proof}
		By \eqref{eq:cl-extensive}, $\mathcal{F}_1 + \mathcal{F}_2 \subseteq \cl(\mathcal{F}_1) + \cl(\mathcal{F}_2)$.
		By \cref{lem:cl-componental}, $\cl(\mathcal{F}_1)$ and $\cl(\mathcal{F}_2)$ are componental.
		By \cref{lem:componental-closed}, $\cl(\mathcal{F}_1) + \cl(\mathcal{F}_2)$ is homomorphism distinguishing closed
		and thus 
        \[
            \cl(\mathcal{F}_1 + \mathcal{F}_2)  
            \overset{\eqref{eq:cl-increasing}}{\subseteq }\cl(\cl(\mathcal{F}_1) + \cl(\mathcal{F}_2))
            \overset{\eqref{eq:cl-idempotent}}{=} \cl(\mathcal{F}_1) + \cl(\mathcal{F}_2).
        \]
		Conversely, $\cl(\mathcal{F}_1), \cl(\mathcal{F}_2) \subseteq \cl(\mathcal{F}_1 + \mathcal{F}_2)$ by \eqref{eq:cl-increasing}.
		By \cref{lem:cl-componental},
		$\cl(\mathcal{F}_1) + \cl(\mathcal{F}_2) \subseteq \cl(\mathcal{F}_1 + \mathcal{F}_2)$.
	\end{proof}

    \subsection{Topological properties}

    \propFrechet*

    The proof of \cref{prop:frechet} requires the following theorem.

    \begin{theorem}[\textcite{lovasz_operations_1967,chaudhuri_optimization_1993}]\label{thm:dual-lovasz}
        Two graphs $F, K \in \G$ are isomorphic if, and only if,  for all graphs $G \in \G$, it holds that $\hom(F, G) = \hom(K,G)$.       
    \end{theorem}

	\begin{proof}[Proof of \cref{prop:frechet}]
		Let $F$ be a connected graph.
		We show that the componental graph class generated by $\{F\}$ is homomorphism distinguishing closed.
		That is, for every graph $K$ not in this class, there exist graphs $X$ and $Y$ such that $\hom(K, X) \neq \hom(K, Y)$ and $\hom(F, X) = \hom(F, Y)$.
		
		By \cref{lem:cl-componental}, it suffices to consider a connected graph $K$ such that $F \not\cong K$.
		By \cref{thm:dual-lovasz},
		there exists a graph $G \in \G$ such that $\hom(F, G) \neq \hom(K, G)$.
		Consider the graphs $X \coloneqq G \otimes K$ and $Y \coloneqq \hom(F, G) K$.
		Then $\hom(F, X) = \hom(F, G) \hom(F, K) = \hom(F, Y)$ by \cref{eq:connected,eq:prod}.
		Moreover, $\hom(K, X) = \hom(K, G) \hom(K, K) \neq \hom(F, G) \hom(K, K) = \hom(K, Y)$
		since $\hom(K, K) \neq 0$.
		The final assertion now follows from \cref{lem:componental-closed}. 
	\end{proof}
    \propDisconnected*
	\begin{proof}
		In order to show that our topology is disconnected, we must exhibit a non-trivial clopen set. 
		In other words, by \cite[Lemma 1.1]{milne_basic_2012}, we must find a non-trivial idempotent in $\mathbb{C}\lG$.
		We use the identity $K_2 \otimes K_2 = K_2 + K_2$, which yields the non-trivial idempotent $\frac12 K_2$. In particular, the sets $\{F \in \Z \mid \hom(F, K_2) = 0\} = \mathcal{V}(\{K_2\}) \cap \Z$ and $\{F \in \Z  \mid \hom(F, K_2) = 2\} = \mathcal{V}(\{K_2 - 2\}) \cap \Z$ are closed by \cref{lem:closed-to-vanishing-looped} and disjoint. They are the componental graph classes of bipartite and non-bipartite graphs respectively.

		A further example is given by the idempotent $K_1 = K_1 \otimes K_1$, which shows that the set  $\{F \in \Z \mid \hom(F, K_1) = 0\}$ of all graphs with at least one edge is closed; hence $\{K_1\}$ is open, and it is closed by \cref{prop:frechet}, thus clopen.

		For non-compactness, we recast the proof of \cite[Theorem~6.15]{roberson_oddomorphisms_2022}, which shows that, for every $S \subseteq \mathbb{N}$, the componental graph class $\mathcal{K}_S$ generated by the complete graphs $K_s$ with $s \in S$ is homomorphism distinguishing closed.
		Write $\mathcal{K}_\ell \coloneqq \mathcal{K}_{\{\ell, \ell+1, \ell+2, \dots \}}$.
		The family of (closed) sets $\mathcal{K}_\ell$ for $\ell \in \mathbb{N}$ has the finite intersection property, i.e., any finite subfamily has non-empty intersection.
		However, the intersection $\bigcap_{\ell \in \mathbb{N}} \mathcal{K}_\ell = \emptyset$ is empty.
	\end{proof}

    \appendix

    \section{Möbius inversion over the Boolean lattice}
    \label{sec:mobius}

	For the reader's convenience, we state the fact about Möbius inversion
	underlying the proof of \cref{lem:unique-subgraph-series}.
	It is the instance of the Möbius inversion formula
	\cite[Proposition~3.7.1]{stanley_enumerative_2012}
	for the Boolean lattice $2^X$,
	whose Möbius function is $\mu(T, S) = (-1)^{|S \setminus T|}$,
	and it is equivalent to the principle of inclusion--exclusion.

	\begin{lemma}[{\cite[Example~3.8.3]{stanley_enumerative_2012}}]\label{lem:mobius-inversion}
		Let $X$ be a finite set and let $f, g \colon 2^X \to \mathbb{C}$.
		The following are equivalent:
		\begin{assertions}
			\item\label{it:mobius-down} $g(S) = \sum_{T \subseteq S} (-1)^{|S \setminus T|} f(T)$ for all $S \subseteq X$,
			\item\label{it:mobius-up} $f(S) = \sum_{T \subseteq S} g(T)$ for all $S \subseteq X$.
		\end{assertions}
	\end{lemma}

	\printbibliography
\end{document}